\documentclass[10pt]{article}
\usepackage{amsmath,amsthm,amsfonts}
\usepackage{graphicx}
\usepackage{longtable}
\usepackage{multirow, bigdelim}
\usepackage{float}
\usepackage{hyperref}
\usepackage{stmaryrd}

\def\qed{\vrule height5pt width3pt depth.5pt}

\theoremstyle{plain}
\newtheorem{thm}{Theorem}[section]

\newtheorem{lem}[thm]{Lemma}

\newtheorem{exa}{Example}[section]

\newtheorem{defn}{Definition}[section]

\newtheorem{rem}{Remark}[section]

\newcommand{\Across}{\raisebox{-0.25\height}{\includegraphics[width=0.5cm]{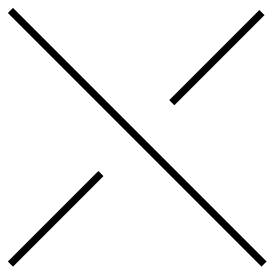}}}
\newcommand{\Asmooth}{\raisebox{-0.25\height}{\includegraphics[width=0.5cm]{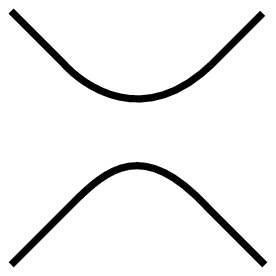}}}
\newcommand{\Bsmooth}{\raisebox{-0.25\height}{\includegraphics[width=0.5cm]{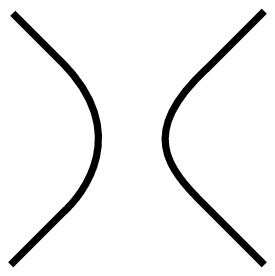}}}
\newcommand{\creation}{\raisebox{-0.25\height}{\includegraphics[height=.8cm]{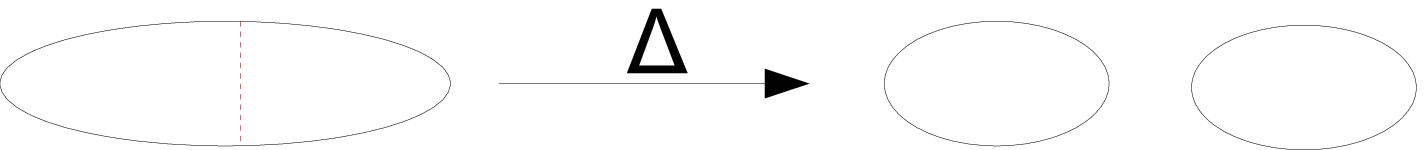}}}
\newcommand{\annihilation}{\raisebox{-0.25\height}{\includegraphics[height=.8cm]{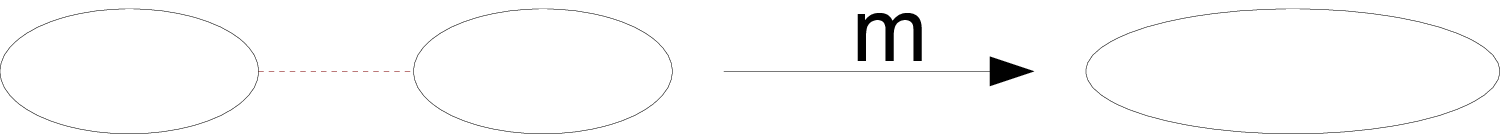}}}
\newcommand{\singlecycle}{\raisebox{-0.25\height}{\includegraphics[height=1.1cm]{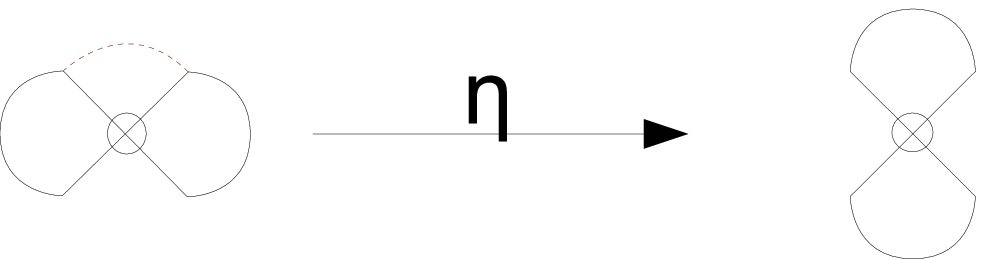}}}
\newcommand{\gradingexample}{\raisebox{-0.25\height}{\includegraphics[width=0.5\textwidth]{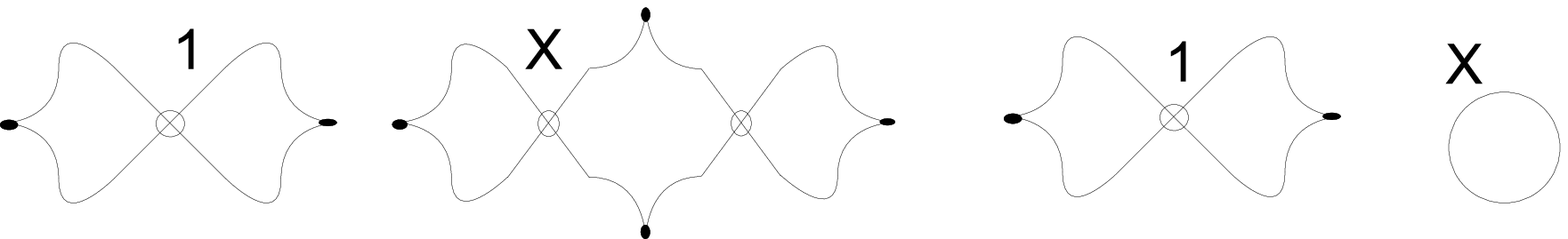}}}
\newcommand{\twoone}{\raisebox{-0.25\height}{\includegraphics[height=.8cm]{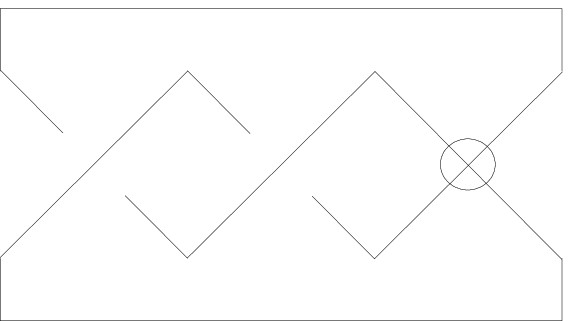}}}
\newcommand{\extendedbracketRIII}{\raisebox{-0.25\height}{\includegraphics[height=.9cm]{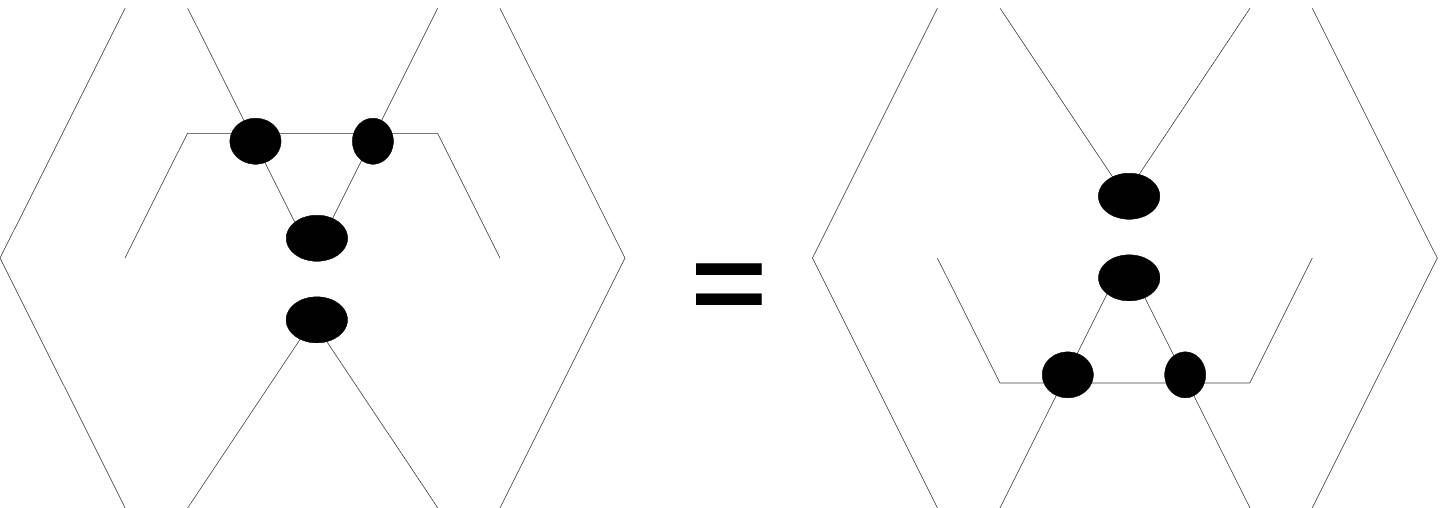}}}

\newcommand{\DTwoOne}{\raisebox{-0.25\height}{\includegraphics[height=1.5cm]{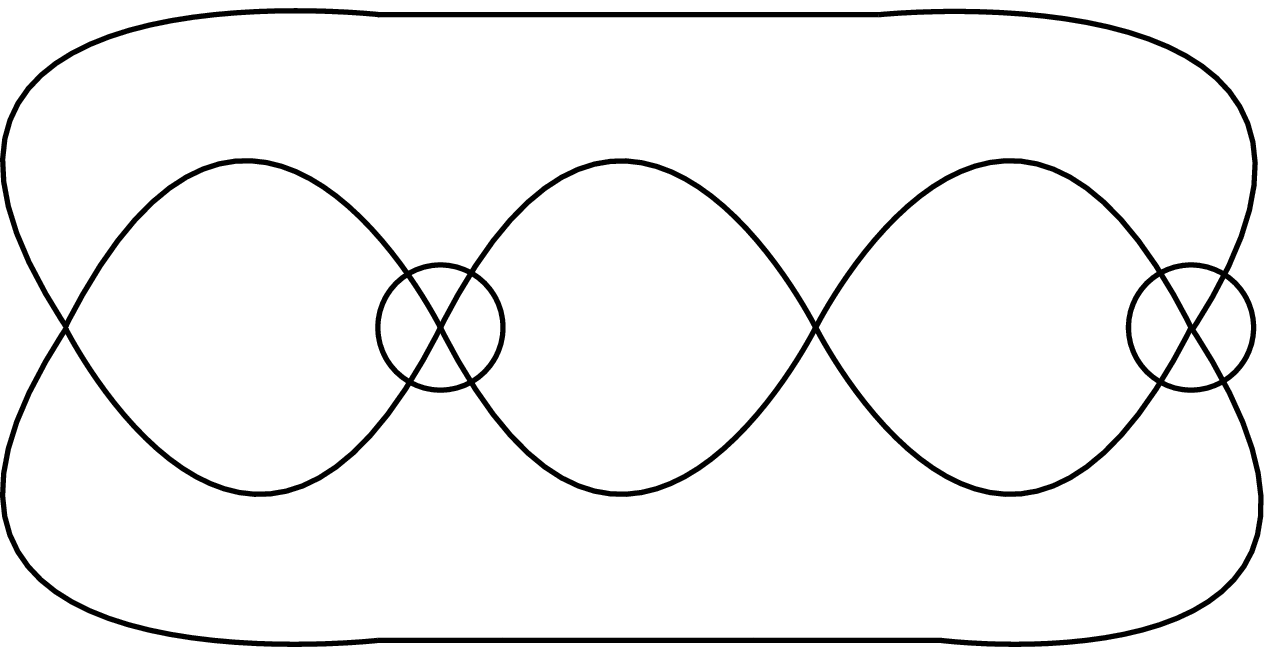}}}
\newcommand{\DTwoTwo}{\raisebox{-0.25\height}{\includegraphics[height=1.5cm]{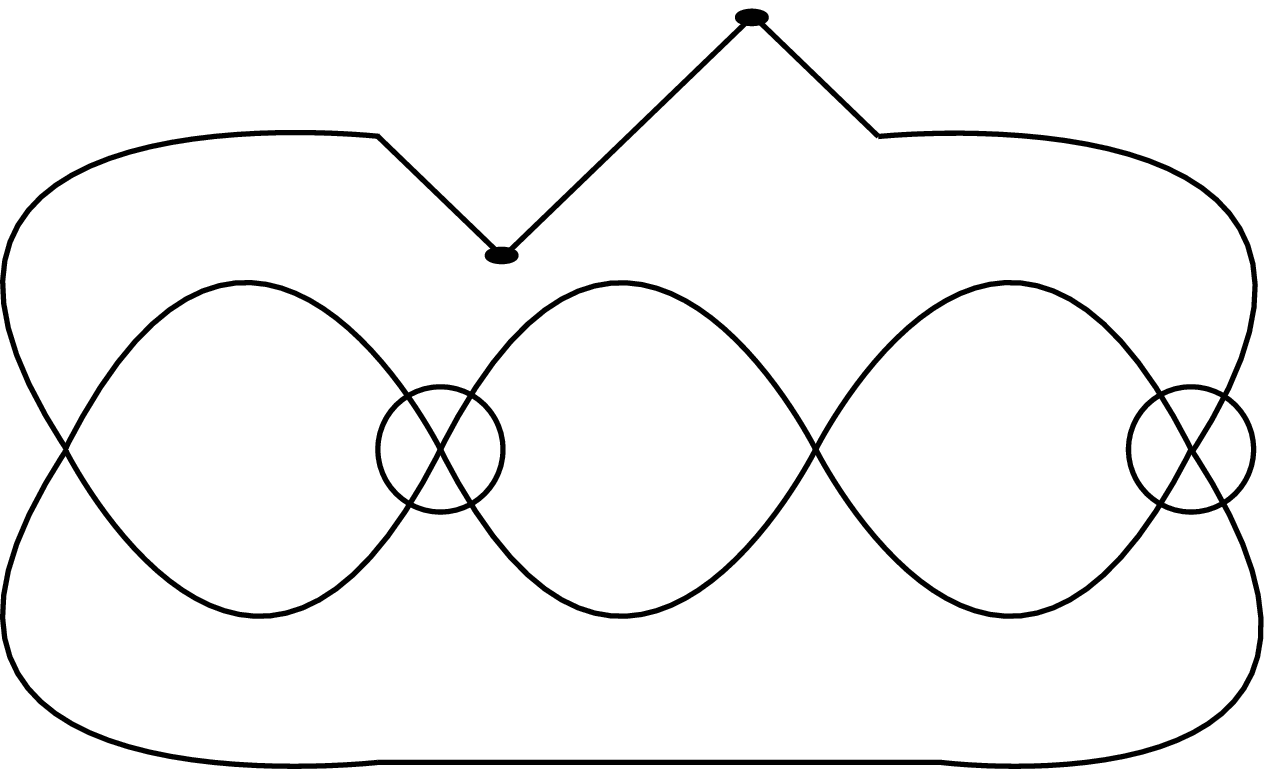}}}
\newcommand{\DTwoThree}{\raisebox{-0.25\height}{\includegraphics[height=1.5cm]{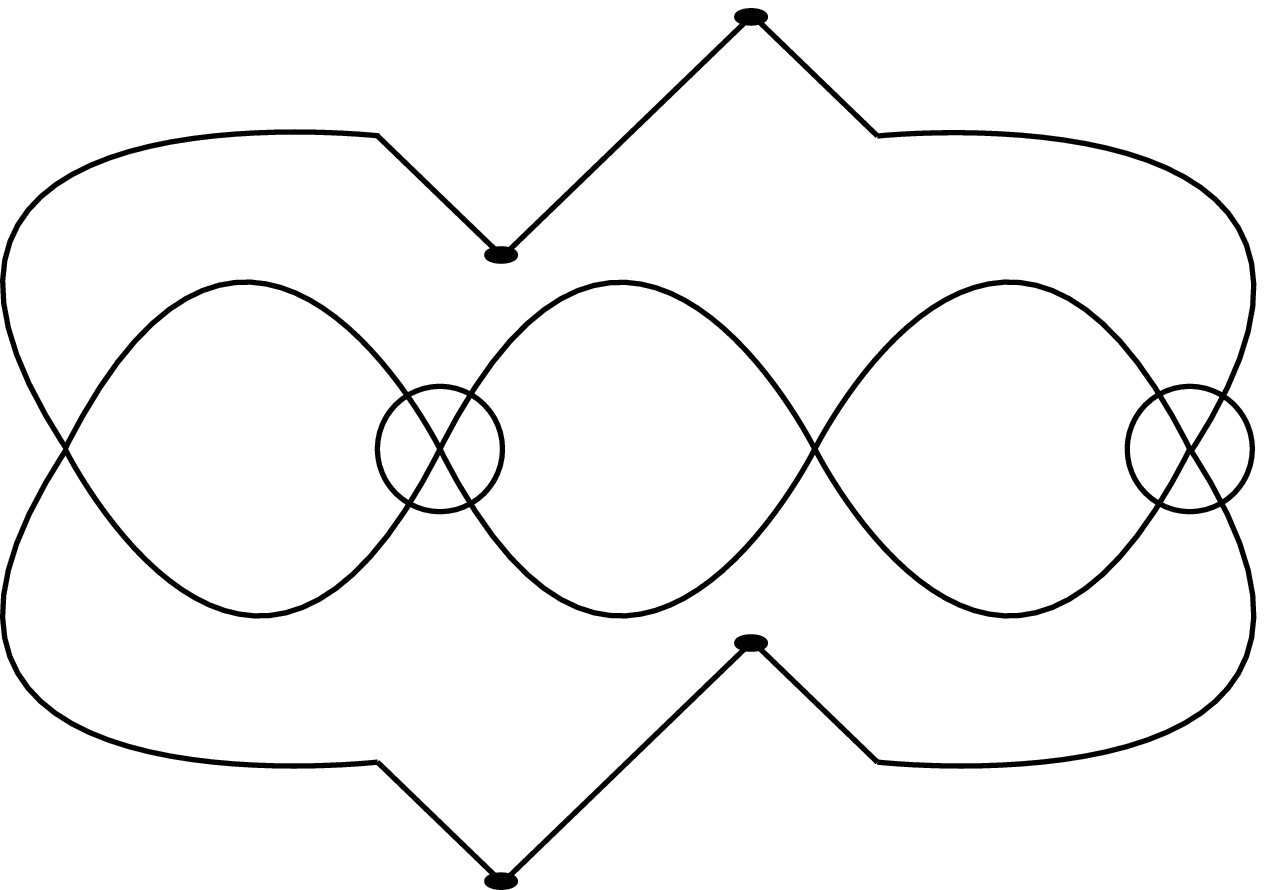}}}

\newcommand{\DFourOne}{\raisebox{-0.25\height}{\includegraphics[height=1.5cm]{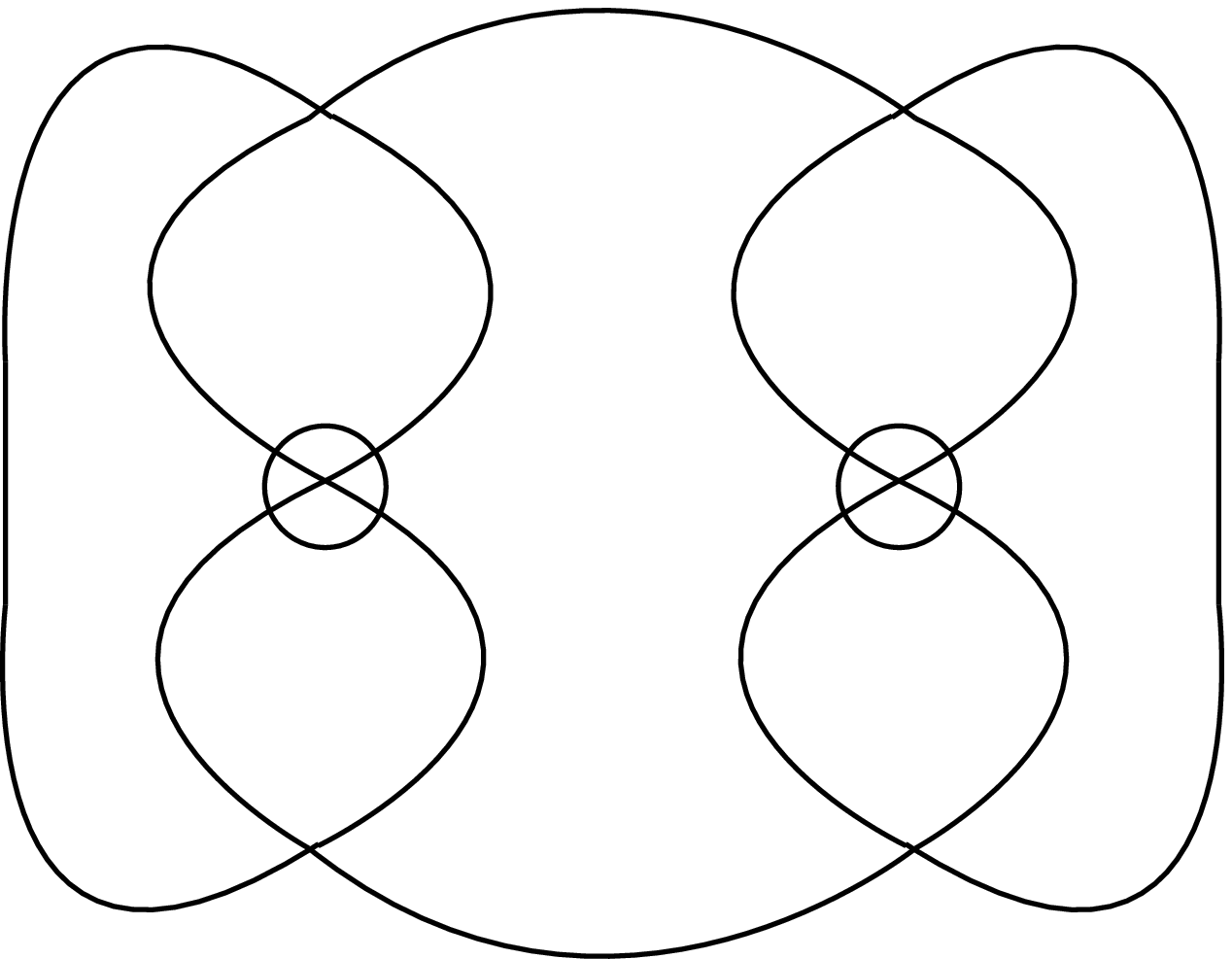}}}
\newcommand{\DFourTwo}{\raisebox{-0.25\height}{\includegraphics[height=1.5cm]{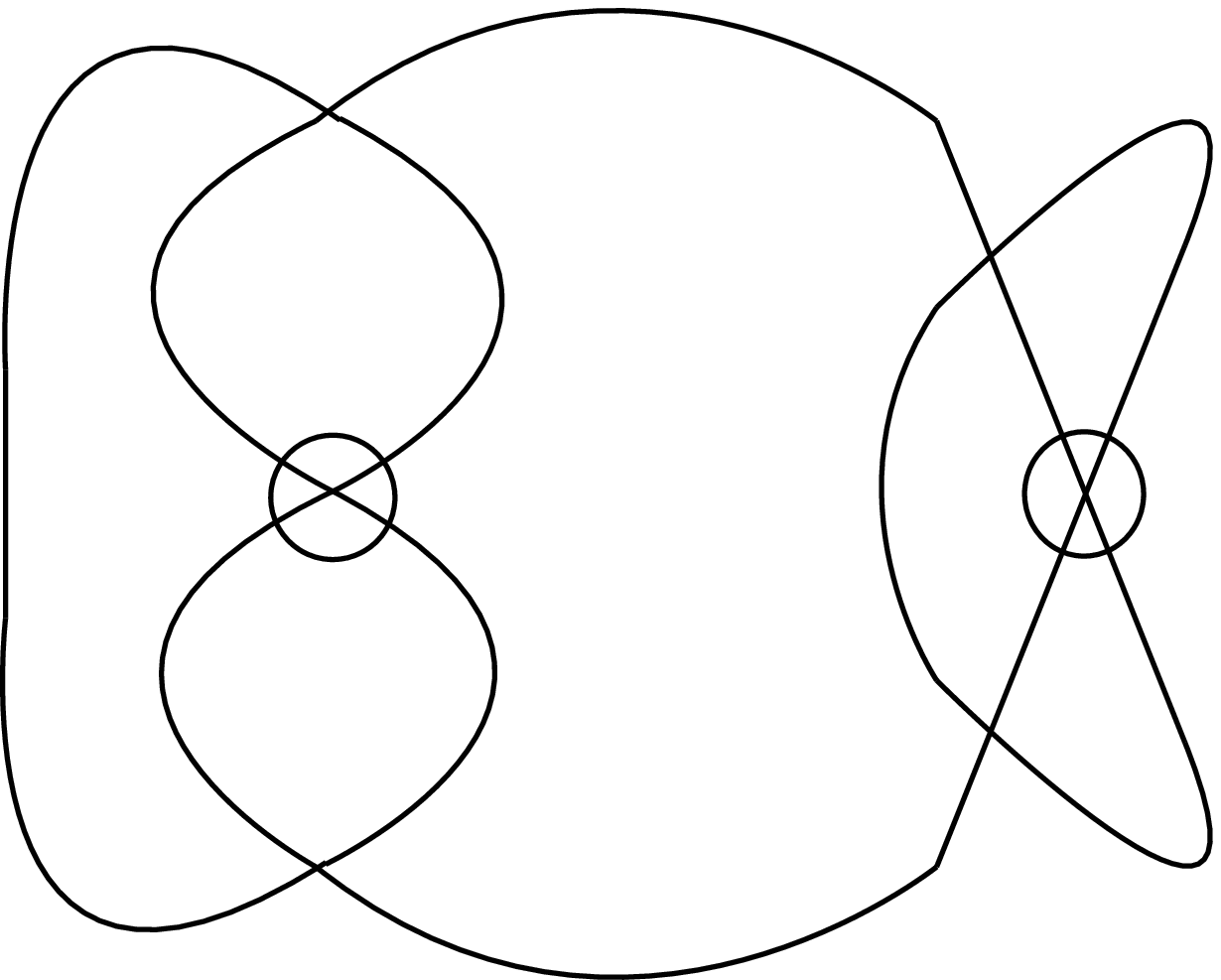}}}
\newcommand{\DFourThree}{\raisebox{-0.25\height}{\includegraphics[height=1.5cm]{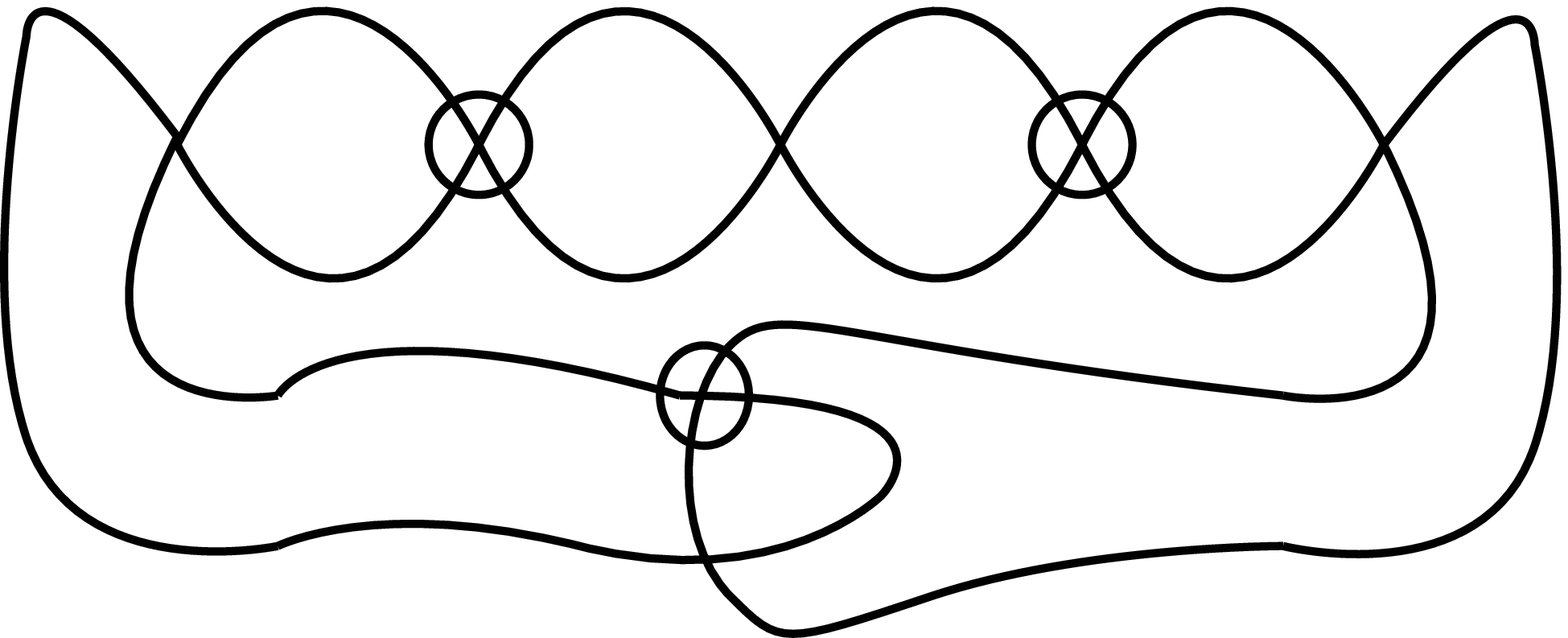}}}
\newcommand{\DFourFour}{\raisebox{-0.25\height}{\includegraphics[height=1.5cm]{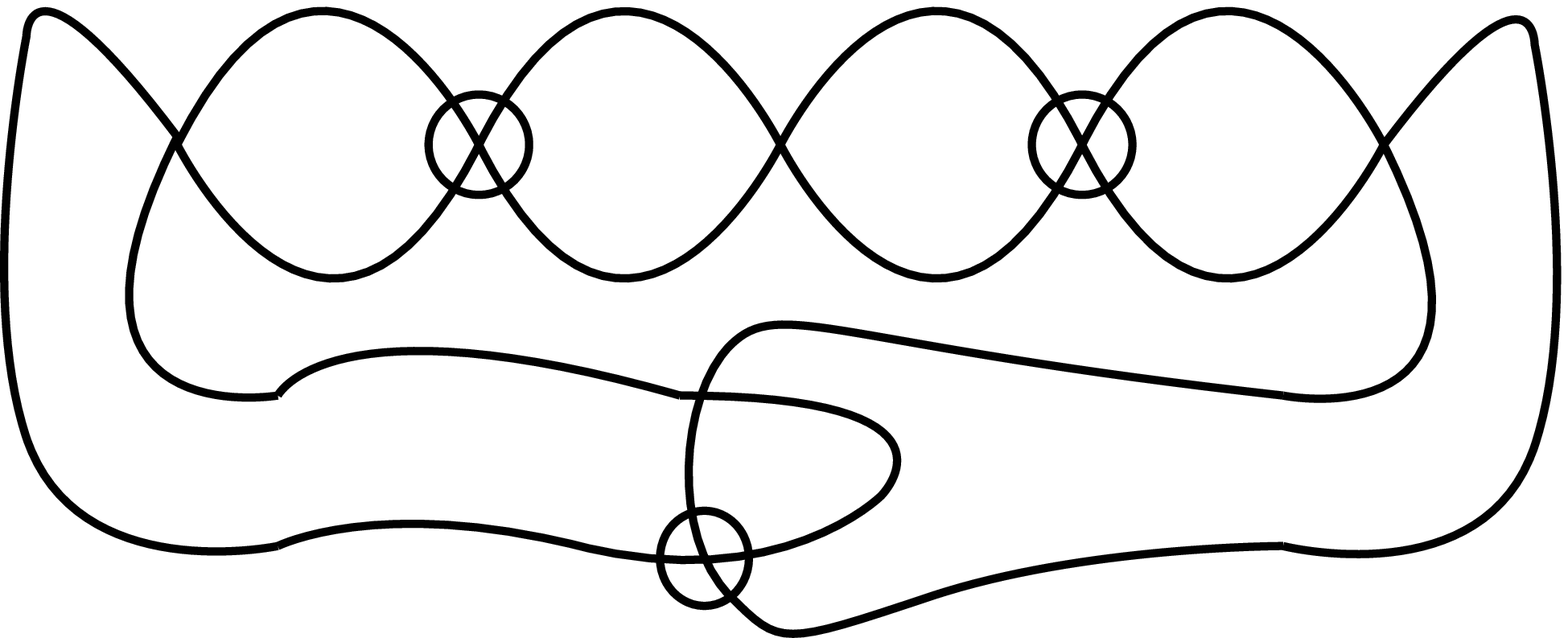}}}

\begin{document}

\title{Parity, Skein Polynomials and Categorification}


\author{Aaron Kaestner \\
North Park University\\
amkaestner@northpark.edu \\
\\
Louis H. Kauffman \\
University of Illinois at Chicago \\
kauffman@uic.edu }

\maketitle

\begin{abstract} We investigate an application of crossing parity for the bracket expansion of the Jones polynomial for virtual knots. In addition we consider an application of parity for the arrow polynomial as well as for the categorifications of both polynomials. We present a number of examples found through our calculations.  We provide tables of calculations for these invariants on virtual knots with at most 4 real crossings.
\end{abstract}

\section{Introduction}

Since the introduction of virtual knot theory, crossing parity has provided a valuable resource for creating invariants.  (See Definition 3.1 in the present paper for our definition of crossing parity.) For instance, given a virtual knot the odd writhe \cite{OddWrithe} (i.e. sum of the signs of the odd crossings) is an easily computable invariant.  Recently, Manturov introduced a parity version of the bracket polynomial \cite{ParityBracket} and described how this construction can pass to Khovanov homology via a filtration on the space of virtual knots.  Our goal here is to investigate these constructions and show how we can apply a similar construction to the arrow polynomial. Along the way we show how known facts about minimal surface genus for virtual knots extend to the categorical setting.\\

The bracket polynomial (and Jones polynomial) as well as the arrow polynomial have a rich history. For an introduction to virtual knot theory we recommend \cite{IntroVKT}.  For an introduction to the bracket polynomial we point the reader to \cite{NewInvariants} and \cite{OnKnots}. Similarly, for the arrow polynomial we recommend the paper \cite{ExtendedBracket} as well as the work of Dye and Kauffman in \cite{DKMinimalSurface} \cite{DKVirtualCrossingNumber}. Note that the arrow polynomial is equivalent to the Miyazawa polynomial as constructed in \cite{Miy06} and \cite{Miy08}. We use the term arrow polynomial as our definition follows the construction introduced by Dye and Kauffman.  \\

The Jones polynomial was categorified by Khovanov in \cite{Khovanov1} and for an introduction to Khovanov homology we point the reader to Bar-Natan \cite{DrorCat},  \cite{DrorCob} and Kauffman \cite{KauffmanKhoHo}.  Since Khovanov's seminal work categorification in classical knot theory has been a fruitful topic of research. Notably, Rasmussen \cite{Rasmussen} used a spectral sequence introduced by Lee \cite{Endomorphism} to produce a lower bound on the slice genus for a classical knot. Recently it was also announced by Kronheimer and Mrowka \cite{UnknotDetector} that Khovanov homology detects the unknot by showing there is a spectral sequence from (reduced) Khovanov homology to instanton Floer homology.  This is exciting news as it provides support for the similar conjecture for the Jones polynomial.  \\

The study of Khovanov homology in relation to virtual knot theory is still rather new.  The introduction of the virtual crossing brings with it a new diagram for the unknot which has the property $d^2 = 0$ (using Khovanov's original definition) only for coefficients in $\mathbb{Z}_2$ (see Figure 6 of \cite{AddnlGradings}). A wonderful explanation of how far we can take Khovanov's original work is given by Viro in \cite{ViroKhoHo}.  To work around this new problem methods have been introduced by Asaeda, Przytycki, and Sikora \cite{APS} as well as Manturov \cite{ArbitraryCoeffs}. The Asaeda-Przytycki-Sikora approach requires not only a diagram but an embedding in a (fixed) thickened surface. On the other hand, Manturov's version introduces a signed, oriented and ordered basis on the states of the cube complex to make all faces commute. \\

Other than its introduction in \cite{DKM} and recent review in \cite{IntroVKT}, little has been written about the categorifications of the arrow polynomial. In this regard we produce lower bounds for the homological width of the fully-graded categorification and we as extend bounds for the surface genus to the categorical setting.  In the Appendix we provide a collection of programs for computing all of the invariants discussed in this paper.  We hope these programs will increase the awareness and interest in the study of these categorifications.\\

\subsection{Virtual Knots}

In \cite{VKT} Kauffman introduced virtual knots and links as a natural extension of classical knots and links.  Virtual knot theory can be thought of both as (1) equivalence classes of embeddings of closed curves in a thickened surface (possibly non-orientable) up to isotopy and handle stabilization on the surface and (2) the completion of the signed oriented Gauss codes (i.e. an arbitrary Gauss code corresponds to a virtual knot while not every Gauss code corresponds to a classical knot.)

We recall in Figure \ref{fig:RMs} the Reidemeister Moves for classical knot diagrams. Figure \ref{fig:VRMs} displays the additional Virtual Reidemeister Moves for the theory in terms of planar diagrams. Here we have introduced the \emph{virtual crossing} which is neither an under-crossing or over-crossing. We represent the virtual crossing by two arcs which cross and have a circle around the crossing point.

\begin{figure}[h!]
\centering
    \includegraphics[width=.7\textwidth]{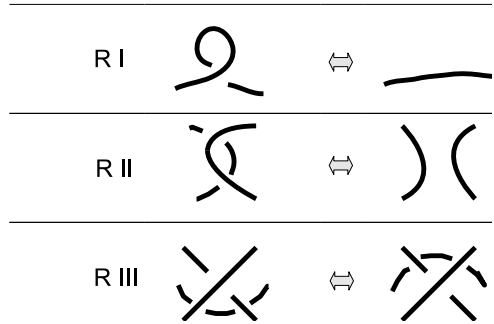}
\caption{Reidemeister Moves}
\label{fig:RMs}
\end{figure}

\begin{figure}[h!]
\centering
    \includegraphics[width=.7\textwidth]{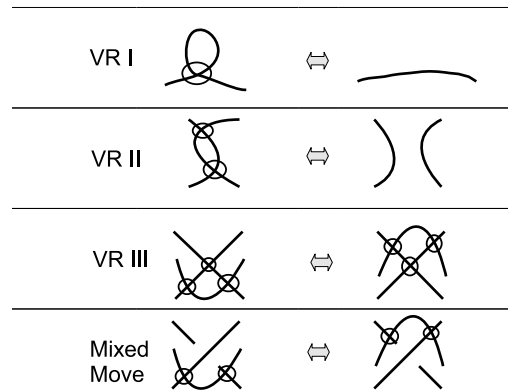}
\caption{Virtual Reidemeister Moves}
\label{fig:VRMs}
\end{figure}

Note that the move in Figure \ref{fig:Forbidden} is not an equivalence relation for diagrams of virtual knots and links. It can be shown that adding this relation and its mirror image to the Virtual Reidemeister Moves allows one to unknot any virtual knot. For this reason we refer to this move as the \emph{forbidden move}. Adding just one of these two forbidden moves yields a nontrivial theory called welded knots \cite{FRR}.

\begin{figure}[h!]
\centering
    \includegraphics[width=.6\textwidth]{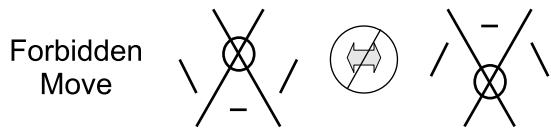}
\caption{The ``Forbidden Move''}
\label{fig:Forbidden}
\end{figure}

\subsection{The Normalized Bracket Polynomial}

For completeness we recall the construction of the bracket polynomial introduced by Kauffman in \cite{NewInvariants}.

\begin{defn}Given a diagram D for a virtual knot K the bracket polynomial of K is defined by the relations in Figures \ref{fig:BracketSmooth}.\\
\end{defn}

\begin{figure}[h!]
\centering
    \includegraphics[width=0.4\textwidth]{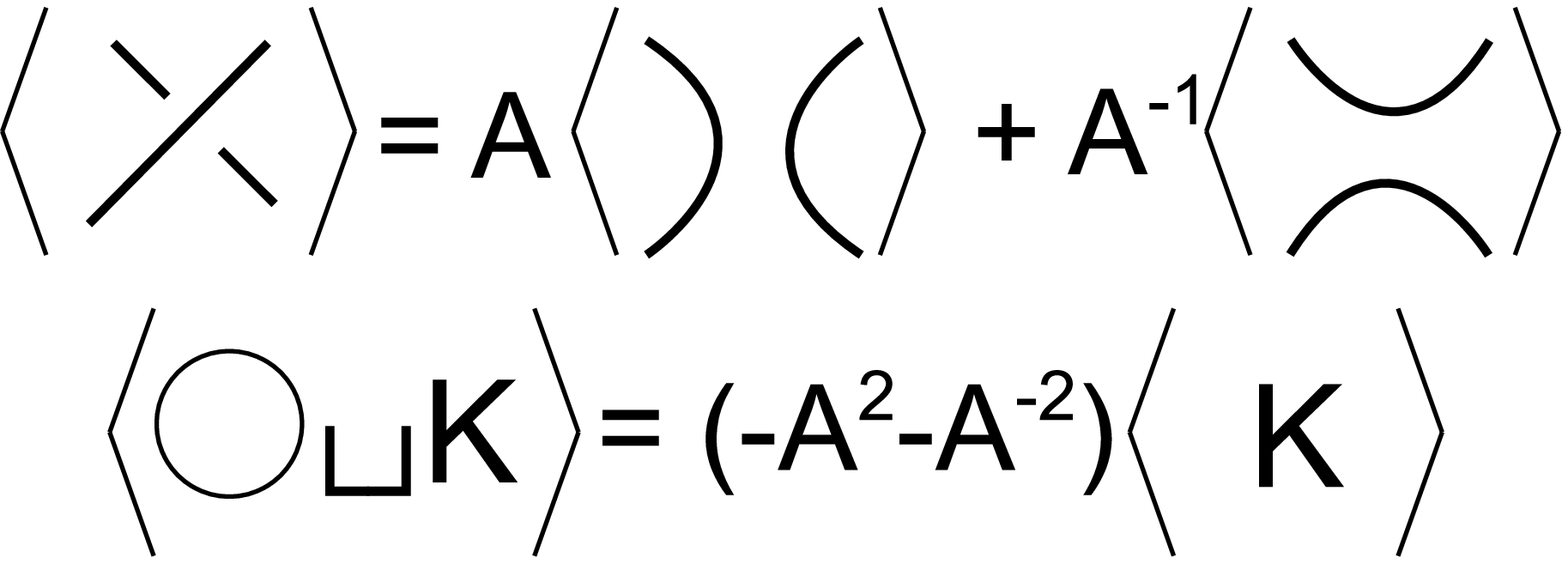}
\caption{Bracket Polynomial Skein Relations}
\label{fig:BracketSmooth}
\end{figure}

To be more precise, suppose D is an n-crossing diagram for the (virtual) knot K. We generate the polynomial by smoothing every crossing in each of the two ways possible. The result of each smoothing at a particular crossing is multiplication by the term, $A$ or $A^{-1}$, following the conventions of Figure \ref{fig:BracketSmooth}, and we take the sum of these two weighted smoothings.  For a particular crossing we call these smoothings the \emph{A-smoothing} and \emph{B-smoothing} respectively.

Next let $s$ = $( s_{1},\ldots, s_{n} )$ where $s_{i} \in \{ 0,1 \} $. Define the \emph{state} of D corresponding to $s$ to be the result of applying an B-smoothing at crossing i when $s_{i} = 0$ and an A-smoothing at crossing i when $s_{i} = 1$.  Furthermore define $\langle D \mid s \rangle$ to be the product of $A$'s and $A^{-1}$'s that label the state $s$ multiplied by the loop value $d^{\| s \|}$ where $d = (-A^2 - A^{-2})$ and $\| s \|$ denotes the number of loops in the state. Here each state is obtained by a choice of smoothing for each crossing and is labeled with the type of smoothing at each of the smoothing sites.

\begin{rem} This definition of the bracket polynomial is shifted from the standard definition. Here the value of a single loop is $d = (-A^2 - A^{-2})$ as opposed to $1$ in the standard definition.  This has the effect that our definition is $d \times \langle \langle K \rangle \rangle$ , where $\langle \langle K \rangle \rangle$ is the standard definition of the bracket polynomial as in \cite{NewInvariants}.
\end{rem}

\begin{defn}
Let $\mathcal{S}$ be the collection of all states of a diagram D for a knot K. We may define the \emph{bracket polynomial of K} to be
\[ \langle K \rangle= \displaystyle\sum\limits_{s \in \mathcal{S}} \langle D \mid s \rangle
\]
\end{defn}

It follows from this definition that the state sum is well defined and the expansion identities in Figure \ref{fig:BracketSmooth} follow from the state sum definition.

\begin{defn}
Given a virtual knot K with diagram D, the \emph{normalized bracket polynomial of K} is given by
\[f_{A}(K) = f_{A}(D) = (-A)^{-3\omega(D)} \langle D \rangle\]
where $\langle D \rangle$ is the bracket polynomial of D and
\[\omega(D) = \textrm{writhe}(D) = (\#\textrm{ positive crossings in D}) - (\#\textrm{ negative crossings in D}).\]
\end{defn}

\begin{thm}
The normalized bracket polynomial is an invariant of (virtual) knots.
\end{thm}

\begin{proof} Here each loop, regardless of virtual crossings, evaluates as the loop value $d$. See \cite{NewInvariants}
\end{proof}

\subsection{The Arrow Polynomial}

Similarly we recall the construction of the arrow polynomial. Given a diagram D for a virtual knot K the (un-normalized) arrow polynomial of K is defined by the smoothing relations in Figure \ref{fig:APSmooth} analogous to the prior construction of the bracket polynomial and reduction relations in Figure \ref{fig:APReduce}.

\begin{figure}[h!]
\centering
    \includegraphics[width=.6\textwidth]{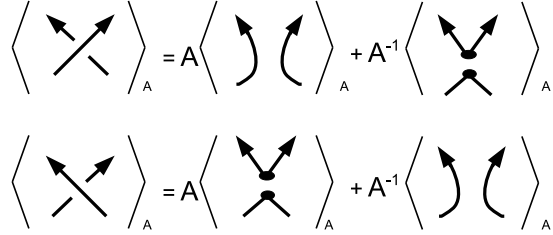}
\caption{Arrow Polynomial Crossing Relations}
\label{fig:APSmooth}
\end{figure}

\begin{figure}[h!]
\centering
    \includegraphics[width=.6\textwidth]{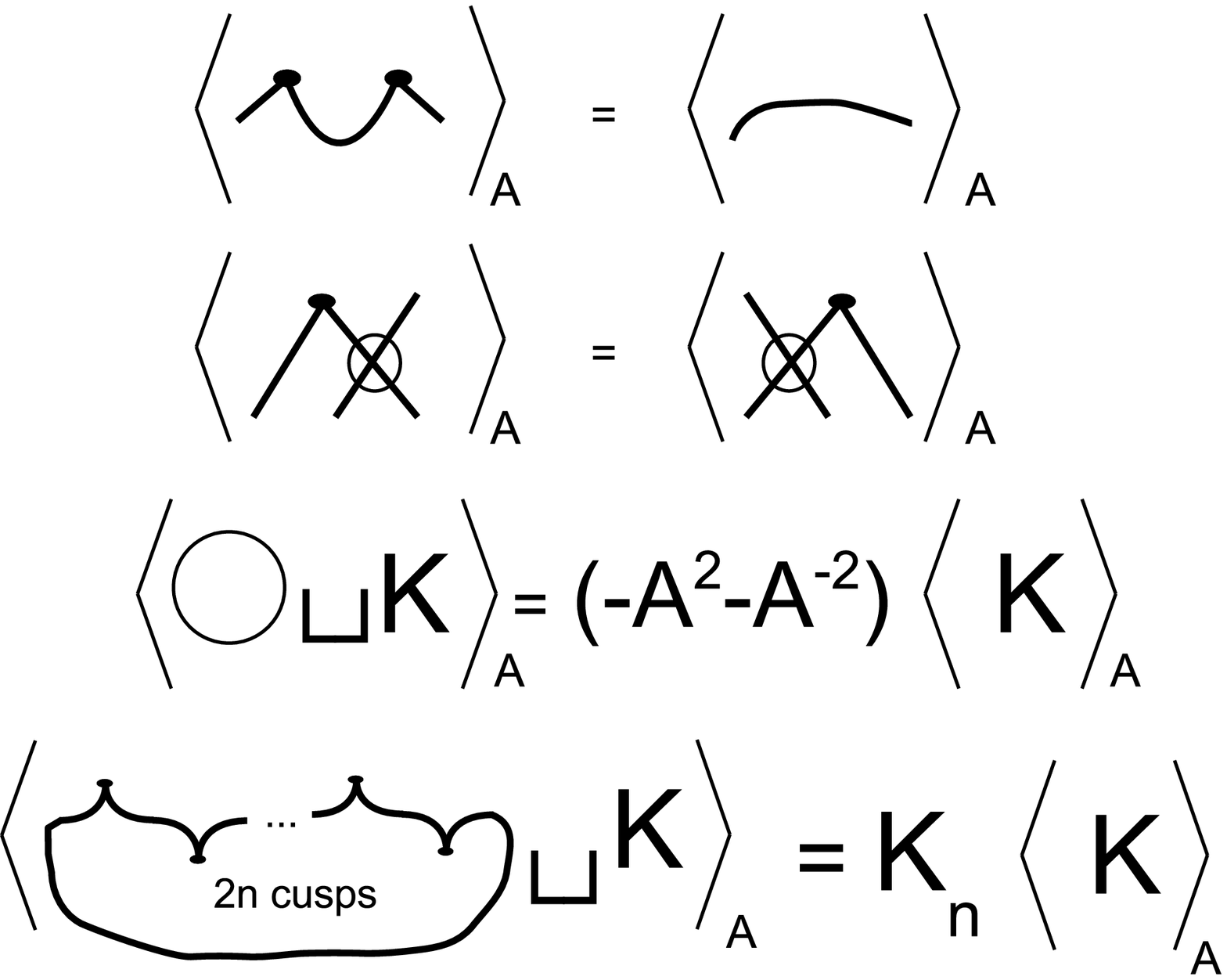}
\caption{Arrow Polynomial Reduction Relations}
\label{fig:APReduce}
\end{figure}

Notice that the smoothing relations for the arrow polynomial depend on the sign of the crossing.  While we still have an A-smoothing and B-smoothing and will refer to these choices, we also differentiate the smoothings by whether they agree with the original (pre-smoothed) orientation of the knot diagram.  If the orientations agree we call these \emph{oriented smoothings} else they are \emph{disoriented smoothings}. The disoriented smoothings create cusps in the state which satisfy the relations in Figure \ref{fig:APReduce}. Moreover, these cusps introduce an infinite family of variables $ \{ K_{n} \}$,$ {n \in \mathbb{N}}$. That is, we cancel consecutive cusps pointing in the same direction (locally both inward or outward) and resolve virtual crossings following the rules in Figure \ref{fig:APReduce}. The remaining $2n$ alternately-oriented cusps on a loop are counted and we assign the loop the value $ \{ K_{n} \}$ when $n > 0$.  We refer to these variables as \emph{arrow numbers} \cite{DKVirtualCrossingNumber} \cite{ExtendedBracket}.

\begin{defn}
Given a virtual knot K with diagram D, the \emph{un-normalized arrow polynomial of K} is given by
\[ \langle D \rangle_{A} = \displaystyle\sum\limits_{s \in \mathcal{S}} \langle D \mid s \rangle_{A}
\]
where $\mathcal{S}$ is the collection of all states of D.\end{defn}

\begin{defn}
Given a virtual knot K with diagram D, the \emph{normalized arrow polynomial of K} is given by
\[AP(K) = AP(D) = (-A)^{-3\omega(D)} \langle D \rangle_{A}\]
where $\langle D \rangle_{A}$ is the arrow polynomial of D and
\[\omega(D) = \textrm{writhe}(D) = (\#\textrm{ positive crossings in D}) - (\#\textrm{ negative crossings in D}).\]
\end{defn}

\begin{thm} The normalized arrow polynomial is an invariant of virtual knots.
\end{thm}

\begin{proof}See \cite{ExtendedBracket}.
\end{proof}

\begin{rem} The Jordan Curve Theorem implies that the arrow polynomial is equivalent to the normalized bracket polynomial for classical knots.
\end{rem}

\subsubsection{The Arrow Polynomial and Surface Genus}

Recall that virtual knots are in 1-1 correspondence with equivalence classes of knots in thickened oriented surfaces modulo 1-handle stabilization and Dehn twists. This raises the question, given a knot, what is the minimal genus for this embedding.

\begin{defn} Given a virtual knot $K$, the \emph{(orientable) surface genus of $K$}, $s(K)$, is the minimal genus of the surface $S_g$ such that $S^1 \hookrightarrow S_g \times I$ corresponds to $K$.
\end{defn}

\begin{thm}\label{thm:APSurfaceGenus} (Theorem 4.5 of \cite{DKVirtualCrossingNumber}) Let K be a virtual knot diagram with arrow polynomial $\langle K \rangle_{A}$. Suppose that $\langle K \rangle_{A}$ contains a summand with the monomial $K_{i_1}^{e_1}K_{i_2}^{e_2}\cdots K_{i_n}^{e_n}$ where $i_j \neq i_k$ for all $j, k$ in the set $\{1, 2, \ldots, n \}$. Then $n$ determines a lower bound on the genus $g$ of the minimal genus surface in which K embeds. That is, if $n \geq 1$ then the minimum genus is 1 or greater and for $g \geq 2$ if $n > 3g - 3$ then $s(K) > g$.
\end{thm}

\begin{proof}
The proof follows from showing that non-zero arrow numbers correspond to essential curves on the surface. The bounds follow from considering the maximal number of non-intersecting essential curves on the surface which do not bound annuli.
\end{proof}

\begin{rem} In the $g = 1$ case it can be shown that if there is more than 1 non-intersecting essential curve then they must bound an annulus.  Hence if $\langle L \rangle_A$ contains a summand with a monomial of the form $K_{i}K_{j}$ with $i \neq j$ then the minimal surface genus is at least 2.\\
\end{rem}

\subsection{Examples}

\begin{figure}[h!]
\centering
    \includegraphics[height=3.3cm]{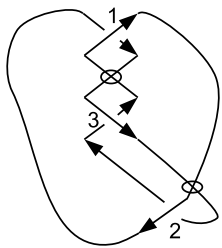}
\caption{Virtual Knot 3.1}
\label{fig:3.1}
\end{figure}

Following the naming conventions in Jeremy Green's virtual knot tables \cite{GreenTables} we consider virtual knot 3.1 as shown in Figure \ref{fig:3.1}. Figures \ref{fig:3.1Bracket} and \ref{fig:3.1Arrow} show how one uses the state-sum formulas from the previous sections to arrive at the respective polynomials. Note that by the previous theorem the arrow polynomial gives that the surface genus of virtual knot 3.1 is at least one.  The diagram above has surface genus 2.

\begin{figure}[h!]
\centering
    \includegraphics[height=0.5\textheight]{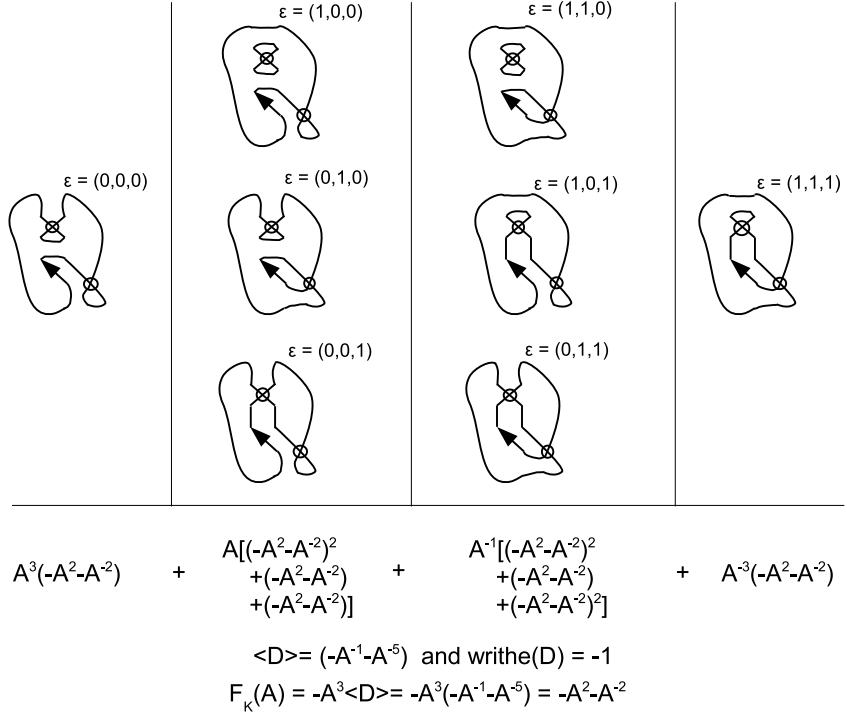}
\caption{Virtual Knot 3.1 Bracket Polynomial State-Sum}
\label{fig:3.1Bracket}
\end{figure}

\begin{figure}[h!]
\centering
    \includegraphics[height=0.5\textheight]{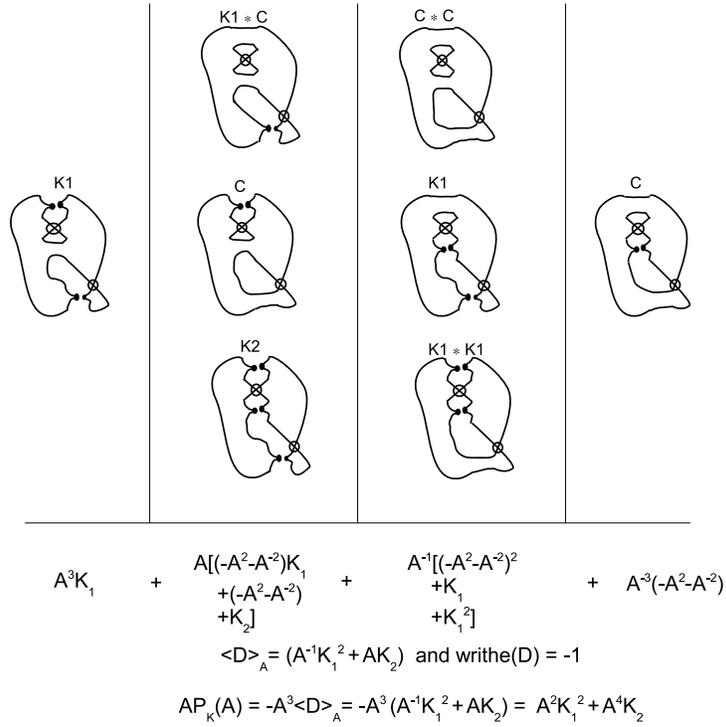}
\caption{Virtual Knot 3.1 Arrow Polynomial State-Sum}
\label{fig:3.1Arrow}
\end{figure}

\section{Categorification}

\subsection{Khovanov Homology for Virtual Knots}

For completeness we recall the definition of Khovanov homology \cite{Khovanov1} \cite{Khovanov2}. Our construction follows closely that of Dror Bar-Natan \cite{DrorCat} and Kauffman \cite{IntroVKT}.  For other descriptions of the construction we point the reader to Khovanov \cite{Khovanov1} \cite{Khovanov2}, Wehrli \cite{WehrliSpanningTrees}, Viro \cite{ViroKhoHo}, Shumakovitch \cite{Shumakovitch}, Elliott \cite{Elliott}, Kauffman\cite{IntroVKT}, and Manturov \cite{ArbitraryCoeffs}. For technical reasons involving the construction of the categorification of the arrow polynomial we will take coefficients over the field $\mathbb{Z}_2$.  It should be noted that the construction of Khovanov homology for virtual knots can be extended to arbitrary coefficients following the construction of Manturov in \cite{ArbitraryCoeffs}. As one would expect, Manturov's definition is equivalent to that of Khovanov for classical knots.

Before we recall the construction, we first rewrite the normalized bracket polynomial in order to simplify the definition. For a given virtual knot or link K with corresponding diagram D, let $c(K)$ denote the crossing number of D. Sending $\langle K \rangle$ to $A^{-c(K)} \langle K \rangle$ and $A$ to $-q^{-1}$ we get the following definition of the bracket polynomial:
\[\langle {\O} \rangle = 1 \text{ ; } \langle \bigcirc K \rangle = (q + q^{-1})  \langle  K \rangle \text{ ; }
\langle \Across \rangle  = \langle \Asmooth \rangle - q \langle \Bsmooth \rangle
\]

And, as pointed out in Bar-Natan \cite{DrorCat} we can summarize the Khovanov bracket via the axioms:
\[\llbracket {\O} \rrbracket = 0 \rightarrow \mathbb{Z}_2 \rightarrow 0 \text{ ; } \llbracket \bigcirc K \rrbracket = V \otimes \llbracket  K \rrbracket  \text{ ; }
\llbracket \Across \rrbracket  =  \mathit{Cone}(\llbracket \Asmooth \rrbracket \stackrel{\text{\tiny d}}{\rightarrow}  \llbracket \Bsmooth \rrbracket\{1\})
\]
where $V = \mathbb{Z}_2 [ X ] / (X^2)$, $\{1\}$ is the ``degree shift by one'' operation on the quantum grading and  $\mathit{Cone}$ is the mapping cone over the differential $d$ which we have yet to define. Note that we will use the enhanced state definition common in the literature, where each enhanced state corresponds to a labeling of the circle by 1 or X.  This has the correspondence
 \[1 \Leftrightarrow q^{+1}  \text{ and } X \Leftrightarrow q^{-1} .\]

\begin{rem}
For convenience we will continue to use the terms A-smoothing and B-smoothing as defined earlier.
\end{rem}

More constructively, consider the collection of enhanced states S arising from applying an A-smoothing and B-smoothing at each crossing and furthermore labeling each of the resulting circles by either 1 or X. Define the \emph{Khovanov complex} $\mathcal{C}^{\bullet,\bullet}$ by setting $ \mathcal{C}^{i,j}$ to be the linear span of states $s \in S$ where $i = n_B(s) = \text{``the number of B-smoothings in s''}$ and $j = j(s) = n_B(s) + \lambda(s)$ where $\lambda(s) =$ ``the number of loops in s labeled 1 minus the number of loops labeled X''. We will refer to $i$ as the \emph{homological grading} and $j$ as the \emph{quantum grading}.\\

For any two states $s, s'$ that differ by replacing an A-smoothing by a B-smoothing at a single site respectively we define a \emph{local differential}, $d_{s, s'}$, such that the homological grading is increased by 1 and the quantum grading is preserved.  Once this is defined, we then have the differential
\[d : \mathcal{C}^{i,j} \rightarrow \mathcal{C}^{i+1,j}\]
 defined by
\[d(s) = \sum_{s'} d_{s, s'}\]\\

All that remains is to determine the possible values for $d_{s, s'}$. Since we are only concerned with resmoothing at a single site there are 3 possible scenarios relating $s$ and $s'$ in the setting of virtual knots.\\

\begin{tabular}{l c}
  \emph{Circle Annihilation} & \annihilation \\
  \\
  \emph{Circle Creation} & \creation \\
  \\
  \emph{Single-Cycle Smoothing} & \singlecycle \\
  \\
\end{tabular}

For any circles in state $s$ not involved in resmoothing we set $d_{s, s'}$ to act as the identity on the enhanced states. For the enhanced circles involved in the resmoothing we define $m, \Delta$ and $\eta$ as follows:\\
\begin{center}
\begin{tabular}{r c c c}
\ldelim \{ {4}{5mm}[\emph{m :}] & $1 \otimes 1$ & $\rightarrow$ & $1$ \\
& $1 \otimes X$ & $\rightarrow$ & $X$ \\
& $X \otimes 1$ & $\rightarrow$ & $X$ \\
& $X \otimes X$ & $\rightarrow$ & $0$ \\
\\
\ldelim \{ {2}{1.5mm}[\emph{$\Delta$:}] & $1$ & $\rightarrow$ & $1 \otimes X + X \otimes 1$ \\
& $X$ & $\rightarrow$ & $X \otimes X$ \\
\\
\ldelim \{ {2}{1.5mm}[\emph{$\eta$:}] & 1 & $\rightarrow$ & $0$ \\
& $X$ & $\rightarrow$ & $0$ \\
\end{tabular}
\end{center}

We are now in a position to define the \emph{Khovanov homology of a knot or link K},
\[ \mathcal{H}(K) = \llbracket K \rrbracket [-n_{-}]\{n_{+}-2n_{-}\} \]
where $[l]$ is the shift operator on the homological grading and $\{l\}$ is the shift operator on the quantum grading.
Moreover we can define the \emph{Khovanov invariant} to be the Poincar\'{e} polynomial:
\[ Kh(K) := \sum_{i,j} t^i q^j dim \mathcal{H}^{i, j}(K) \]

\begin{exa}
Consider the virtual knot 3.1 in Figure \ref{fig:3.1}. The Khovanov complex for the unenhanced states is shown in Figure \ref{fig:3.1KhoHo}. It is a small exercise to show $Kh(VK_{3.1}) = q + q^{-1}$
\end{exa}

\begin{figure}[h!]
\centering
    \includegraphics[width=.96\textwidth]{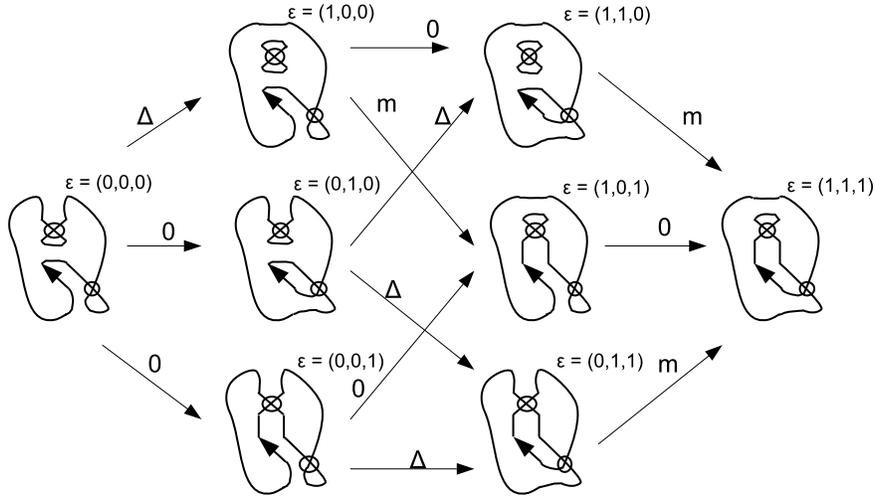}
\caption{Virtual Knot 3.1 Khovanov Homology Complex}
\label{fig:3.1KhoHo}
\end{figure}

We will omit the well known proof that $d^2 = 0$, which amounts to checking all cases in the virtual setting, as well as the proof of invariance under the Reidemeister Moves. These details can be found in \cite{Khovanov1}, \cite{DrorCat}, \cite{DrorCob},  \cite{ArbitraryCoeffs}, and \cite{ViroKhoHo}.\\

If not for the self-imposed $\mathbb{Z}_2$ setting, we could also discuss applications of Lee's spectral sequence \cite{Endomorphism}, \cite{Rasmussen} in the virtual setting.  We plan to return to this subject in a future paper.\\

\subsection{Categorifications of the Arrow Polynomial}\label{rem:ArrowCat}

In 2009 Dye, Kauffman and Manturov \cite{DKM} introduced two categorifications of the arrow polynomial for virtual knots.  Both constructions are homology theories defined over $\mathbb{Z}_{2}$ and agree with Khovanov homology over $\mathbb{Z}_{2}$ for classical knots. We remark that this construction is similar in flavor to the construction of Khovanov homology described in the previous section. The major difference in the constructions presented here are the considerations of additional gradings arising from the arrow numbers. We will use the same renormalization for categorification of the arrow polynomial as we did in Khovanov homology, namely sending $\langle K \rangle$ to $A^{-c(K)} \langle K \rangle$ and $A$ to $-q^{-1}$, where $c(K)$ is the crossing number for the diagram of K.

We first recall the construction introduced in \cite{DKM} introducing the multiple grading and vector grading.  Consider the collection of enhanced states S arising from applying an A-smoothing and B-smoothing at each crossing, where A- and B-smoothings are defined as in Figure \ref{fig:APSmooth}, and furthermore labeling each of the resulting circles by either 1 or X. Define the \emph{arrow complex} $\mathcal{C}_{A}^{\bullet,\bullet,\bullet,\bullet}$ by setting $ \mathcal{C}_{A}^{i,j,m,v}$ to be the linear span of enhanced states $s \in S$ where $i = n_B(s) = \text{``the number of B-smoothings in s''}$ and $j = j(s) = n_B(s) + \lambda(s)$ where $\lambda(s) =$ ``the number of loops in s labeled 1 minus the number of loops labeled X''. We will refer to $i$ as the \emph{homological grading} and $j$ as the \emph{quantum grading}.\\

Given a state s define the \emph{multiple grading of s}, $mg(s)$, to be the set of arrow numbers of s.\\

Given an enhanced state s, consider the collection $\Lambda_s$ of enhanced circles carrying nonzero arrow numbers.  For a circle $c \in \Lambda$ with arrow number $p$ let the \emph{order} of $c$, $o(c)$,  be the value of $k$ such that $p = 2^{k-1}*l$ with $gcd(2,l)=1$. Define the function $vg(c)$ by $vg(c) = e_{o(c)}$ if c is labeled by X and $vg(c) = -e_{o(c)}$ if c is labeled by 1, where ${e_1, e_2, \ldots }$ is the standard basis for $\mathbb{R}^{\infty}$.  Then the \emph{vector grading}, $vg(s)$, is given by
\[ vg(s) = \sum_{c \in \Lambda} vg(c)
\]

\begin{exa}The state $s = \gradingexample$ has
\[mg(s) = \{K_1, K_2\} \text{ and } vg(s) = (-2,1,0,0,0,...)\]
\end{exa}

As before, for any two states $s, s'$ that differ by an A- to B- resmoothing  at a single site it remains to define a \emph{local differential} $d_{s, s'}$ such that the homological grading is increased by 1 and the quantum grading, multiple grading and vector grading are all preserved.  Once this is defined, we have the differential
\[d : \mathcal{C}_{A}^{i,j, m, v} \rightarrow \mathcal{C}_{A}^{i+1,j, m, v}\]
 defined by
\[d(s) = \sum_{s'} d_{s, s'}\]\\

Finally, The local differential $d_{s, s'}$ is defined by $d_{s, s'} = \tilde{\partial}_{s, s'} \circ \partial_{s, s'} $ where $\partial_{s, s'}$ is the Khovanov local differential between the corresponding states and $\tilde{\partial}_{s, s'}$ is the projection map preserving the multiple grading and vector grading. It is a short exercise to show that this satisfies the requisite $d^2 = 0$ through checking all possible cases. A key observation for the proof is noting that for circle annihilation and circle creation arrow numbers are $\pm$-additive while a single cycle resmoothing causes the arrow number to change by 1 as is shown in Figure \ref{fig:ArrowNumberResmoothing}.

\begin{figure}[h!]
\centering
    \includegraphics[width=.5\textwidth]{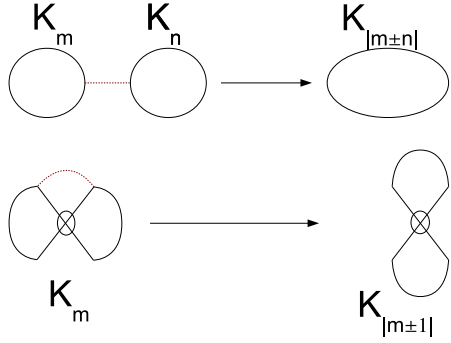}
\caption{The Effect of Resmoothing on Arrow Numbers}
\label{fig:ArrowNumberResmoothing}
\end{figure}

We are now in a position to define the homology for the fully-graded categorification of the arrow polynomial of a knot or link K, $\mathcal{H}{_A}(K)$, by renormalizing analogously to Khovanov homology. Moreover we can define the \emph{fully-graded arrow categorification invariant} to be the Poincar\'{e} polynomial:
\[ AKh(K) := \sum_{i,j \in \mathbb{Z}, m\in M(D), v\in V(D)} m \times v \times t^i \times q^j \times dim \mathcal{H}{_A}^{i, j, m, v}(K)
\]
where
\[S(D) = \{\text{enhanced states of D}\} \text{ , } M(D) = \bigcup_{s \in S} mg(s) \text{ and } V(D)= \bigcup_{s \in S} vg(s)\]
for a diagram D of K.\\

\begin{exa}
Consider the virtual knot 3.1 in Figure \ref{fig:3.1}. The cube complex for the unenhanced states is shown in Figure \ref{fig:3.1ArrowCat}. It is a small exercise to show
\[AKh(VK_{3.1}) = \frac{\text{vg}(2,1) K[2]}{q^3 t}+\frac{\text{vg}(1,2) K[1]}{q^3}+\frac{\text{vg}(2,-1) K[2]}{q t}+q \text{vg}(1,-2) K[1]+\frac{2 K[1]}{q}\]
\end{exa}

\begin{rem} To translate the polynomial into the form of the definition one only has to see how to read the multiple grading and vector grading for a given monomial.  The multiple grading is simply the collection of coefficients of the form $K[i] (= K_i)$. The vector grading is obtained as follows.  Each coefficient of the form $\text{vg}(i,a)$ corresponds to having coefficient $a_i = a$ when the vector grading is written as $\displaystyle\sum\limits_{i \in \mathbb{N}} a_i e_i$. The product of vector gradings corresponds to the sum of the individual gradings.  For example $\text{vg}(2,1)\text{vg}(1,-1)$ corresponds to the vector grading $(-1,1,0,0,0,\ldots)$.
\end{rem}

\begin{figure}[h!]
\centering
    \includegraphics[width=.9\textwidth]{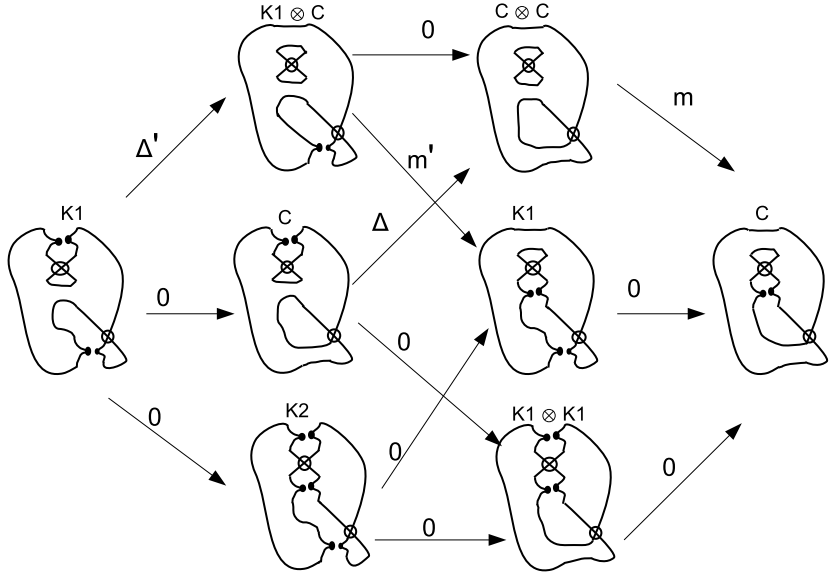}
\caption{Virtual Knot 3.1 Arrow Polynomial Categorification Complex}
\label{fig:3.1ArrowCat}
\end{figure}

We will again omit the proof of invariance under the Reidemeister Moves, which is similar to the equivalent proof for Khovanov homology. Similarly we leave out most of the proof that $d^2 = 0$ other than to point out why we work over $\mathbb{Z}_2$. As with Khovanov homology, the proof follows by showing that the differential commutes for every possible face with every possible grading configuration in the cube complex. Much of this is follows from the additivity relations for arrow numbers under resmoothing. However, the face in Figure \ref{fig:badface} is an example of a general type necessitating working over $\mathbb{Z}_2$.\\

\begin{figure}[h!]
\centering
    \includegraphics[width=.6\textwidth]{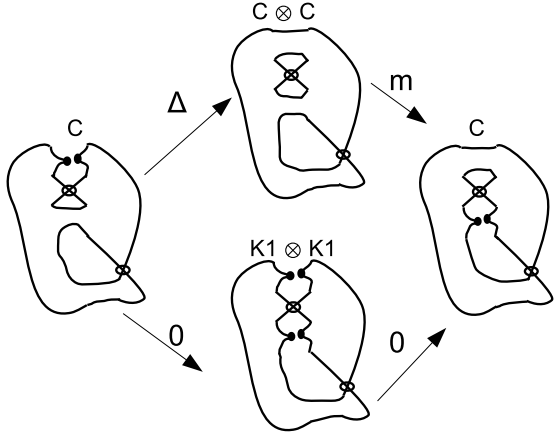}
\caption{An example of the necessity of working over $\mathbb{Z}_2$}
\label{fig:badface}
\end{figure}

In \cite{DKM} a simpler categorification for the arrow polynomial is introduced.  We can arrive at this construction through a simple modification of the differential $\tilde{\partial}_{s, s'}$. Suppose we represent the vector grading as $vg(s) = \sum_{i \in \mathbb{N}} a_{i_s} * e_i$ where ${e_1, e_2, \ldots }$ is the standard basis for $\mathbb{R}^{\infty}$. Rather than preserving the multiple grading and vector grading, $\tilde{\partial}_{s, s'}$ is defined by\\

\begin{tabular}{ccc}
\ldelim \{{2}{9mm}[$\tilde{\partial}_{s, s'} = $] & 1  $\text{ , if } a_{1_s} \equiv a_{1_{s'}} \text{ mod 2}$\\
   & 0  $\text{ , if } a_{1_s} \not\equiv  a_{1_{s'}} \text{ mod 2}$\\
\end{tabular}\\

We have constructed a Mathematica program to calculate all of the categorifications mentioned in this section for knots with at most 6 classical crossings based on Jeremy Green's table \cite{GreenTables}. We have been unable to find two virtual knots that are distinguished by the fully-graded categorification and not by the simpler categorification. Additionally, we have no examples of knots which are distinguished from the unknot by the categorification and not by the arrow polynomial.  \\

\subsubsection{The Fully-Graded Arrow Categorification and Surface Genus}

We may extend Theorem \ref{thm:APSurfaceGenus} on surface genus bounds produced by the arrow polynomial to the fully-graded categorification as follows.

\begin{thm}\label{sgac} Let K be a virtual knot diagram with fully-graded arrow categorification invariant $ AKh(K)$. Suppose that $ AKh(K)$ contains a summand having multiple grading $\mathcal{M}$, a nonempty set of arrow numbers, with $\left| \mathcal{M} \right|=n$. Then $n$ determines a lower bound on the genus $g$ of the minimal genus surface in which K embeds. That is, if $n = 1$ then the minimum genus is at least 1, if $n = 2$ then the minimum genus is 2 or greater and for $n \geq 3$ if $n > 3g - 3$ then $s(K) > g$.\\
\end{thm}

We add a bit more detail to the earlier sketch to highlight why we only get the extension in the fully-graded categorification.  The proof relies on the following fact from \cite{Hatcher}.\\

\begin{lem} Consider a collection $\mathcal{A}$ of non-intersecting essential curves (i.e. not contractible) on an orientable surface $S_g$ of genus $g$ no pair of which co-bound an annulus. If $g=1$ then $\left|\mathcal{A}\right| \leq 1$ and if $g \geq 2$ then $\left|\mathcal{A}\right| \leq 3g-3$.\\
\end{lem}

\textit{Proof of \ref{sgac}:} The proof follows by the above lemma once we show that multiple grading corresponds to n non-intersecting essential curves of which no pair co-bound an annulus. Since the multiple grading is an invariant of the knot we have that any embedding into $S_g \times I$ for $K$ must contain a state $s$ with  $mg(s) = \mathcal{A}$.\\
To see that each element of $\mathcal{A}$ is an essential curve suppose for contradiction $K_{i_l} \in \mathcal{A}$ bounds a disk.  By the disoriented smoothing relation each cusp in $K_{i_l}$ is paired with another cusp somewhere in $s$ corresponding to the other half of the smoothing. Since $K_{i_l}$ bounds a disk (in the projection to $S_g$) the Jordan Curve Theorem implies the interior and exterior cusps cannot be paired.  If we consider only the internal cusps in $K_{i_l}$ they too cannot be paired with one-another (for odd arrow numbers this follows from parity and for even arrow numbers this follows from orientation.) Thus an inner cusp of $K_{i_l}$  must be paired with the cusp either of another $K_{i_j}$ or of circle with an even number of canceling cusps. In either case we produce another cusps with which to repeat the argument.  Since our knot has a finite number of crossing (hence a finite number of cusps) this is a contradiction. A similar argument shows that given $K_{i_l}, K_{i_j} \in \mathcal{A}$ with ${i_l} \neq {i_j}$ they cannot co-bound an annulus.\qed

\subsection{Categorification and Width}

The subject of the width of Khovanov homology for various classes of knots has been of interest since Khovanov's seminal work \cite{Khovanov1}. Here we produce a definition of width for the virtual setting, recall known results in the classical setting and give some basic results in the virtual knot setting.\\

In the classical case it was noticed early on that when plotting the homological degree versus the quantum degree for the support of the Khovanov invariant the majority of small knots were supported on 2 diagonals corresponding to the signature of the knot $\pm 1$.  In the classical case we say that a knot is \emph{H-thin} if its Khovanov homology is supported on 2 diagonals corresponding to lines $t-2q = constant$, else it is called \emph{H-thick}. The thickness of a number of classes of classical knots is known and a recent summary of the known results can be found in \cite{Elliott}.  In the case of alternating knots Lee \cite{Endomorphism} proved that they are H-thin. Since all alternating knots are of even parity (which we will define shortly), Lee's proof extends to alternating virtual knots as was pointed our by Viro \cite{ViroKhoHo}.  \\

In the virtual knot setting we need to re-examine our definition of H-thick and H-thin. If we wish to use the thickness of the homology to determine if the Khovanov homology holds additional information over the Jones polynomial we quickly run into trouble in the virtual case. It is no longer enough to determine if the homology is supported on 2 diagonals to determine if the homology contains more information than the polynomial (not to mention that the signature is not well-defined). For instance virtual knot 2.1 \twoone \textrm{ has} bracket polynomial
\[ (A^4 + A^6 - A^{10})(-A^2 - A^{-2})\]
and Khovanov invariant
\[(q^6 + q^4)t^2 + (q^4 + q^2) t + q^3 + q.\]
By the classical definition virtual knot 2.1 is H-thick as it is supported on 4 diagonals.  However, the Khovanov homology does not hold any additional information.  \\

Two good questions to ask are (1) for a given virtual knot, on how many diagonals is the homology $\mathcal{H}(K)$ supported and (2) what is the maximal width between the diagonals (i.e. $c_{max} - c_{min}$ for the supported diagonals $t - 2q = c$.)  We call the solution to the first the \emph{thickness} of the homology and denote it by $Th(\mathcal{H}(K))$ for a given knot K. The second is referred to as the \emph{width} of the homology and denoted by $W(\mathcal{H}(K))$.

More can be said about knots with orientable atoms.  We recall from Manturov \cite{ManturovKnotTheory} that for a given knot K, an atom for K is a pair $(M, \Gamma)$  where M is a surface without boundary (not necessarily connected or orientable) and $\Gamma$ is a 4-valent graph on M such that $(M, \Gamma)$ admits a checkerboard coloring. We say an atom is orientable if M is orientable.\\

In \cite{ManturovWidth} Manturov proves that for knots with orientable atoms we have:\\

\begin{thm} For a knot K with orientable atom, $Th(Kh(K)) \leq g(K) + 2$ where $g(K)$ is the Turaev genus (or atom genus) of the knot.\\
\end{thm}

The proof requires a careful examination of the interplay between the Turaev genus of the knot and the number of crossings in the knot diagram.  Note that the orientable condition is important. Virtual knot 2.1 has $Th(Kh(K)) = 4$ (as shown above), however $g(K) = 1/2$ since K can be placed on the projective plane in a checkerboard fashion.  \\

Asking the same questions for the fully-graded arrow categorification we immediately see:\\

\begin{thm} Given a virtual knot $K$ suppose the arrow polynomial $\langle K \rangle_{A}$ contains a monomial with non-zero arrow number $K_{i_1}^{e_1} \ldots K_{i_n}^{e_n}$ then \[W(AKh(K))) \geq 2(\sum_{i = 1, \ldots, n} e_i)\]
\end{thm}

\begin{proof} This is an immediate consequence of categorification. Since we have categorified up to multiple grading and vector grading the monomial with non-zero arrow number $K_{i_1}^{e_1} \ldots K_{i_n}^{e_n}$ corresponds to a term in $AKh(K)$. If we consider the vector gradings related to the corresponding unenhanced state in the cube complex we see the the largest and smallest correspond to the ``all 1'' labeling and the``all X'' labeling.  The corresponding width between these states is $2(\displaystyle\sum\limits_{i = 1, \ldots, n} e_i)$ giving the theorem.
\end{proof}

The above proof actually gives more if we consider all possible labelings.\\

\begin{thm} Given a virtual knot $K$ suppose the arrow polynomial $\langle K \rangle_{A}$ contains a monomial with non-zero arrow number $K_{i_1}^{e_1} \ldots K_{i_n}^{e_n}$, then \[ Th(AKh(K))) \geq \sum_{i = 1, \ldots, n} e_i \]
\end{thm}

\subsection{Computations and Examples}

Inspired by Dror Bar-Natan's program for Khovanov homology \cite{DrorCat} we have created a collection of Mathematica programs which calculate the categorifications of the arrow polynomial mentioned above. More information on these programs can be found in the Appendix. The program are also available from the first authors website.  The following examples have been computed using the list of knots available at Jeremy Green's Knot Tables \cite{GreenTables} as well as those in LinKnot \cite{LinKnot}.\\

\subsubsection{Identical Khovanov Invariant but Distinguished by a Categorification of the Arrow Polynomial}

\begin{figure}[h!]
\centering
    \includegraphics[width =.37\textwidth]{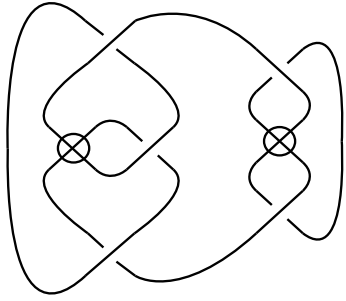}
\caption{Virtual Knot 5.129}
\label{fig:5.129}
\end{figure}

\begin{figure}[h!]
\centering
    \includegraphics[width =.37\textwidth]{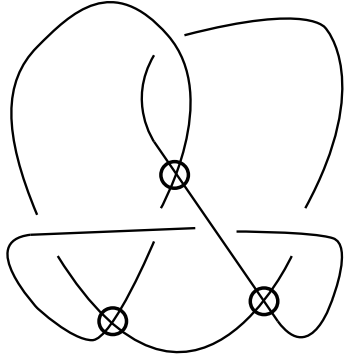}
\caption{Virtual Knot 5.267}
\label{fig:5.267}
\end{figure}

\begin{exa} Virtual knots 5.129 in Figure \ref{fig:5.129} and 5.267 in Figure \ref{fig:5.267} both have Khovanov invariant
\[Kh(5.129) = Kh(5.267) = \frac{1}{q^5 t^2}+\frac{1}{q^3 t^2}+\frac{1}{q^3 t}+q^2 t+\frac{1}{q^2}+\frac{1}{q t}+q+\frac{1}{q}+t+1 \]
and arrow polynomial
\[AP(5.129) = AP(5.267) = -A^{10}+A^6-A^4 \text{K2}-2 A^2 \text{K1}^2-A^2 \text{K1}+\frac{\text{K1}}{A^2}+2 A^2-\text{K2} \]
However,
\begin{align*}
AKh(5.129) = {} & \frac{\text{vg}(2,1) K[2]}{q^3 t}+\frac{2 \text{vg}(1,2) K[1]}{q^3}+q^2 t \text{vg}(1,-1) K[1]+\frac{\text{vg}(1,1) K[1]}{q^2}\\
& +q t \text{vg}(2,-1) K[2]+\frac{\text{vg}(2,-1) K[2]}{q t}+\frac{t \text{vg}(2,1) K[2]}{q}\\
& +2 q \text{vg}(1,-2) K[1] +t \text{vg}(1,1) K[1]+\text{vg}(1,-1) K[1]+\frac{4 K[1]}{q}\\
& +\frac{1}{q^5 t^2}+\frac{1}{q^3 t^2}+\frac{2}{q^3 t} +\frac{2}{q t}+q t+\frac{t}{q}+q+\frac{1}{q}
\end{align*}
and
\begin{align*}
AKh(5.267) = {} & \frac{\text{vg}(2,1) K[2]}{q^3 t}+\frac{2 \text{vg}(1,2) K[1]}{q^3}+q^2 t \text{vg}(1,-1) K[1]+\frac{\text{vg}(1,1) K[1]}{q^2}\\
& +q t \text{vg}(2,-1) K[2]+\frac{\text{vg}(2,-1) K[2]}{q t}+\frac{t \text{vg}(2,1) K[2]}{q}\\
& +2 q \text{vg}(1,-2) K[1] +t \text{vg}(1,1) K[1]+\text{vg}(1,-1) K[1]+\frac{4 K[1]}{q}\\
& +\frac{1}{q^5 t^2}+\frac{1}{q^3 t^2}+\frac{2}{q^3 t}+\frac{2}{q t}
\end{align*}
\end{exa}

\subsubsection{Rational Virtual Knots}

Lee \cite{Endomorphism} showed that the Khovanov homology of an alternating knot is completely determined by its Jones polynomial. Recall that every classical rational knot is isotopic to an alternating knot (see Theorem 3.5 of  \cite{RationalKnots}). Hence the Khovanov homology of every classical rational knot is completely determined by its Jones polynomial.\\

This result of Lee is not the case for virtual knots and the categorifications of the arrow polynomial. The following examples were found with the help of Slavik Jablan and the program LinKnot \cite{LinKnot}.\\

\begin{figure}[h!]
\centering
    \includegraphics[height=4cm]{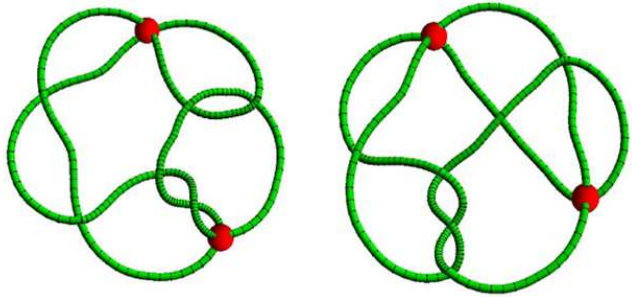}
\caption{Virtual Knots with Equivalent Polynomials but Distinguished by a Categorification}
\label{fig:Slavik}
\end{figure}

Both of the knots in Figure \ref{fig:Slavik} have identical normalized bracket polynomial
\[ \frac{1}{A^{26}}+\frac{1}{A^{24}}-\frac{1}{A^{20}}-\frac{1}{A^{18}}+\frac{1}{A^{12}}-\frac{1}{A^{10}}-\frac{1}{A^8}-\frac{1}{A^6}\]

and normalized arrow polynomial
\[A^{10}+A^8 \text{K1}-\frac{1}{A^6}-2 A^4
   \text{K1}-\frac{\text{K1}}{A^4}+\frac{1}{A^2}+2 \text{K1}\]

The knot on the left hand side has Khovanov invariant

\begin{eqnarray*}
q^{13} t^5+q^{12} t^4+q^{11} t^5+q^{11} t^4+q^{10} t^4+2 q^{10} t^3\\
+q^9 t^4+2 q^8 t^3+2 q^8 t^2+2 q^6 t^2+q^6 t+q^5+q^4 t+q^3
\end{eqnarray*}

and fully-graded arrow categorification invariant
\begin{eqnarray*}
q^{12} t^4 \text{vg}(1,-1) K[1]+q^{10} t^4 \text{vg}(1,1) K[1]\\
+2 q^{10} t^3 \text{vg}(1,-1) K[1]+2 q^8 t^3 \text{vg}(1,1) K[1]\\
+2 q^8 t^2 \text{vg}(1,-1) K[1]+2 q^6 t^2 \text{vg}(1,1) K[1]\\
+q^6 t \text{vg}(1,-1) K[1]+q^4 t \text{vg}(1,1) K[1]+q^{13} t^5\\
+q^{11} t^5+q^{11} t^4+q^9 t^4+q^5+q^3
\end{eqnarray*}

while the knot on the right hand side has Khovanov invariant

\begin{eqnarray*}
q^{13} t^5+q^{12} t^4+q^{11} t^5+q^{11} t^4+q^{10} t^4+2 q^{10} t^3\\
+q^9 t^4+q^9 t^3+q^9 t^2+2 q^8 t^3+2 q^8 t^2+q^7 t^3+2 q^7 t^2+q^7 t\\
+2 q^6 t^2+q^6 t+q^5 t^2+q^5 t+q^5+q^4 t+q^3
\end{eqnarray*}

and fully-graded arrow categorification invariant
\begin{eqnarray*}
q^{12} t^4 \text{vg}(1,-1) K[1]+q^{10} t^4 \text{vg}(1,1) K[1]\\
+2 q^{10} t^3 \text{vg}(1,-1) K[1]+2 q^8 t^3 \text{vg}(1,1) K[1]\\
+2 q^8 t^2 \text{vg}(1,-1) K[1]+2 q^6 t^2 \text{vg}(1,1) K[1]\\
+q^6 t \text{vg}(1,-1) K[1]+q^4 t \text{vg}(1,1) K[1]+q^{13} t^5\\
+q^{11} t^5+q^{11} t^4+q^9 t^4+q^9 t^3+q^9 t^2+q^7 t^3+2 q^7 t^2\\
+q^7 t+q^5t^2+q^5 t+q^5+q^3
\end{eqnarray*}

\section{Parity and Virtual Knots}
\subsection{Parity and the Reidemeister Moves} \label{sec:ParityRMs}

Given a diagram $D$ for a knot $K$ label each crossing uniquely 1 through $n$, where $n$ is the total number of crossings in $D$.  Let $P$ an arbitrary base-point on the knot. Starting at $P$ and following the orientation of the knot we can construct a sequence of length $2n$ with terms corresponding to each crossing we encounter. Each term is a 3-tuple of the form ($O$ or $U$, Crossing Number, $\pm$) where $O$ or $U$ corresponds to an over or under-crossing respectively and $\pm$ corresponds to the sign of the crossing.  The resulting code is referred to as the (signed, oriented) \emph{Gauss code} for the diagram $D$ of the knot $K$. \\

The Gauss code can be represented diagrammatically as follows.  Given a circle (often refereed to as the \emph{core circle}) place upon it in a counterclockwise fashion $2n$ points where each point is labeled by a crossing name (an integer between 1 and n) in the cyclic order corresponding to the Gauss code.  Between the two occurrences of a crossing on the core circle, place a signed, oriented chord where the sign corresponds to the crossing sign and the orientation goes from the over crossing to the under crossing. We call this the \emph{chord diagram} for $D$ \cite{GPV}, \cite{VKT}. For example, the knot $3.1$ in Figure \ref{fig:3.1} has Gauss Code \[``O1-,O2-,U1-,O3+,U2-,U3+''\] and chord diagram as in Figure \ref{fig:3_1CD}.

\begin{figure}[h!]
\centering
    \includegraphics[height=3.5cm]{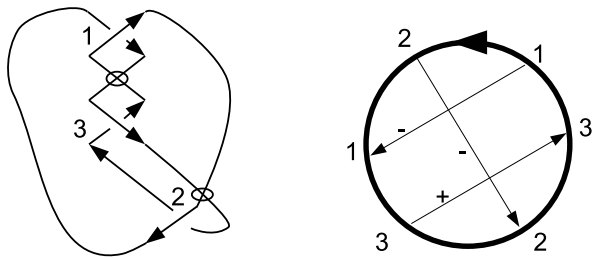}
\caption{Chord Diagram for Virtual Knot 3.1}
\label{fig:3_1CD}
\end{figure}

\begin{defn}
Given a diagram D for a knot K we can label each crossing as even or odd in the following manner. For each crossing $v$ locate the 2 occurrences of $v$ in the Gauss code for D. If the number of crossing labels between the two occurrences of $v$ is even then label the crossing even.  Else it is labeled odd.\\
\end{defn}

\begin{rem} This parity is well-defined for a 1-component links (i.e. knots) as the number of crossing labels in the Gauss code is $2n$ where $n$ is the number of crossings.
\end{rem}

It is important to notice how parity behaves under the classical Reidemeister moves. Note that virtual Reidemeister moves do not change the Gauss code or chord diagram and thus do not affect parity.

\begin{itemize}
	\item \textbf{Reidemeister I}\\
	A first Reidemeister move is always even, as is shown in Figure \ref{fig:CDR1}
	
\begin{figure}[H]
\centering
    \includegraphics[height=1.4cm]{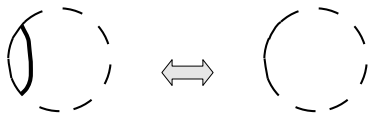}
\caption{Reidemeister I equivalence for flat chord diagram}
\label{fig:CDR1}
\end{figure}
	
	\item \textbf{Reidemeister II}\\
	The two crossings involved in a second Reidemeister move are either both even or both odd.  To see this, note that in Figure \ref{fig:CDR2}  if the number of crossings before the second Reidemeister move is $n+2$ and $a$ and $b$ denote the number of markings on the core circle as labeled in the figure then $a + b = 2n$ is even. Hence either $a$ and $b$ are both even or both odd.
	
	\begin{figure}[H]
\centering
    \includegraphics[height=3cm]{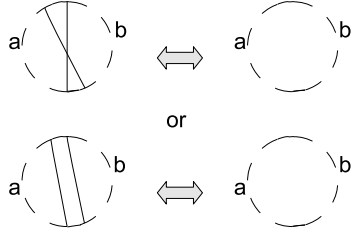}
\caption{Reidemeister II equivalence for flat chord diagram}
\label{fig:CDR2}
\end{figure}

\item \textbf{Reidemeister III}\\
	In a third Reidemeister move either all crossings are even or two are even and one is odd. To see this note that in Figure \ref{fig:CDR3}  if the number of crossings not involved in the third Reidemeister move is $n$ and $a, b$ and $c$ denote the number of markings on the core circle as labeled in the figure then $a + b + c= 2n$ is even. Hence either $a, b$ and $c$ are all even or two are odd and one is even.

\begin{figure}[H]
\centering
    \includegraphics[height=3.2cm]{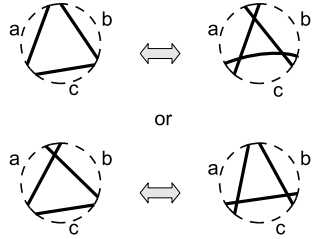}
\caption{Reidemeister III equivalence for flat chord diagram}
\label{fig:CDR3}
\end{figure}
\end{itemize}

\subsection{Manturov's Graphical Parity Bracket Polynomial}

Manturov \cite{ParityBracket} introduced the following graphical modification for the bracket polynomial.

\begin{defn} The \emph{parity bracket polynomial of a virtual knot K} is defined by the relations in Figure \ref{fig:ParityBracketReduce}.
\end{defn}

\begin{figure}[h!]
\centering
    \includegraphics[width=.9\textwidth]{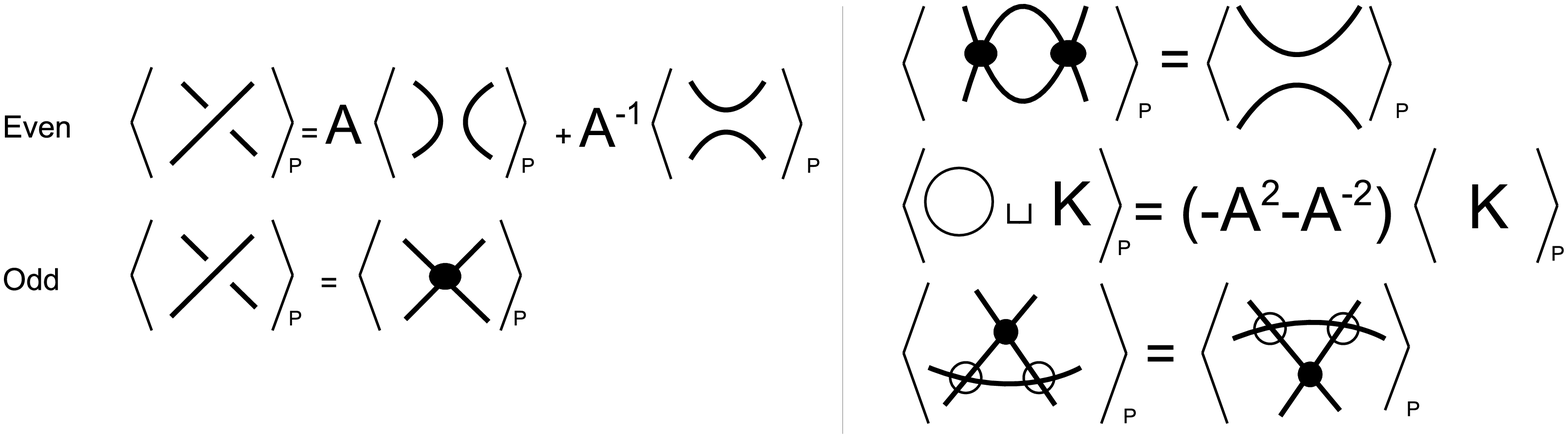}
\caption{Parity Bracket Polynomial Skein Relations}
\label{fig:ParityBracketReduce}
\end{figure}

\begin{defn}
Given a virtual knot K with diagram D, the \emph{normalized parity bracket polynomial of K} is given by
\[PF_{A}(K) = PF_{A}(D) = (-A)^{-3\omega(D)}  \langle D \rangle_{P}\]
where $\langle D \rangle_{P}$ is the parity bracket polynomial of D and
\[\omega(D) = \textrm{writhe}(D) = (\#\textrm{ positive crossings in D}) - (\#\textrm{ negative crossings in D}).\]
\end{defn}

\begin{thm} The parity bracket polynomial is an invariant of virtual knots.\\
\end{thm}

\begin{proof} We give an outline of the proof. The majority of this proof follows from Kauffman's proof of invariance for the bracket polynomial \cite{NewInvariants}.  You can find a similar proof by Manturov in \cite{ParityBracket}.
\begin{enumerate}
\item Reidemeister I follows from the writhe normalization.
\item Reidemeister II follows for even crossing as in the classical case and for odd crossing by the reduction relations.
\item Reidemeister III follows by applying the usual trick at a single even crossing. (Note in the mixed case there is only 1 even crossing to choose.) \qedhere
\end{enumerate}
\end{proof}

\begin{exa}
Figure \ref{fig:3.1GraphicalJones} displays the calculation of the parity bracket polynomial for virtual knot 3.1 in Figure \ref{fig:3.1}.\\
\end{exa}

\begin{figure}[h!]
\centering
    \includegraphics[width=.8\textwidth]{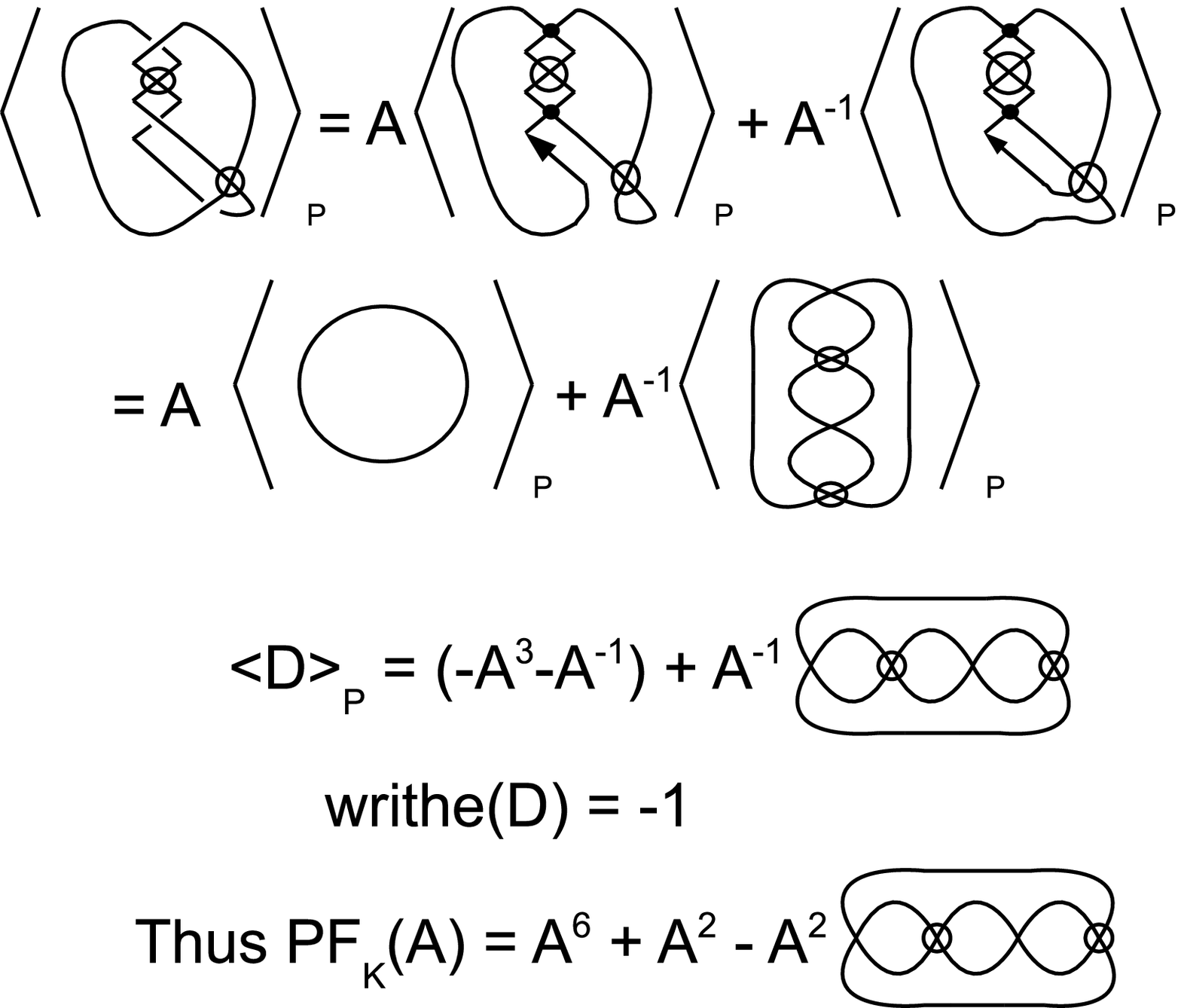}
\caption{Virtual Knot 3.1 Graphical Parity Bracket Polynomial}
\label{fig:3.1GraphicalJones}
\end{figure}

Notice that the parity bracket polynomial contains graphical coefficients. Since all of the remaining crossings are virtual or graphical (i.e. coming from an odd crossing) the only skein relations we may apply to the graphical coefficients are the graphical Reidemeister II move as well as the graphical detour move.\\

\subsubsection{The Parity Bracket Polynomial and Surface Genus}

\begin{defn} Given a graphical coefficient $D$ the minimal surface genus $s(D)$ is the minimal genus for an orientable surface $S_g$ such that there is an embedding $D \rightarrow S_g$. Given the product of graphical coefficients $D_{1}D_{2} \cdots D_{n}$  the surface genus $s(D_{1}D_{2} \cdots D_{n})$ is the minimal genus for an orientable surface $S_g$ such that there exist disjoint embedding $D_i \rightarrow S_g$ for $i \in \{1, \ldots,n \}$.\\
\end{defn}

\begin{rem} Note that in the case of the parity bracket polynomial the minimal surface genus for a graphical coefficient is the same as the minimal surface genus for the underlying flat virtual knot. Hence for a graphical coefficient $D$, $s(D) \geq s(K)$ where $K$ is any virtual knot arising from $D$ by resolving all crossing in any manner.
\end{rem}

\begin{thm} Given a knot K, the parity bracket polynomial gives a lower bound on the surface genus of K, $s(K)$. More precisely, if $PF_K(A)$ contains a monomial with graphical coefficients $D_{1}D_{2} \cdots D_{n}$ then
\[s(D_{1}D_{2} \cdots D_{n}) \leq s(K) \]
\end{thm}

\begin{proof} Suppose K is given by an embedding $S^1 \hookrightarrow S_g \times I$ and $s'$ is the state of the parity bracket polynomial corresponding to the term with graphical coefficients $D_{1}D_{2} \cdots D_{n}$. Projecting $s$ down onto $S_g \times \{0\}$, we see that the minimal surface genus of $s'$ is at least $s(D_{1}D_{2} \cdots D_{n})$.  Since the polynomial is an invariant of the knot this holds for every projection.
\end{proof}

\begin{exa} Let K be virtual knot 4.72 in Figure \ref{fig:4.72}. It is a short exercise to show $\langle K \rangle_A = 1$ and that \[\langle K \rangle_P = -A^4-A^2+\text{D}_{2}[1]-2-\frac{1}{A^2}-\frac{1}{A^4} \]
where $\text{D}_2[1] = \DTwoOne$. Note that $s(\text{D}_2[1]) = 1$ \cite{DKMinimalSurface} and hence $s(K) \geq 1$.  Note that the diagram in Figure \ref{fig:4.72} has genus 2.  It is not currently known to us if the minimal surface genus is 1 or 2.  Similarly the diagram in Figure \ref{fig:4.72} has virtual crossing number 3. It is unclear if there is a diagram for this knot with a lower virtual crossing number.
\end{exa}

\begin{figure}[h!]
\centering
    \includegraphics[height=3cm]{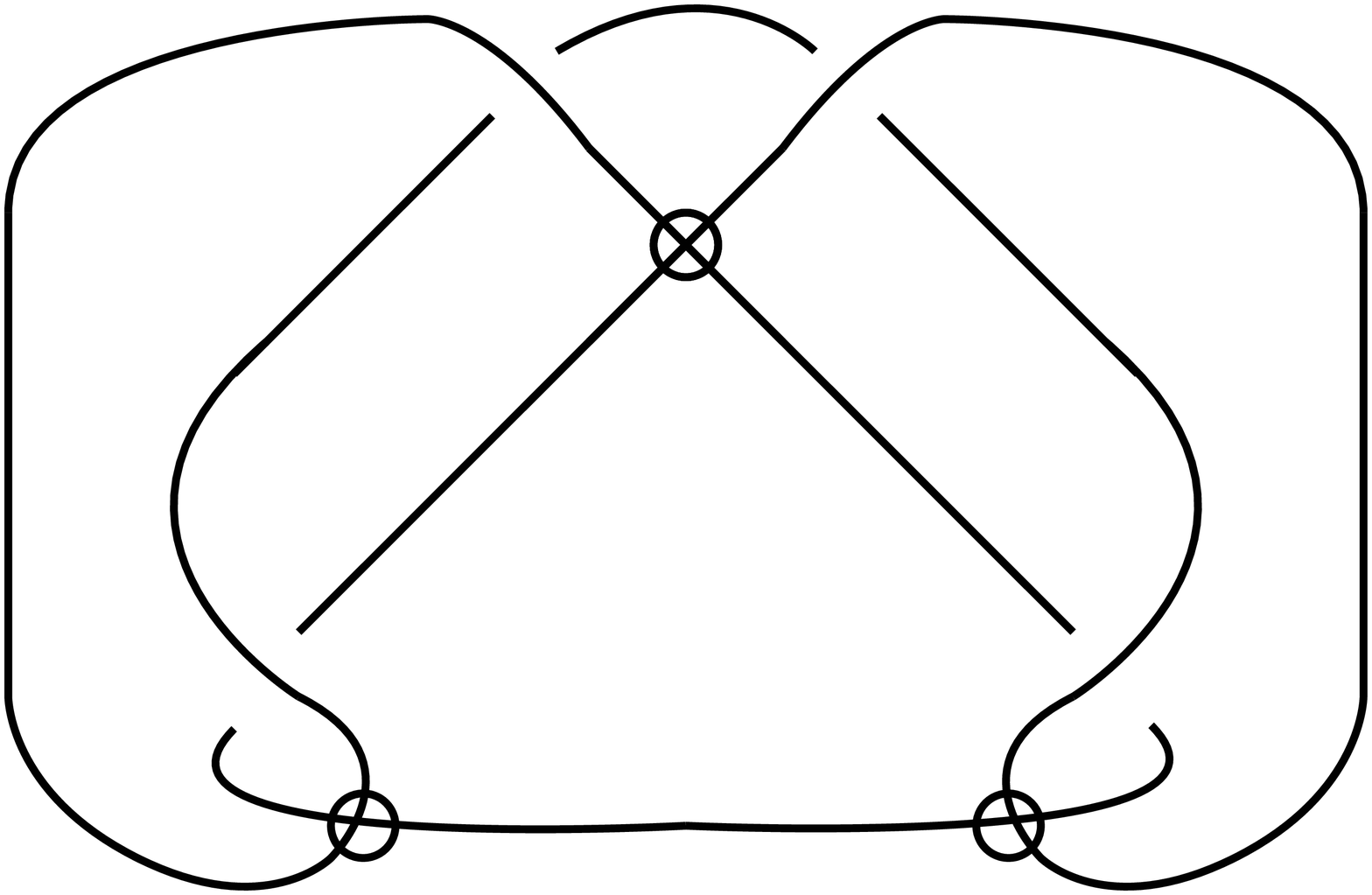}
\caption{Virtual Knot 4.72}
\label{fig:4.72}
\end{figure}

\subsection{A Parity Arrow Polynomial}

We use Manturov's idea of graphical coefficients \cite{ParityBracket} to extend the arrow polynomial via parity as follows.

\begin{defn} The \emph{(un-normalized) parity arrow polynomial of a virtual knot K} is defined by the relations in Figures \ref{fig:ParityArrowSmooth} and \ref{fig:ParityArrowReduce}. We expand as usual on the even crossing and make a graphical vertex for odd crossings.
\end{defn}

\begin{figure}[h!]
\centering
    \includegraphics[width=.6\textwidth]{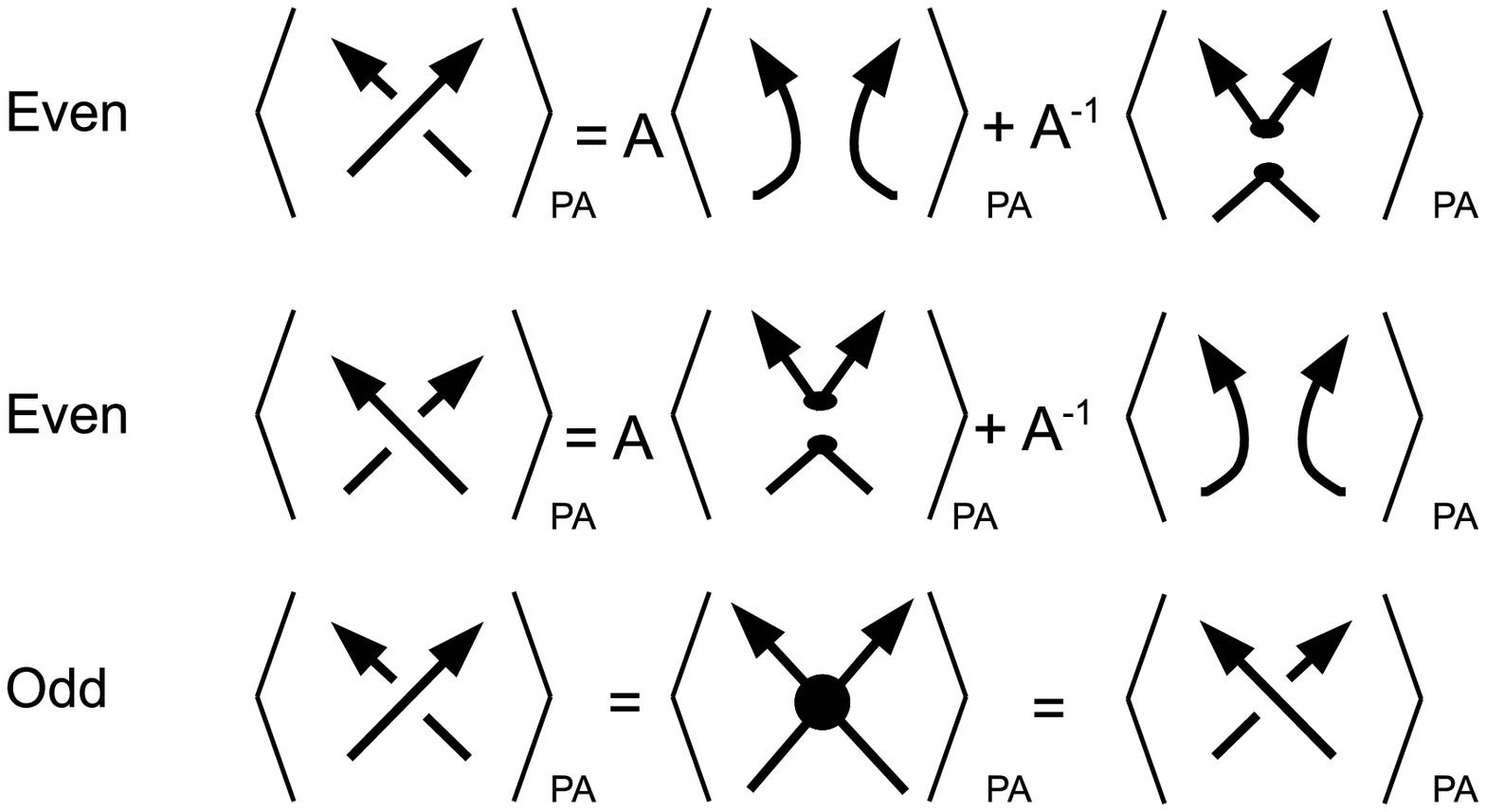}
\caption{Parity Arrow Polynomial Crossing Skein Relations}
\label{fig:ParityArrowSmooth}
\end{figure}

\begin{figure}[h!]
\centering
    \includegraphics[width=.9\textwidth]{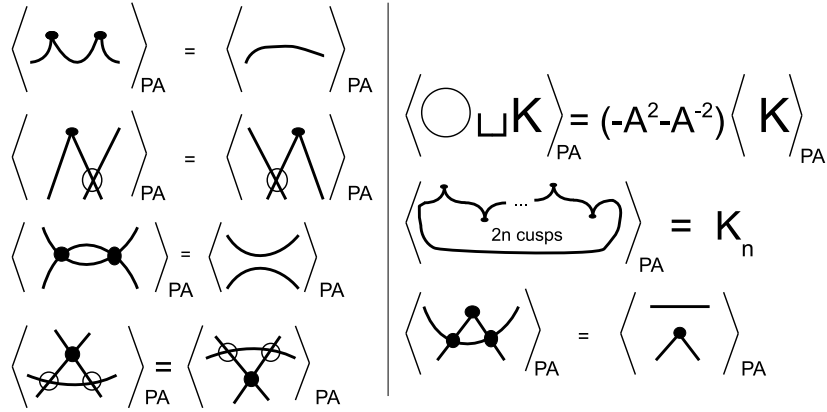}
\caption{Parity Arrow Polynomial Reduction Skein Relations}
\label{fig:ParityArrowReduce}
\end{figure}

\begin{defn} Given a virtual knot K with diagram D, The \emph{normalized parity arrow polynomial of K} is given by
\[PAP_{A}(K) = PAP_{A}(D) = (-A)^{-3\omega(D)}  \langle D \rangle_{PA}\]
where $\langle D \rangle_{PA}$ is the parity arrow polynomial of D and
\[\omega(D) = \textrm{writhe}(D) = (\#\textrm{ positive crossings in D}) - (\#\textrm{ negative crossings in D}).\]
\end{defn}

\begin{thm} The normalized parity arrow polynomial is an invariant of virtual knots.
\end{thm}

\begin{proof}As in the case of the parity bracket polynomial much of the proof is the same as in the non-parity version.
\begin{enumerate}
\item Reidemeister I follows from the writhe normalization.\\
\item Reidemeister II follows for even crossing by the equivalent proof for the arrow polynomial and for odd crossing by the reduction relations.\\
\item Reidemeister III follows by applying the usual trick at a single even crossing in conjunction with the cusped `Reidemeister II'-like relation. (Note in the mixed case there is only 1 even crossing to choose.)  See Figure \ref{fig:R3GraphicalArrow} for one of diagramatic proofs.  The others follow similarly. \qedhere \\
\end{enumerate} \end{proof}

\begin{figure}[h!]
\centering
\includegraphics[width=.7\textwidth]{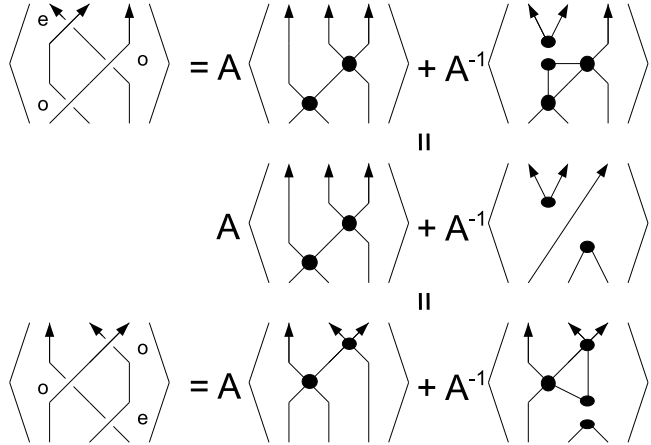}
\caption{Invariance for one of the mixed Reidemeister III move.}
\label{fig:R3GraphicalArrow}
\end{figure}

\begin{exa} The normalized parity arrow polynomial for virtual knot 4.70 in Figure \ref{fig:4.70} is equal
\[ \text{D}_{2}[3] A^8+2 \text{K}_{1} A^6-A^6-A^2\]
where $\text{D}_{2}[3]$ is the graphical coefficient \DTwoThree. We should note that the parity bracket polynomial of virtual knot 4.70 also contains a graphical coefficient.
\end{exa}

\begin{figure}[h!]
\centering
\includegraphics[width=.3\textwidth]{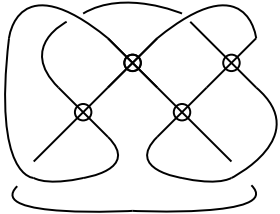}
\caption{Virtual Knot 4.70}
\label{fig:4.70}
\end{figure}


\begin{rem}
In \cite{ExtendedBracket} Kauffman introduced a polynomial called with extended bracket polynomial which is a generalization of the arrow polynomial. One key difference between these two polynomials is that the cusps in the extended bracket polynomial are maintained in associated pairs.  One can extend the parity arrow polynomial in a similar fashion by replacing the mixed `Reidemeister II'-like relation with a similar relation \extendedbracketRIII.
\end{rem}

\subsubsection{The Parity Arrow Polynomial and Surface Genus}

Similar to the parity bracket polynomial we have the following lower bound on the surface genus.\\

\begin{thm} Given a knot K, the parity bracket polynomial gives a lower bound on the surface genus of K, $s(K)$. More precisely, if $\langle D \rangle_{PA}$ contains a monomial with graphical coefficients $D_{1}D_{2} \cdots D_{n}$ then
\[s(D_{1}D_{2} \cdots D_{n}) \leq s(K) \]
\end{thm}

\begin{proof} This is identical to the proof for the parity bracket polynomial.
\end{proof}

\begin{exa} Virtual knot 5.5 in Figure \ref{fig:5.5} has  arrow polynomial
\[A^{10} \text{K1}^2-A^{10} \text{K1}-3 A^6 \text{K1}^2+3 A^6 \text{K1}+A^6-\frac{1}{A^6}-2 A^4 \text{K2}-3 A^2 \text{K1}+\frac{\text{K1}}{A^2}+2 A^2\]
and parity arrow polynomial
\[ -A^4 D_4[1] -A^6-2 A^2-\frac{1}{A^2}\]
Where $D_4[1] = \DFourOne$
Note that by \cite{DKVirtualCrossingNumber} the arrow polynomial gives a minimal surface genus, $s(K) \geq 1$.  However, in \cite{DKMinimalSurface} Dye and Kauffman show that a Kishino knot lying under $D_4[1]$ (in the sense of resolving graphical nodes into knot crossings) has surface genus 2.  Hence the parity arrow polynomial gives the lower bound $s(K) \geq 2$.
\end{exa}

\begin{figure}[h!]
\centering
\includegraphics[width=.5\textwidth]{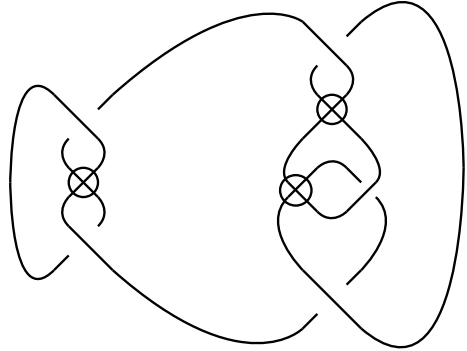}
\caption{Virtual Knot 5.5}
\label{fig:5.5}
\end{figure}

\subsection{Z-Equivalence and Graphical Parity Polynomials}

When computing the normalized bracket polynomial for virtual knots the relation of Z-Equivalence, depicted in Figure \ref{fig:ZEquiv}, goes undetected.  Passing to the parity bracket polynomial we have the option to add the corresponding graphical relation to the coefficients (i.e. including the relation in Figure \ref{fig:ZEquiv} when the classical crossings are replaced by graphical nodes.) We call the resulting polynomial the \emph{z-parity bracket polynomial}

\begin{figure}[h!]
\centering
    \includegraphics[height=1.2cm]{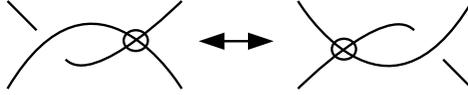}
\caption{Z-Equivalence and the Bracket Polynomial}
\label{fig:ZEquiv}
\end{figure}

 Going a step further, we can choose to completely ignore the graphical coefficients by sending the graphical nodes to virtual crossings. This leads to the following \emph{forgetful parity bracket polynomial}.

\begin{lem} The parity bracket polynomial is strictly stronger than the z-parity bracket polynomial which in turn is strictly stronger than the forgetful parity bracket polynomial.
\end{lem}

\begin{proof}For the parity bracket polynomial and z-parity bracket polynomial the follows immediately from the definition.  For the forgetful parity bracket polynomial and z-parity bracket polynomial, notice that the action of swapping an odd crossing for a virtual crossing in the Z-Equivalence relation yields an identity.
\end{proof}

\subsection{Manturov's Parity Filtration}

In \cite{ParityBracket} Manturov introduced the following descending filtration on the category of virtual knots. Given a virtual knot K with chord diagram D.  Let $D_0 = D$ be the equivalence class of $D$ up to Reidemeister moves and define $D_{i+1}$ to be the equivalence class of $D_i$ after removing all odd chords (or equivalently, turning odd crossings into virtual crossings.)\\

This is well-defined since parity is invariant under the Reidemeister moves as shown above, most notably this filtration does not introduce the forbidden move from Figure \ref{fig:Forbidden} when applied to the mixed-parity Reidemeister III move.  Hence for 2 representatives $D_i, D_{i}' \in [D_i] :=\text{Equivalence Class of }D_i$ we have $[D_{i+1}] \equiv [D_{i+1}']$.\\

Diagrammatically this filtration is described by the map sending odd crossing to virtual crossing and hence the forgetful parity polynomials are precisely an application of the respective polynomial on the filtration.\\

\begin{thm}

\begin{enumerate}
  \item For any virtual knot this filtration is finite. (i.e. there exists $n$ such that $D_n = D_i$ for $n \leq i$.)
  \item For any classical knot this filtration is of the form $D_0 = D_1 = \ldots$, that is all levels of the filtration are identical.
\end{enumerate}
\end{thm}

\begin{proof}

\begin{enumerate}
  \item This follows from the finiteness of crossings.
  \item Every classical knot is equivalent to a knot with all even crossings. \qedhere
\end{enumerate}
\end{proof}

It is not difficult to construct a family of virtual knots for which given an $n$, $D_n = D_i$ for $n \leq i$.  Consider the family in Figure \ref{fig:filtrationfamily}.  Here the knot $F_1$ is the 2-crossing virtual knot (hence not equivalent to the unknot.) Each additional member of the family is produced from its predecessor by the addition of two odd positive crossing arcs and which become the only odd arcs in the new diagram. Since we are unable to cancel the two new odd crossings with one another we see each new virtual knot created in this way is unique from its predecessor. Moreover for $F_n$ it is easy to verify that for $D_n = D_i =\text{Unknot}$ for $n \leq i$.\\

\begin{figure}[h!]
\centering
    \includegraphics[width=.9\textwidth]{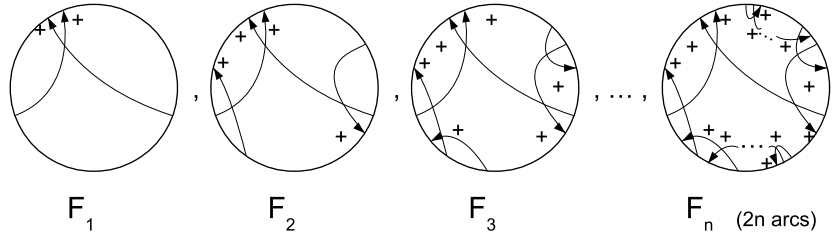}
\caption{A Parity Filtration Family}
\label{fig:filtrationfamily}
\end{figure}

\subsection{Extending the Parity Polynomials to Links}\label{sec:LinkParityPolys}
For this subsection we will take a wider view and consider the space of virtual links (2 or more components). One should note that our definition of even and odd parity does not naturally extend. For example, the links in Figure \ref{fig:paritylinks1} illustrate some of the difficulty in the natural extension.\\

\begin{figure}[h!]
\centering
    \includegraphics[width=0.5\textwidth]{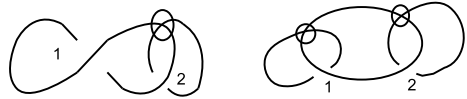}
\caption{Examples of Difficulty in Extending the Definition to Links}
\label{fig:paritylinks1}
\end{figure}

Omitting signs, the left link in \ref{fig:paritylinks1} has Gauss code ``$O1,U1,O2; U2$'' while the other has Gauss code ``$U1; O1,O2; U2$''. In the first of these Crossing 1 is both even and odd in the first component while Crossing 1 is either even or odd depending upon whether you examine the first or second link component.\\

We may circumvent this pitfall by defining even and odd for self-crossings based on the parity of self-crossing in each component while labeling crossings shared by 2 components as link crossings.\\

\begin{figure}[h!]
\centering
    \includegraphics[height=3.3cm]{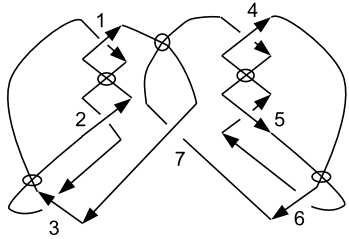}
\caption{Link Parity Crossing Labelings}
\label{fig:LinkParityExample}
\end{figure}

\begin{exa}The link in Figure \ref{fig:LinkParityExample} has Gauss code \[``O1,O7,O3,U1,U2,U3,O2; U4,O5,U6,U5,O4,O6,U7'' \]  Crossings 1, 2, 4 and 5 are odd, crossings 3, and 6 are even and crossing 7 is a link crossing.
\end{exa}

As we did with odd crossings, we investigate the invariance of crossings between links to provide the framework for generalizing the parity polynomials. We will call a crossing where both arcs involved are in the same link component a self-crossing while a crossing whose arcs are in separate components a link-crossing.\\

\begin{itemize}
	\item \textbf{Reidemeister I}\\
	In a Reidemeister I move only a single link component is involved hence is always an even self-crossing.\\
	
	\item \textbf{Reidemeister II}\\
	The two arcs involved in a second Reidemeister move are either both in the same component or each is in a different component. Thus either both crossings above are self-crossings or both crossings are link-crossings.\\

\item \textbf{Reidemeister III}\\
	In a third Reidemeister move either all strands involved are in one component, or two in one component and one in another or all three in separate components. Thus either all crossings are self-crossings, or one self-crossing and two link-crossings or three link-crossings respectively.\\
\end{itemize}

This motivates the following definitions for link parity polynomials.\\

\begin{defn}Given a diagram D for a virtual knot K the \emph{(link) parity bracket polynomial of K} is defined by the relations in Figures \ref{fig:LinkParityBracketSmooth} and \ref{fig:LinkParityBracketReduce}.\\
\end{defn}

\begin{figure}[h!]
\centering
    \includegraphics[width=0.7\textwidth]{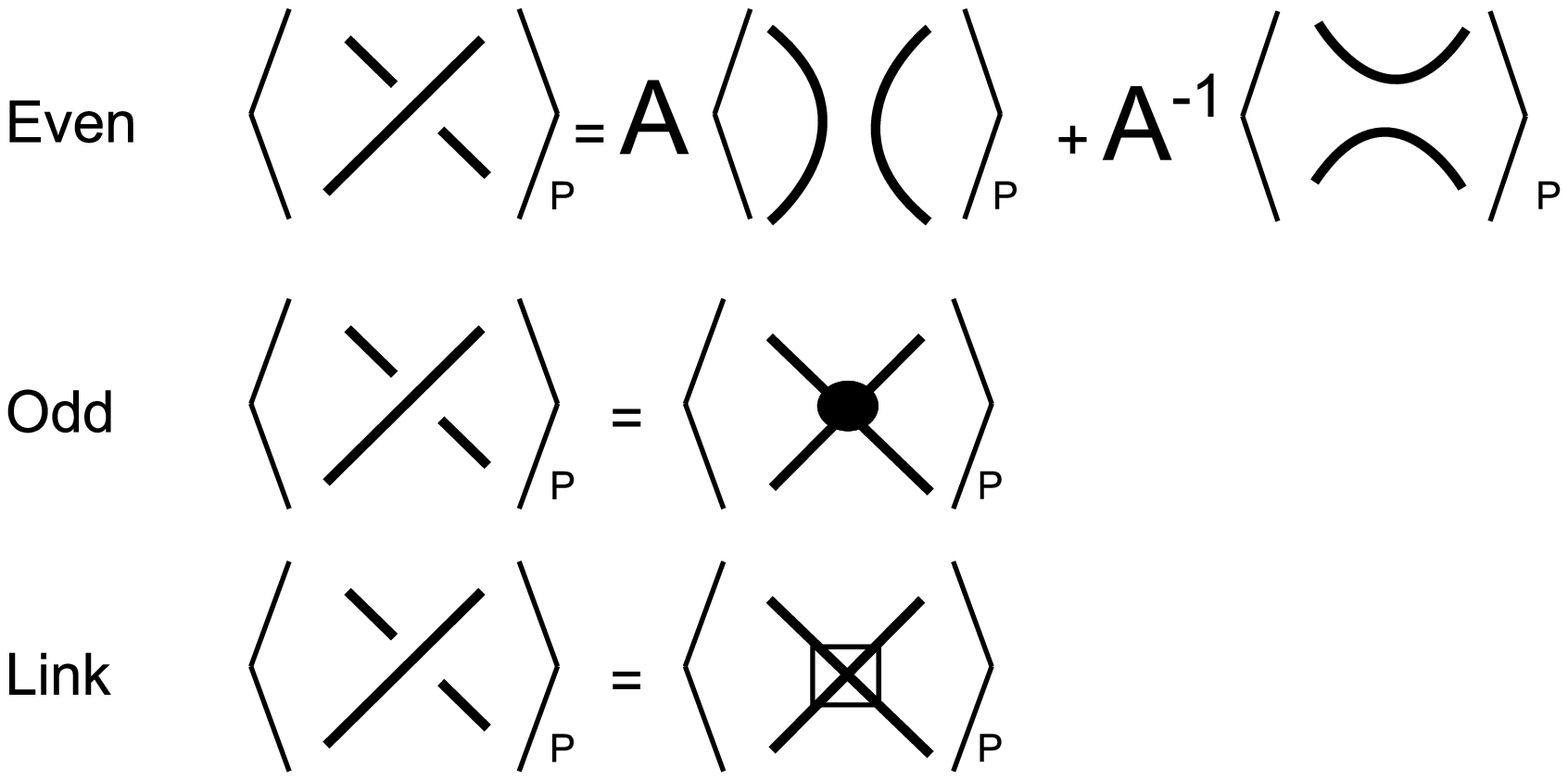}
\caption{Link Parity Bracket Polynomial Smoothing Relations}
\label{fig:LinkParityBracketSmooth}
\end{figure}

\begin{figure}[h!]
\centering
    \includegraphics[width=0.9\textwidth]{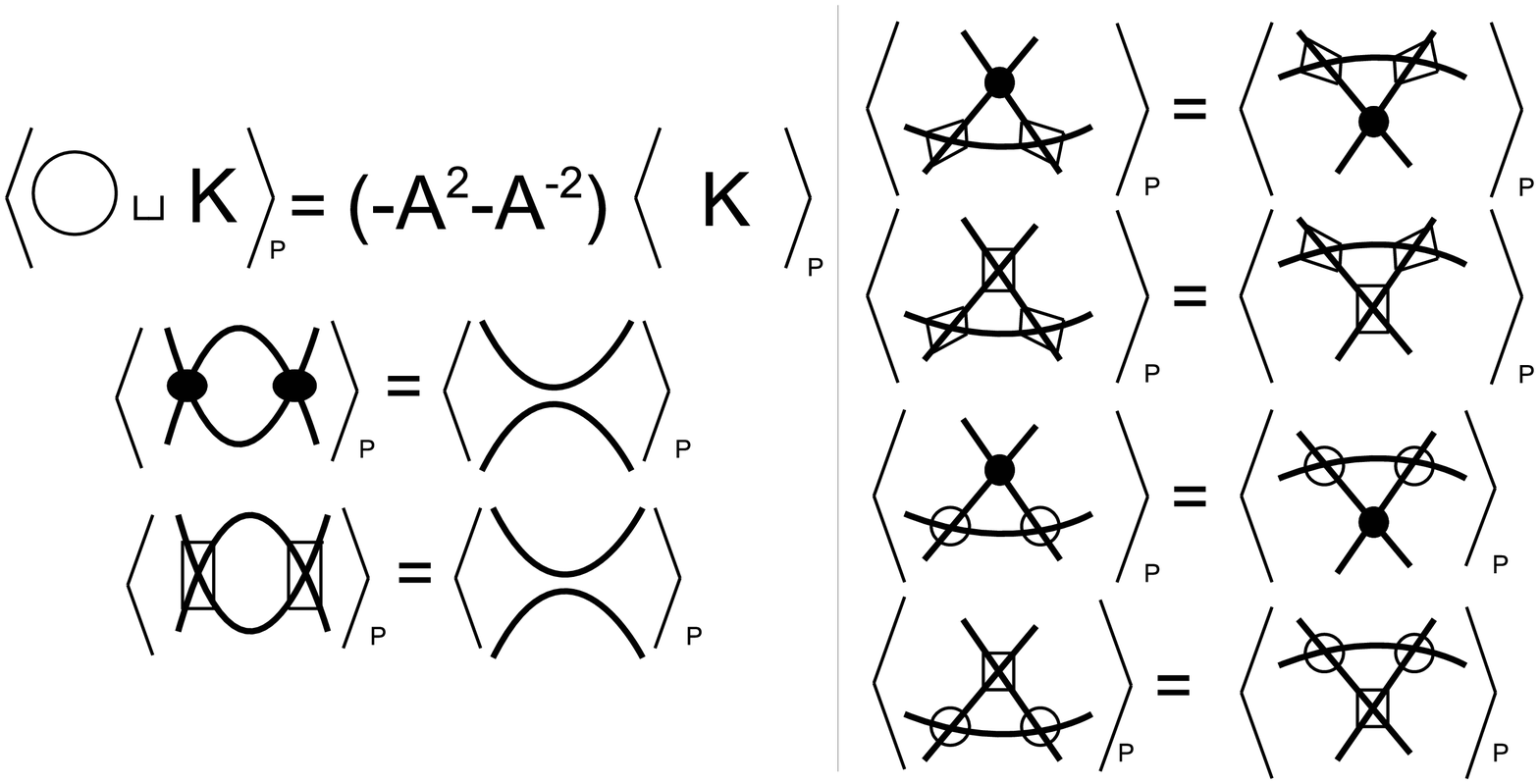}
\caption{Link Parity Bracket Polynomial Reduction Relations}
\label{fig:LinkParityBracketReduce}
\end{figure}

\begin{defn}
Given a virtual link K with diagram D, The \emph{normalized (link) parity bracket polynomial of K} is given by
\[PF_{A}(K) = PF_{A}(D) = (-A)^{-3\omega(D)}  \langle D \rangle_{P}\]
 where $\langle D \rangle_{P}$ is the parity bracket polynomial of D and
 \[\omega(D) = \textrm{writhe}(D) = (\#\textrm{ positive crossings in D}) - (\#\textrm{ negative crossings in D}).\]
\end{defn}

\begin{thm} The (link) parity bracket polynomial is an invariant of virtual knots.\\
\end{thm}

\begin{proof}

\begin{enumerate}
\item \textbf{RI: } A crossing involved in a Reidemeister I move is always an even self-crossing, hence invariance follows from the writhe normalization as in the normalized bracket polynomial.
\item \textbf{RII: } Reidemeister II follows for even self-crossing as in the classical case and for odd self-crossing and link-crossings by the reduction relations.
\item \textbf{RIII: } Here we have three cases to consider:
\begin{enumerate}
  \item Reidemeister III for three self-crossings follows by applying the usual trick at a single even crossing.
  \item For a Reidemeister III involving one self-crossings and two link crossings, if the self-crossing is even then the usual trick with the RII-like move for link crossings gives invariance.  If the self-crossing is odd then the result is immediate by the reduction relations.
  \item Reidemeister II for three link-crossings also follows immediately from the reduction relations.
\end{enumerate}
\item \textbf{Mixed Move: } For an even self-crossing this is the standard proof and for an odd self-crossing or link-crossing it follows from the reduction relations.
\end{enumerate} \end{proof}

\begin{exa}
Figure \ref{fig:ParityLinkBracket} displays the calculation for the (link) parity bracket polynomial for the given 2-component link.
\end{exa}

\begin{figure}[h!]
\centering
    \includegraphics[width=0.7\textwidth]{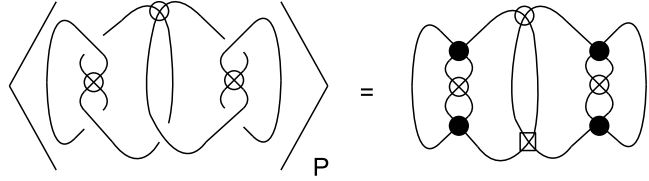}
\caption{Graphical Link Parity Bracket Polynomial Example}
\label{fig:ParityLinkBracket}
\end{figure}

Similarly for the parity arrow polynomial we have:\\

\begin{defn}Given a diagram D for a virtual link K the \emph{(link) parity arrow polynomial of K} is defined by the relations in Figures \ref{fig:LinkParityBracketSmooth} and \ref{fig:LinkParityBracketReduce}.\\
\end{defn}

\begin{figure}[h!]
\centering
    \includegraphics[width=0.9\textwidth]{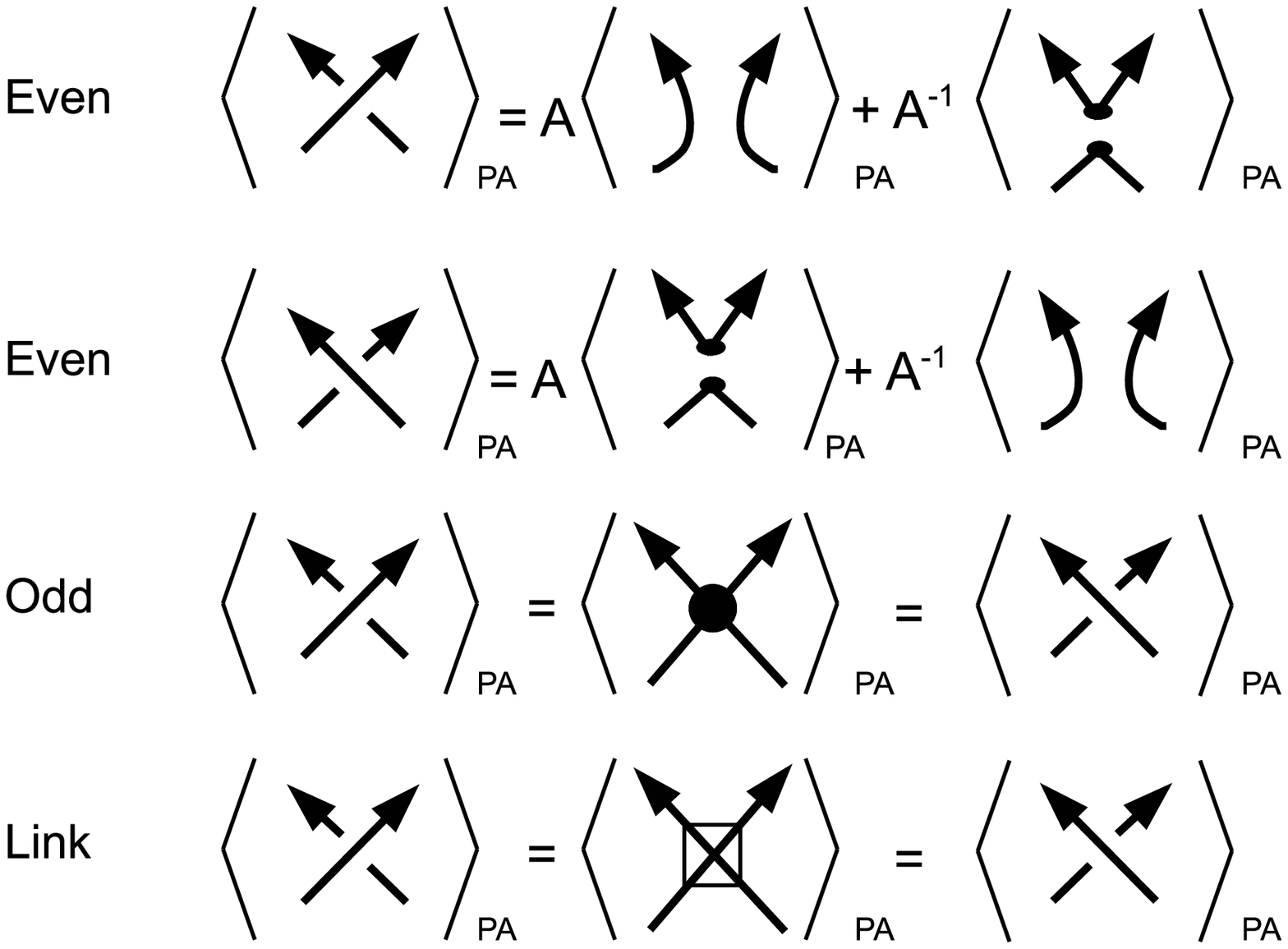}
\caption{Link Parity Arrow Polynomial Smoothing Relations}
\label{fig:LinkParityArrowSmooth}
\end{figure}

\begin{figure}[h!]
\centering
    \includegraphics[width=0.9\textwidth]{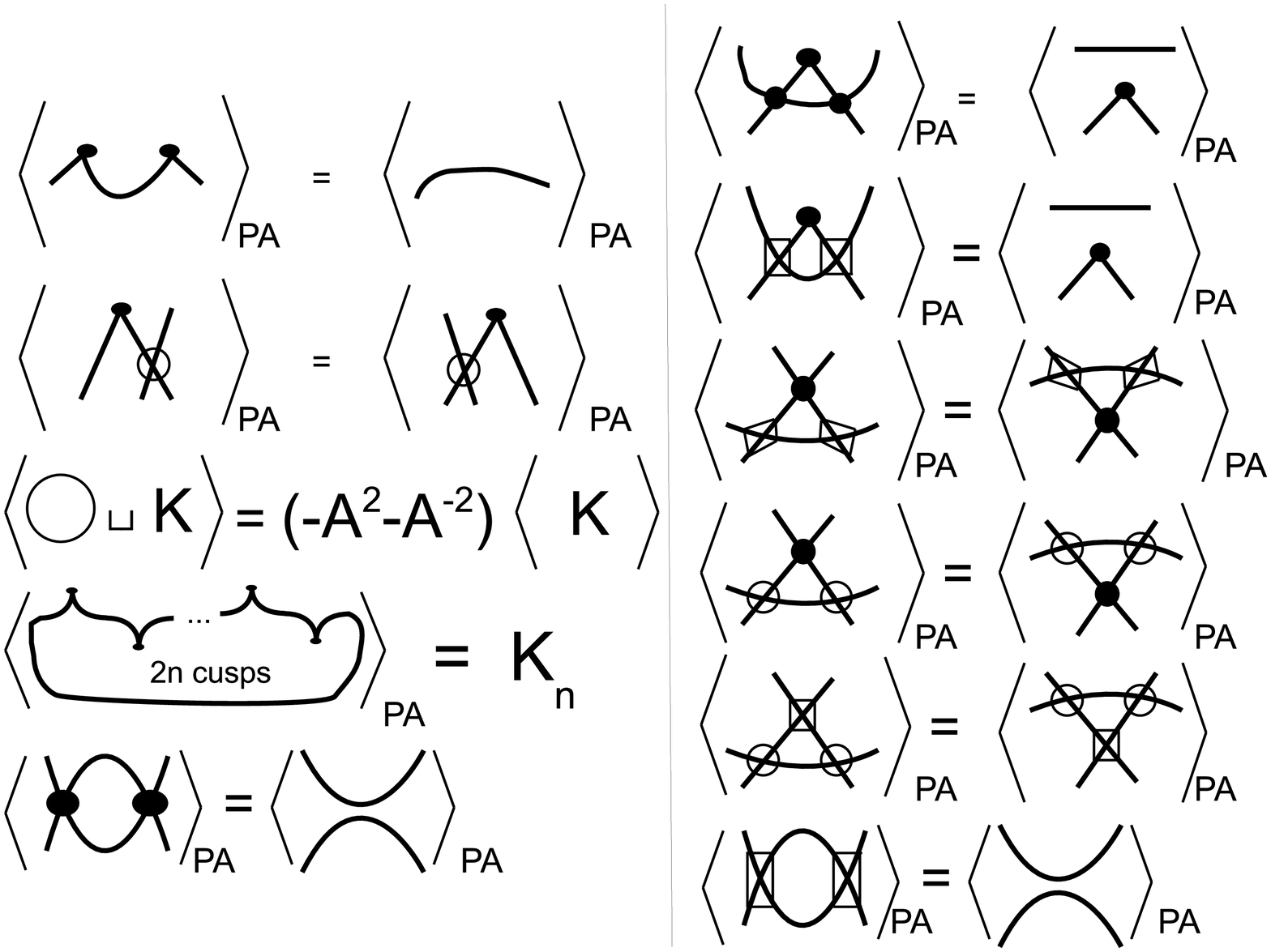}
\caption{Link Parity Arrow Polynomial Reduction Relations}
\label{fig:LinkParityArrowReduce}
\end{figure}

\begin{defn} Given a virtual link K with diagram D, The \emph{normalized (link) parity arrow polynomial of K} is given by
\[PAP_{A}(K) = PAP_{A}(D) = (-A)^{-3\omega(D)}  \langle D \rangle_{PA}\]
 where $\langle D \rangle_{PA}$ is the parity arrow polynomial of D and
 \[\omega(D) = \textrm{writhe}(D) = (\#\textrm{ positive crossings in D}) - (\#\textrm{ negative crossings in D}).\]
\end{defn}

\begin{thm} The (link) parity arrow polynomial is an invariant of virtual knots.\\
\end{thm}

\begin{proof} This proof is nearly identical to the one above. The only difference is the proof for the Reidemeister III move involving an even self-crossing and two link crossings.  Here we use a cusped RII-like move for link crossing in the usual trick for RIII invariance.
\end{proof}

\begin{rem}We have used a similar extension to links in \cite{ParityBiquandles} to generalize the construction of Parity Biquandles.
\end{rem}
\section{Parity and Categorifications}

We would like to categorify the parity polynomials in an analogous manner to the original polynomials by adding an additional grading, similarly to the construction of the categorification of the arrow polynomial, based on the equivalence classes of graphified flat knot diagrams. However, it is fairly simple to construct an example showing this to be naive.  For instance, consider the virtual knot in Figure \ref{fig:KishinoR2}.\\
\begin{figure}[h!]
\centering
    \includegraphics[height=1.8cm]{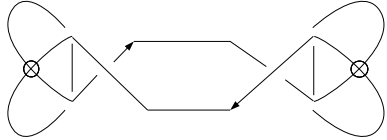}
\caption{Kishino Knot}
\label{fig:KishinoR2}
\end{figure}

Performing the available Reidemeister II move, the resulting diagram is one of three virtual knots often referred to as Kishino knots. Figure \ref{fig:KishinoR2Bracket} shows that both the parity bracket polynomial and parity arrow polynomial of the knot is the graphified version of the diagram as there are no graphical Reidemeister II moves  or detour moves available.\\

\begin{figure}[h!]
\centering
    \includegraphics[height=1.8cm]{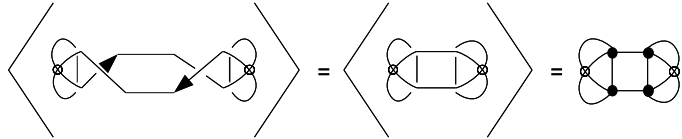}
\caption{Graphical Parity Bracket Polynomial of a Kishino Knot}
\label{fig:KishinoR2Bracket}
\end{figure}

However, when we consider the Khovanov complex (the arrow polynomial categorifications have equivalent complexes) as in Figure \ref{fig:KishinoParityProblem} we can see that $d^2 \neq 0$.  In particular, the all-A and all-$A^{-1}$ states are both graphically equivalent to two circles while in the middle we have the top state graphically equivalent to a graphified Kishino knot and the bottom state graphically equivalent to 3 circles. Hence, as shown in the figure, the upper differentials are both the 0-map.  Considering the element $(x \otimes 1)$, we see
\[
d^2(x \otimes 1) = (1 \otimes m) \circ (\Delta \otimes 1)(x \otimes 1) = (1 \otimes m)(x \otimes x \otimes 1) = x \otimes x \neq 0
\]

\begin{figure}[h!]
\centering
    \includegraphics[height=3.8cm]{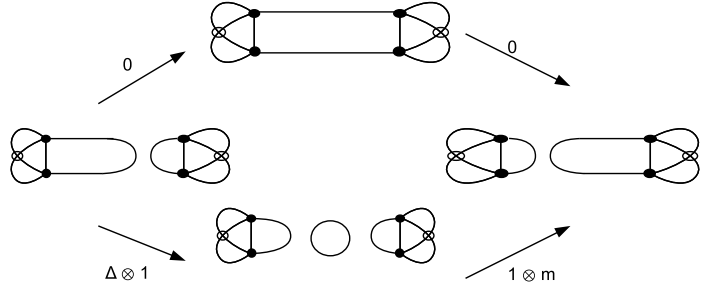}
\caption{Graphical Parity Kishino Complex}
\label{fig:KishinoParityProblem}
\end{figure}

However, this does not prevent us from applying Manturov's parity filtration along with the forgetful version of the parity categorifications.  In doing so we clearly lose some of the power of the parity polynomial (for instance the Kishino knot is no longer detected) but we still retain an invariant which is capable of detecting non-classicality. We can define the \emph{parity Khovanov homology} (respectively, \emph{parity arrow categorification})  to be the homology theory produced by first applying Manturov's parity filtration to the given knot and then computing the Khovanov homology (respectively, the arrow categorification) of the resulting knot.  Similarly, define the \emph{parity Khovanov invariant} (respectively, \emph{parity arrow invariant}) to be the resulting Poincar\'{e} polynomial as produced previously.

For instance, consider the knot in Figure \ref{fig:FiltrationExample}.  Applying the parity Khovanov homology we have that the original knot has Khovanov invariant
\[\frac{1}{q^9 t^3}+\frac{1}{q^8 t^2}+\frac{1}{q^7 t^3}+\frac{1}{q^7 t^2}+\frac{1}{q^6 t^2}+\frac{1}{q^6 t}+\frac{1}{q^5 t^2}+\frac{1}{q^5}+\frac{1}{q^4
   t}+\frac{1}{q^3}\]

Applying the filtration and turning crossings 1 and 4 into virtual crossings we have that the underlying knot at this level of the filtration is the two crossing virtual knot. Hence virtual knot 4.9 has parity Khovanov invariant
\[\frac{1}{q^6 t^2}+\frac{1}{q^4 t^2}+\frac{1}{q^4 t}+\frac{1}{q^3}+\frac{1}{q^2 t}+\frac{1}{q}\]
and moreover is non-classical.\\

\begin{figure}[h!]
\centering
    \includegraphics[width=.5\textwidth]{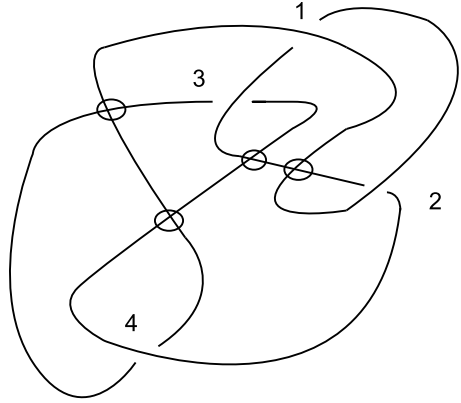}
\caption{Virtual Knot 4.9}
\label{fig:FiltrationExample}
\end{figure}

Using Jeremy Green's virtual knot table \cite{GreenTables} we have been able to calculate the parity categorifications on knots with at most 6 real crossings.  Table \ref{tab:distinguishedbyparity} is a collection of planar diagram codes for 8 knots which are not distinguished from the unknot via the bracket polynomial, arrow polynomial or their categorifications, but are distinguished from the unknot via the parity arrow categorification. Please see the Appendix for an explanation of our planar diagram conventions. The top four knots have parity arrow invariant
\[\frac{\text{vg}(2,1) K[2]}{q^3 t}+\frac{\text{vg}(1,2) K[1]}{q^3}+\frac{\text{vg}(2,-1) K[2]}{q t}+q \text{vg}(1,-2) K[1]+\frac{2 K[1]}{q}\]
while the lower four have parity arrow invariant
\[q^3 t \text{vg}(2,-1) K[2]+q^3 \text{vg}(1,-2) K[1]+q t \text{vg}(2,1) K[2]+\frac{\text{vg}(1,2) K[1]}{q}+2 q K[1] \]

{\small
\begin{table} [h!]
\begin{center}
\begin{tabular}{| l | l l l l | }
\hline
6.5508  & \multicolumn{4}{l|}{PD[X[4, 2, 5, 1], X[7, 4, 8, 3], X[10, 6, 11, 5], Y[12, 3, 1, 2],}\\
 & \multicolumn{4}{r|}{Y[9, 7, 10, 6],   Y[8, 12, 9, 11]]}  \\ \hline
6.5627  & \multicolumn{4}{l|}{PD[X[4, 2, 5, 1], X[7, 4, 8, 3], X[10, 6, 11, 5], Y[12, 3, 1, 2],}\\
 & \multicolumn{4}{r|}{Y[9, 7, 10, 6], Y[11, 9, 12, 8]]} \\ \hline
6.7613  & \multicolumn{4}{l|}{PD[X[4, 2, 5, 1], X[7, 4, 8, 3], X[9, 7, 10, 6], Y[12, 3, 1, 2],}\\
 & \multicolumn{4}{r|}{Y[10, 6, 11, 5], Y[8, 12, 9, 11]]}  \\ \hline
6.7701  & \multicolumn{4}{l|}{PD[X[4, 2, 5, 1], X[7, 4, 8, 3], X[9, 7, 10, 6], Y[12, 3, 1, 2],}\\
 & \multicolumn{4}{r|}{Y[10, 6, 11, 5], Y[11, 9, 12, 8]]}  \\ \hline
6.24828  & \multicolumn{4}{l|}{PD[X[6, 4, 7, 3], X[9, 5, 10, 4], X[11, 8, 12, 7], Y[12, 3, 1, 2],}\\
 & \multicolumn{4}{r|}{Y[1, 11, 2, 10], Y[8, 6, 9, 5]]}  \\ \hline
6.37012  & \multicolumn{4}{l|}{PD[X[6, 4, 7, 3], X[8, 6, 9, 5],  X[11, 8, 12, 7], Y[12, 3, 1, 2],}\\
 & \multicolumn{4}{r|}{Y[1, 11, 2, 10], Y[9, 5, 10, 4]]} \\ \hline
6.60677  & \multicolumn{4}{l|}{PD[X[3, 7, 4, 6], X[9, 5, 10, 4], X[11, 8, 12, 7], Y[12, 3, 1, 2],}\\
 & \multicolumn{4}{r|}{Y[1, 11, 2, 10], Y[8, 6, 9, 5]]}  \\ \hline
6.65816  &  \multicolumn{4}{l|}{PD[X[3, 7, 4, 6], X[8, 6, 9, 5], X[11, 8, 12, 7], Y[12, 3, 1, 2],}\\
 & \multicolumn{4}{r|}{Y[1, 11, 2, 10], Y[9, 5, 10, 4]]}  \\ \hline
\end{tabular}
\end{center}
\caption{Undistinguished from Unknot by Categorification but Distinguished by Parity}\label{tab:distinguishedbyparity}
\end{table}
}

\subsection{Link Parity Polynomials and Categorification}
Following the construction presented in Section \ref{sec:LinkParityPolys} we can extend the graphical polynomials to links.  However, the graphical coefficients for links suffer a similar problem to that of knots when categorified.  As with knots we map use the forgetful map to send the graphical link coefficients to virtual crossings.  The effect of this for links is rather unfortunate as it reduces a link to the disjoint union of its components. (This is easiest to see by thinking of the chord diagram.) Hence it reduces the link parity categorification back to the knot parity categorification setting.\\

\subsection{Intriguing Examples}
Using Jeremy Green's tables \cite{GreenTables} we have calculated the above invariants as along with the Sawollek polynomial and z-parity Sawollek polynomial(\cite{ParityBiquandles}) for knots with at most 6 real crossings.  The knots in Figure \ref{fig:6_32008} and Figure \ref{fig:6_73583} are special in that they are not distinguished from the unknot via any of the invariants. Knot $6.32008$ has 4 odd crossings while Knot $6.73583$ has no odd crossings and both knots are trivial as flats. Using a 2-cable Jones polynomial calculator adapted from Dror Bar-Natan's ``faster'' Jones polynomial Calculator \cite{KnotAtlas} we have been able to distinguish each of these knots from one-another and from the unknot.\\

\begin{figure}[h!]
\centering
    \includegraphics[height=3.3cm]{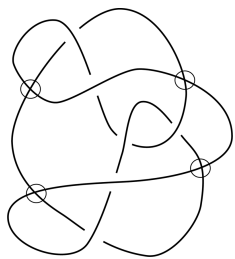}
\caption{Knot 6.32008}
\label{fig:6_32008}
\end{figure}

\begin{figure}[h!]
\centering
    \includegraphics[height=3.3cm]{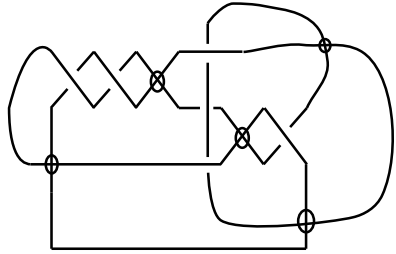}
\caption{Knot 6.73583}
\label{fig:6_73583}
\end{figure}

\appendix
\section{Appendix}

\subsection{Computational Results}
\subsubsection{Planar Diagram Conventions}

The following examples and programs compute the associated invariant based on a planar diagram code for a diagram for the given knot whose arcs have been consecutively labeled.  Our conventions are listed in Figure \ref{fig:PDConventions}.\\

\begin{figure}[h!]
\centering
    \includegraphics[height=1.7cm]{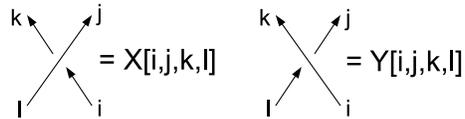}
\caption{Planar Diagram Code Conventions}
\label{fig:PDConventions}
\end{figure}

\begin{exa} Virtual knot 3.1 as labeled in Figure \ref{fig:PDConventionsEx} has planar diagram code:
\[
PD[X[1,5,2,4],X[5,4,6,3],Y[6,3,1,2]]
\]
\end{exa}

\begin{figure}[h!]
\centering
    \includegraphics[height=3.3cm]{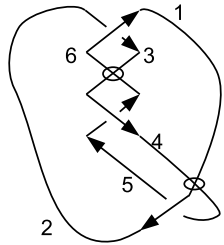}
\caption{Virtual Knot 3.1}
\label{fig:PDConventionsEx}
\end{figure}

\subsubsection{Parity Polynomials}

The following table displays the parity bracket polynomial and parity arrow polynomial for virtual knots with at most four crossings. Naming conventions are as in \cite{GreenTables}. Graphical coefficients are labeled by $D_2 \left[ 1 \right], D_2 \left[ 2 \right], D_2 \left[ 3 \right] $ and $D_4 \left[ 1 \right], \ldots, D_4 \left[ 4 \right] $ corresponding to the following diagrams.\\

\[
\begin{array}{cc}
   D_2 \left[ 1 \right] = \DTwoOne   &  D_2 \left[2 \right] = \DTwoTwo   \\
\end{array}
\]

\[
\begin{array}{c}
  D_2 \left[3 \right] = \DTwoThree   \\
\end{array}
\]

\[
\begin{array}{cc}
  D_4 \left[1 \right] = \DFourOne    &   D_4 \left[2 \right] = \DFourTwo   \\
   & \\
  D_4 \left[3 \right] = \DFourThree   &  D_4 \left[4 \right] = \DFourFour   \\
\end{array}
\]

It has been shown by Dye and Kauffman (Theorem 4.1 of \cite{DKMinimalSurface}) that graphical coefficients $D_4 \left[1 \right]$ and $D_4 \left[2 \right]$  have surface genus $s(D_4 \left[1 \right] )=2$ and $s(D_4 \left[2 \right] )=2$. Using this fact we can see that the parity polynomials are able to give a better bound on the genus than that of the arrow polynomial \cite{DKVirtualCrossingNumber} for certain knots. In particular we have $s(K) \geq 2$ for $K \in \{4.1, 4.2, 4.4, 4.5, 4.7, 4.8, 4.55, 4.56, 4.76, 4.77 \}$.

\newpage

{\small
\begin{longtable}{|l|p{5.2cm}|p{5.2cm}|}
\caption[Parity Bracket and Parity Arrow Polynomial Calculations]{Parity Bracket and Parity Arrow Polynomial Calculations} \\ \hline
Knot & Parity Bracket & Parity Arrow \\ \hline
\endfirsthead
\hline
Knot & Parity Bracket & Parity Arrow \\ \hline
\endhead
$2.1 $ & $ -A^2-\frac{1}{A^2} $ & $ -A^2-\frac{1}{A^2} $ \\ \hline
$3.1 $ & $ A^6-\text{D}_{2}[1] A^2+A^2 $ & $ A^6-\text{D}_{2}[1] A^2+A^2$ \\ \hline
$3.2 $ & $ -A^2-\frac{1}{A^2} $ & $ -A^2-\frac{1}{A^2} $ \\ \hline
$3.3 $ & $ A^6-\text{D}_{2}[1] A^2+A^2 $ & $ A^6-\text{D}_{2}[1] A^2+A^2 $ \\ \hline
$3.4 $ & $ -\frac{\text{D}_{2}[1]}{A^2}+\frac{1}{A^2}+\frac{1}{A^6} $ & $ -\frac{\text{D}_{2}[1]}{A^2}+\frac{1}{A^2}+\frac{1}{A^6} $ \\ \hline
$3.5 $ & $ A^{18}-A^{10}-A^6-A^2 $ & $ \text{K}_{1}^2 A^{14}-A^{14}-\text{K}_{1}^2 A^{10}-A^2 $ \\ \hline
$3.6 $ & $ A^{18}-A^{10}-A^6-A^2 $ & $ A^{18}-A^{10}-A^6-A^2 $ \\ \hline
$3.7 $ & $ -A^2-\frac{1}{A^2} $ & $ -A^{10}+\text{K}_{1}^2 A^6-A^6-\text{K}_{1}^2 A^2 $ \\ \hline

$4.1 $ & $ \text{D}_{4}[1] $ & $ \text{D}_{4}[1]$ \\ \hline
$4.2 $ & $ \text{D}_{4}[2] $ & $\text{D}_{4}[2] $\\ \hline
$4.3 $ & $ -A^2-\frac{1}{A^2} $ & $ -A^2-\frac{1}{A^2} $ \\ \hline
$4.4 $ & $ \text{D}_{4}[1] $ & $ \text{D}_{4}[1]$ \\ \hline
$4.5 $ & $ \text{D}_{4}[2] $ & $ \text{D}_{4}[2] $ \\ \hline
$4.6 $ & $ -A^2-\frac{1}{A^2} $ & $ -A^2-\frac{1}{A^2} $ \\ \hline
$4.7 $ & $ \text{D}_{4}[2] $ & $ \text{D}_{4}[2] $ \\ \hline
$4.8 $ & $ \text{D}_{4}[1] $ & $ \text{D}_{4}[1]$ \\ \hline
$4.9 $ & $ \text{D}_{2}[1] A^8-2 A^8-A^6-2 A^4-A^2 $ & $ \text{D}_{2}[2] A^8+2 \text{K}_{1} A^6-A^6-A^2 $ \\ \hline
$4.10 $ & $ A^{12}-A^6-A^4-A^2 $ & $ -\text{K}_{1} A^{10}+\text{K}_{1} A^6-A^6-A^2 $ \\ \hline
$4.11 $ & $ -A^4-A^2+\text{D}_{2}[1]-2-\frac{1}{A^2}-\frac{1}{A^4} $ & $ \text{K}_{1} A^2-A^2+\text{D}_{2}[3]+\frac{\text{K}_{1}}{A^2}-\frac{1}{A^2} $ \\ \hline
$4.12 $ & $ -A^2-\frac{1}{A^2} $ & $ -A^2-\frac{1}{A^2} $ \\ \hline
$4.13 $ & $ -A^2-\frac{1}{A^2} $ & $ -A^2-\frac{1}{A^2} $ \\ \hline
$4.14 $ & $ \frac{\text{D}_{2}[1]}{A^8}-\frac{1}{A^2}-\frac{2}{A^4}-\frac{1}{A^6}-\frac{2}{A^8} $ & $ \frac{2 \text{K}_{1}}{A^6}+\frac{\text{D}_{2}[2]}{A^8}-\frac{1}{A^2}-\frac{1}{A^6} $ \\ \hline
$4.15 $ & $ A^{12}-A^6-A^4-A^2 $ & $ -\text{K}_{1} A^{10}+\text{K}_{1} A^6-A^6-A^2 $ \\ \hline
$4.16 $ & $ \text{D}_{2}[1] A^8-2 A^8-A^6-2 A^4-A^2 $ & $ \text{D}_{2}[2] A^8+2 \text{K}_{1} A^6-A^6-A^2 $ \\ \hline
$4.17 $ & $ -A^4-A^2+\text{D}_{2}[1]-2-\frac{1}{A^2}-\frac{1}{A^4} $ & $ \text{K}_{1} A^2-A^2+\text{D}_{2}[3]+\frac{\text{K}_{1}}{A^2}-\frac{1}{A^2} $ \\ \hline
$4.18 $ & $ -A^2-\frac{1}{A^2} $ & $ -A^2-\frac{1}{A^2} $ \\ \hline
$4.19 $ & $ -A^4-A^2+\text{D}_{2}[1]-2-\frac{1}{A^2}-\frac{1}{A^4} $ & $ \text{K}_{1} A^2-A^2+\text{D}_{2}[2]+\frac{\text{K}_{1}}{A^2}-\frac{1}{A^2} $ \\ \hline
$4.20 $ & $ -\frac{1}{A^2}-\frac{1}{A^4}-\frac{1}{A^6}+\frac{1}{A^{12}} $ & $ \frac{\text{K}_{1}}{A^6}-\frac{\text{K}_{1}}{A^{10}}-\frac{1}{A^2}-\frac{1}{A^6} $ \\ \hline
$4.21 $ & $ \frac{\text{D}_{2}[1]}{A^8}-\frac{1}{A^2}-\frac{2}{A^4}-\frac{1}{A^6}-\frac{2}{A^8} $ & $ \frac{2 \text{K}_{1}}{A^6}+\frac{\text{D}_{2}[2]}{A^8}-\frac{1}{A^2}-\frac{1}{A^6} $ \\ \hline
$4.22 $ & $ -\frac{1}{A^2}-\frac{1}{A^4}-\frac{1}{A^6}+\frac{1}{A^{12}} $ & $ \frac{\text{K}_{1}}{A^6}-\frac{\text{K}_{1}}{A^{10}}-\frac{1}{A^2}-\frac{1}{A^6} $ \\ \hline
$4.23 $ & $ A^{12}-A^6-A^4-A^2 $ & $ -\text{K}_{1} A^{10}+\text{K}_{1} A^6-A^6-A^2 $ \\ \hline
$4.24 $ & $ -\frac{1}{A^2}-\frac{1}{A^4}-\frac{1}{A^6}+\frac{1}{A^{12}} $ & $ \frac{\text{K}_{1}}{A^6}-\frac{\text{K}_{1}}{A^{10}}-\frac{1}{A^2}-\frac{1}{A^6} $ \\ \hline
$4.25 $ & $ -A^2-\frac{1}{A^2} $ & $ -A^2-\frac{1}{A^2} $ \\ \hline
$4.26 $ & $ \text{D}_{4}[3] $ & $ \text{D}_{4}[3] $ \\ \hline
$4.27 $ & $ -A^2-\frac{1}{A^2} $ & $ -A^2-\frac{1}{A^2} $ \\ \hline
$4.28 $ & $ \text{D}_{4}[4] $ & $ \text{D}_{4}[4] $ \\ \hline
$4.29 $ & $ \text{D}_{2}[1] A^8-2 A^8-A^6-2 A^4-A^2 $ & $ \text{D}_{2}[3] A^8+2 \text{K}_{1} A^6-A^6-A^2 $ \\ \hline
$4.30 $ & $ -A^4-A^2+\text{D}_{2}[1]-2-\frac{1}{A^2}-\frac{1}{A^4} $ & $ \text{K}_{1} A^2-A^2+\text{D}_{2}[2]+\frac{\text{K}_{1}}{A^2}-\frac{1}{A^2} $ \\ \hline
$4.31 $ & $ A^{12}-A^6-A^4-A^2 $ & $ -\text{K}_{1} A^{10}+\text{K}_{1} A^6-A^6-A^2 $ \\ \hline
$4.32 $ & $ -A^2-\frac{1}{A^2} $ & $ -A^2-\frac{1}{A^2} $ \\ \hline
$4.33 $ & $ -A^4-A^2+\text{D}_{2}[1]-2-\frac{1}{A^2}-\frac{1}{A^4} $ & $ \text{K}_{1} A^2-A^2+\text{D}_{2}[2]+\frac{\text{K}_{1}}{A^2}-\frac{1}{A^2} $ \\ \hline
$4.34 $ & $ \frac{\text{D}_{2}[1]}{A^8}-\frac{1}{A^2}-\frac{2}{A^4}-\frac{1}{A^6}-\frac{2}{A^8} $ & $ \frac{2 \text{K}_{1}}{A^6}+\frac{\text{D}_{2}[2]}{A^8}-\frac{1}{A^2}-\frac{1}{A^6} $ \\ \hline
$4.35 $ & $ -A^2-\frac{1}{A^2} $ & $ -A^2-\frac{1}{A^2} $ \\ \hline
$4.36 $ & $ -\frac{1}{A^2}-\frac{1}{A^4}-\frac{1}{A^6}+\frac{1}{A^{12}} $ & $ \frac{\text{K}_{1}}{A^6}-\frac{\text{K}_{1}}{A^{10}}-\frac{1}{A^2}-\frac{1}{A^6} $ \\ \hline
$4.37 $ & $ A^{12}-A^6-A^4-A^2 $ & $ -\text{K}_{1} A^{10}+\text{K}_{1} A^6-A^6-A^2 $ \\ \hline
$4.38 $ & $ -A^2-\frac{1}{A^2} $ & $ -A^2-\frac{1}{A^2} $ \\ \hline
$4.39 $ & $ -A^2-\frac{1}{A^2} $ & $ -A^2-\frac{1}{A^2} $ \\ \hline
$4.40 $ & $ -\frac{1}{A^2}-\frac{1}{A^4}-\frac{1}{A^6}+\frac{1}{A^{12}} $ & $ \frac{\text{K}_{1}}{A^6}-\frac{\text{K}_{1}}{A^{10}}-\frac{1}{A^2}-\frac{1}{A^6} $ \\ \hline
$4.41 $ & $ A^{12}-A^6-A^4-A^2 $ & $ -\text{K}_{1} A^{10}+\text{K}_{1} A^6-A^6-A^2 $ \\ \hline
$4.42 $ & $ -A^2-\frac{1}{A^2} $ & $ -A^2-\frac{1}{A^2} $ \\ \hline
$4.43 $ & $ -A^2-\frac{1}{A^2} $ & $ -A^2-\frac{1}{A^2} $ \\ \hline
$4.44 $ & $ -A^2-\frac{1}{A^2} $ & $ -A^2-\frac{1}{A^2} $ \\ \hline
$4.45 $ & $ \text{D}_{4}[3] $ & $ \text{D}_{4}[3] $ \\ \hline
$4.46 $ & $ -A^2-\frac{1}{A^2} $ & $ -A^2-\frac{1}{A^2} $ \\ \hline
$4.47 $ & $ \text{D}_{4}[4] $ & $ \text{D}_{4}[4] $ \\ \hline
$4.48 $ & $ \text{D}_{2}[1] A^8-2 A^8-A^6-2 A^4-A^2 $ & $ \text{D}_{2}[2] A^8+2 \text{K}_{1} A^6-A^6-A^2 $ \\ \hline
$4.49 $ & $ -A^4-A^2+\text{D}_{2}[1]-2-\frac{1}{A^2}-\frac{1}{A^4} $ & $ \text{K}_{1} A^2-A^2+\text{D}_{2}[3]+\frac{\text{K}_{1}}{A^2}-\frac{1}{A^2} $ \\ \hline
$4.50 $ & $ A^{12}-A^6-A^4-A^2 $ & $ -\text{K}_{1} A^{10}+\text{K}_{1} A^6-A^6-A^2 $ \\ \hline
$4.51 $ & $ -A^2-\frac{1}{A^2}  $ & $ -A^2-\frac{1}{A^2} $ \\ \hline
$4.52 $ & $ \frac{\text{D}_{2}[1]}{A^8}-\frac{1}{A^2}-\frac{2}{A^4}-\frac{1}{A^6}-\frac{2}{A^8} $ & $ \frac{2 \text{K}_{1}}{A^6}+\frac{\text{D}_{2}[2]}{A^8}-\frac{1}{A^2}-\frac{1}{A^6} $ \\ \hline
$4.53 $ & $ -A^2-\frac{1}{A^2} $ & $ -A^2-\frac{1}{A^2} $ \\ \hline
$4.54 $ & $ -A^2-\frac{1}{A^2} $ & $ -A^2-\frac{1}{A^2} $ \\ \hline
$4.55 $ & $ \text{D}_{4}[1] $ & $ \text{D}_{4}[1] $ \\ \hline
$4.56 $ & $ \text{D}_{4}[2] $ & $ \text{D}_{4}[2] $ \\ \hline
$4.57 $ & $ \text{D}_{2}[1] A^8-2 A^8-A^6-2 A^4-A^2 $ & $ \text{D}_{2}[3] A^8+2 \text{K}_{1} A^6-A^6-A^2 $ \\ \hline
$4.58 $ & $ -A^4-A^2+\text{D}_{2}[1]-2-\frac{1}{A^2}-\frac{1}{A^4} $ & $ \text{K}_{1} A^2-A^2+\text{D}_{2}[2]+\frac{\text{K}_{1}}{A^2}-\frac{1}{A^2} $ \\ \hline
$4.59 $ & $ -A^4-A^2+\text{D}_{2}[1]-2-\frac{1}{A^2}-\frac{1}{A^4} $ & $ \text{K}_{1} A^2-A^2+\text{D}_{2}[2]+\frac{\text{K}_{1}}{A^2}-\frac{1}{A^2} $ \\ \hline
$4.60 $ & $ -\frac{1}{A^2}-\frac{1}{A^4}-\frac{1}{A^6}+\frac{1}{A^{12}} $ & $ \frac{\text{K}_{1}}{A^6}-\frac{\text{K}_{1}}{A^{10}}-\frac{1}{A^2}-\frac{1}{A^6} $ \\ \hline
$4.61 $ & $ A^{12}-A^6-A^4-A^2 $ & $ -\text{K}_{1} A^{10}+\text{K}_{1} A^6-A^6-A^2 $ \\ \hline
$4.62 $ & $ -A^2-\frac{1}{A^2} $ & $ -A^2-\frac{1}{A^2} $ \\ \hline
$4.63 $ & $ -A^2-\frac{1}{A^2} $ & $ -A^2-\frac{1}{A^2} $ \\ \hline
$4.64 $ & $ -\frac{1}{A^2}-\frac{1}{A^4}-\frac{1}{A^6}+\frac{1}{A^{12}} $ & $ \frac{\text{K}_{1}}{A^6}-\frac{\text{K}_{1}}{A^{10}}-\frac{1}{A^2}-\frac{1}{A^6} $ \\ \hline
$4.65 $ & $ A^{12}-A^6-A^4-A^2 $ & $ -\text{K}_{1} A^{10}+\text{K}_{1} A^6-A^6-A^2 $ \\ \hline
$4.66 $ & $ -A^2-\frac{1}{A^2} $ & $ -A^2-\frac{1}{A^2} $ \\ \hline
$4.67 $ & $ -A^2-\frac{1}{A^2} $ & $ -A^2-\frac{1}{A^2} $ \\ \hline
$4.68 $ & $ -\frac{1}{A^2}-\frac{1}{A^4}-\frac{1}{A^6}+\frac{1}{A^{12}} $ & $ \frac{\text{K}_{1}}{A^6}-\frac{\text{K}_{1}}{A^{10}}-\frac{1}{A^2}-\frac{1}{A^6} $ \\ \hline
$4.69 $ & $ A^{12}-A^6-A^4-A^2 $ & $ -\text{K}_{1} A^{10}+\text{K}_{1} A^6-A^6-A^2 $ \\ \hline
$4.70 $ & $ \text{D}_{2}[1] A^8-2 A^8-A^6-2 A^4-A^2 $ & $ \text{D}_{2}[3] A^8+2 \text{K}_{1} A^6-A^6-A^2 $ \\ \hline
$4.71 $ & $ -A^4-A^2+\text{D}_{2}[1]-2-\frac{1}{A^2}-\frac{1}{A^4} $ & $ \text{K}_{1} A^2-A^2+\text{D}_{2}[2]+\frac{\text{K}_{1}}{A^2}-\frac{1}{A^2} $ \\ \hline
$4.72 $ & $ -A^4-A^2+\text{D}_{2}[1]-2-\frac{1}{A^2}-\frac{1}{A^4} $ & $ \text{K}_{1} A^2-A^2+\text{D}_{2}[2]+\frac{\text{K}_{1}}{A^2}-\frac{1}{A^2} $ \\ \hline
$4.73 $ & $ -A^2-\frac{1}{A^2} $ & $ -A^2-\frac{1}{A^2} $ \\ \hline
$4.74 $ & $ -A^2-\frac{1}{A^2} $ & $ -A^2-\frac{1}{A^2} $ \\ \hline
$4.75 $ & $ -A^2-\frac{1}{A^2} $ & $ -A^2-\frac{1}{A^2} $ \\ \hline
$4.76 $ & $ \text{D}_{4}[2] $ & $ \text{D}_{4}[2] $ \\ \hline
$4.77 $ & $ \text{D}_{4}[1] $ & $ \text{D}_{4}[1] $ \\ \hline
$4.78 $ & $ A^{12}-A^6-A^4-A^2 $ & $ -\text{K}_{1} A^{10}+\text{K}_{1} A^6-A^6-A^2 $ \\ \hline
$4.79 $ & $ A^{12}-A^6-A^4-A^2 $ & $ -\text{K}_{1} A^{10}+\text{K}_{1} A^6-A^6-A^2 $ \\ \hline
$4.80 $ & $ \text{D}_{4}[3] $ & $ \text{D}_{4}[3] $ \\ \hline
$4.81 $ & $ \text{D}_{4}[4] $ & $ \text{D}_{4}[4] $ \\ \hline
$4.82 $ & $ -A^{10}+\text{D}_{2}[1] A^6-A^6-\text{D}_{2}[1] A^2 $ & $ -A^{10}+\text{D}_{2}[1] A^6-A^6-\text{D}_{2}[1] A^2 $ \\ \hline
$4.83 $ & $ -A^2-\frac{1}{A^2} $ & $ -A^2-\frac{1}{A^2} $ \\ \hline
$4.84 $ & $ -A^{10}+\text{D}_{2}[1] A^6-A^6-\text{D}_{2}[1] A^2 $ & $ -A^{10}+\text{D}_{2}[1] A^6-A^6-\text{D}_{2}[1] A^2 $ \\ \hline
$4.85 $ & $ -A^2-\frac{1}{A^2} $ & $ -A^{10}+\text{K}_{1}^2 A^6-A^6-\text{K}_{1}^2 A^2 $ \\ \hline
$4.86 $ & $ -A^{10}-\frac{1}{A^{10}} $ & $ -A^{10}-A^2+\frac{\text{K}_{1}^2}{A^2}-\frac{1}{A^2}-\frac{\text{K}_{1}^2}{A^6}+\frac{1}{A^6} $ \\ \hline
$4.87 $ & $ -A^{10}+\text{D}_{2}[1] A^6-A^6-\text{D}_{2}[1] A^2 $ & $ -A^{10}+\text{D}_{2}[1] A^6-A^6-\text{D}_{2}[1] A^2 $ \\ \hline
$4.88 $ & $ -A^{10}+\text{D}_{2}[1] A^6-A^6-\text{D}_{2}[1] A^2 $ & $ -A^{10}+\text{D}_{2}[1] A^6-A^6-\text{D}_{2}[1] A^2 $ \\ \hline
$4.89 $ & $ A^{18}-A^{10}-A^6-A^2 $ & $ -A^{18}+2 \text{K}_{1}^2 A^{14}-2 A^{14}-2 \text{K}_{1}^2 A^{10}+A^{10}+A^6-A^2 $ \\ \hline
$4.90 $ & $ -A^{10}-\frac{1}{A^{10}} $ & $ -\text{K}_{1}^2 A^6+A^6+\text{K}_{1}^2 A^2-2 A^2+\frac{\text{K}_{1}^2}{A^2}-\frac{2}{A^2}-\frac{\text{K}_{1}^2}{A^6}+\frac{1}{A^6} $ \\ \hline
$4.91 $ & $ -A^2-\frac{1}{A^2} $ & $ -A^2-\frac{1}{A^2} $ \\ \hline
$4.92 $ & $ -A^2-\frac{1}{A^2} $ & $ -A^2-\frac{1}{A^2} $ \\ \hline
$4.93 $ & $ -A^{10}+\text{D}_{2}[1] A^6-A^6-\text{D}_{2}[1] A^2 $ & $ -A^{10}+\text{D}_{2}[1] A^6-A^6-\text{D}_{2}[1] A^2 $ \\ \hline
$4.94 $ & $ -A^2-\frac{1}{A^2} $ & $ -A^2-\frac{1}{A^2} $ \\ \hline
$4.95 $ & $ -A^2-\frac{1}{A^2} $ & $ -A^2-\frac{1}{A^2} $ \\ \hline
$4.96 $ & $ -A^{10}+\text{D}_{2}[1] A^6-A^6-\text{D}_{2}[1] A^2 $ & $ -A^{10}+\text{D}_{2}[1] A^6-A^6-\text{D}_{2}[1] A^2 $ \\ \hline
$4.97 $ & $ -A^2-\frac{1}{A^2} $ & $ -A^2-\frac{1}{A^2} $ \\ \hline
$4.98 $ & $ -A^2-\frac{1}{A^2} $ & $ -A^2-\frac{1}{A^2} $ \\ \hline
$4.99 $ & $ -A^{10}-\frac{1}{A^{10}} $ & $ -A^{10}-\frac{1}{A^{10}} $ \\ \hline
$4.100 $ & $ -A^2-\frac{1}{A^2} $ & $ -A^2-\frac{1}{A^2} $ \\ \hline
$4.101 $ & $ -A^2-\frac{1}{A^2} $ & $ -A^2-\frac{1}{A^2} $ \\ \hline
$4.102 $ & $ -A^2-\frac{1}{A^2} $ & $ -A^2-\frac{1}{A^2} $ \\ \hline
$4.103 $ & $ -A^{10}+\text{D}_{2}[1] A^6-A^6-\text{D}_{2}[1] A^2 $ & $ -A^{10}+\text{D}_{2}[1] A^6-A^6-\text{D}_{2}[1] A^2 $ \\ \hline
$4.104 $ & $ -A^2-\frac{1}{A^2} $ & $ -A^2-\frac{1}{A^2} $ \\ \hline
$4.105 $ & $ A^{18}-A^{10}-A^6-A^2 $ & $ A^{18}-A^{10}-A^6-A^2 $ \\ \hline
$4.106 $ & $ -A^2-\frac{1}{A^2} $ & $ -A^{10}+\text{K}_{1}^2 A^6-A^6-\text{K}_{1}^2 A^2 $ \\ \hline
$4.107 $ & $ -A^2-\frac{1}{A^2} $ & $ -A^2-\frac{1}{A^2} $ \\ \hline
$4.108 $ & $ -A^{10}-\frac{1}{A^{10}} $ & $ -A^{10}-\frac{1}{A^{10}} $ \\ \hline
\end{longtable}
}

\subsection{A Mathematica Program}

\subsubsection{A Parity Categorification}
The following program for a categorification of the arrow polynomial and the forgetful parity version is based on Dror Bar-Natan's construction \cite{DrorCat} for Khovanov homology.  Here we implement a version of Gaussian elimination for computing homology with coefficients over $\mathbb{Z}_2$ that was pointed out to us by Marc Culler and implemented by Baldwin and Gillam for computation of Heegaard-Floer knot homology in \cite{BaldwinGillam}. This can be described graphically as in Figure \ref{fig:GraphReduction} where we reduce based on the chosen marked edge.  This is equivalent to applying Gaussian Elimination to the chain complex as in Figure \ref{fig:GaussElim} where we assume $\phi$ (the equivalent of the selected edge) is invertible. Maps denoted by $\bullet$ are arbitrary and inconsequential in the final result.  For more on Gaussian Elimination and homotopy equivalence we point the reader to \cite{CMW}.

\begin{rem} A similar program for Khovanov homology and the forgetful parity version with $\mathbb{Z}_2$ coefficients is available on the first authors website.
\end{rem}

\begin{figure}[h!]
\centering
    \includegraphics[height=3cm]{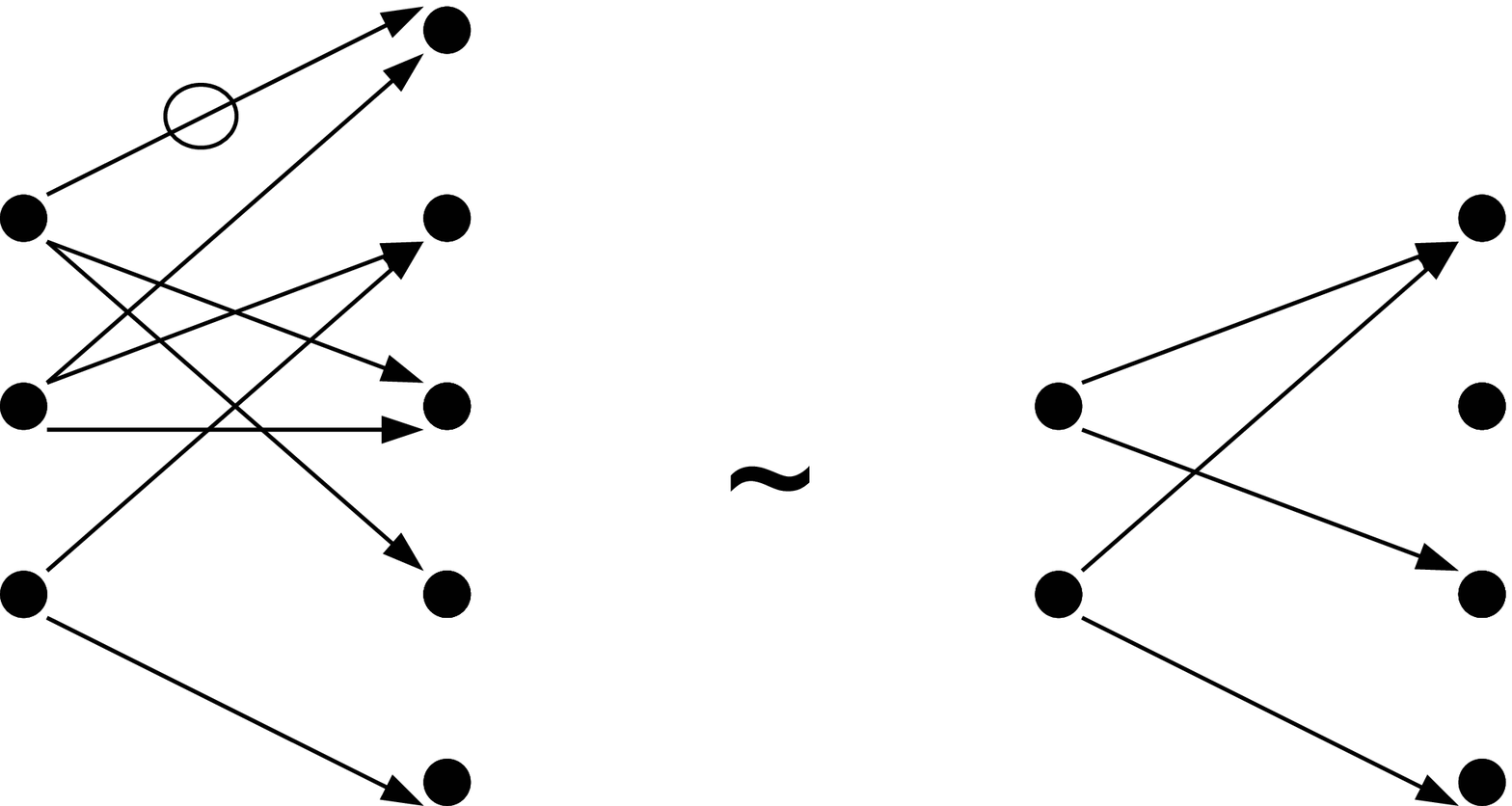}
\caption{}
\label{fig:GraphReduction}
\end{figure}

\begin{figure}[h!]
\centering
    \includegraphics[height=10cm]{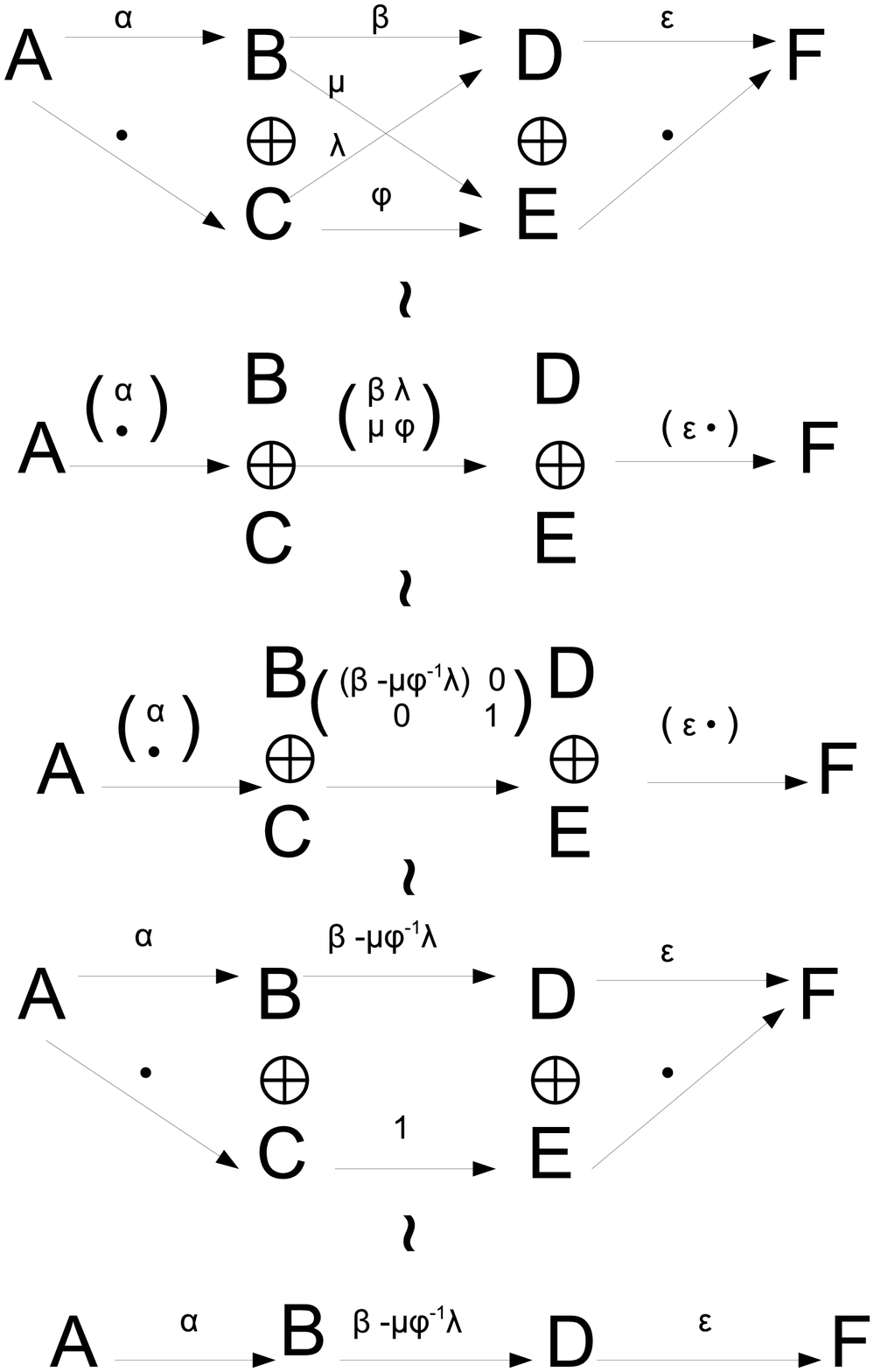}
\caption{}
\label{fig:GaussElim}
\end{figure}

\newpage

{\small
\begin{verbatim}
np[L_PD] := Count[L, _X];
nm[L_PD] := Count[L, _Y];
SetAttributes[del, Orderless]
\end{verbatim}
\verb+np+ and \verb+nm+ count the number of positive and negative crossings for a given planar diagram respectively.   We set  \verb+del+  to be orderless to reduce the number of necessary relations.\\

The following lines of program perform the forgetful mapping on the odd crossings.
\begin{verbatim}
EvenParityPD[L_PD] :=
 Sort[L /. {X[i_, j_, k_, l_] :>
      Odd[i, j, k, l] /; OddQ[i - j]} /. {Y[i_, j_, k_, l_] :>
     Odd[i, j, k, l] /; OddQ[i - j]}]
\end{verbatim}
\verb+EvenParity+ performs a simple check to determine if a given crossing is even or odd.  For an odd crossing it replaces the head \verb+X+ or \verb+Y+ with \verb+Odd+.  Note this subroutine assumes that arcs of a given knot diagram are labeling consecutively.  \\
\begin{verbatim}
OddIdentities[L_PD] :=
 ReplacePart[ ReplacePart[Reverse[Sort[
      Flatten[ReplacePart[Take[EvenParityPD[L],
          Length[EvenParityPD[L]] - (np[EvenParityPD[L]] +
             nm[EvenParityPD[L]])],
         0 :> List] /. {Odd[i_, j_, k_, l_] :> {del[i, k],
           del[j, l]}}]]],
    0 :> Times] //. {del[a_, b_] del[b_, c_] :> del[a, c]}, 0 :> List]
\end{verbatim}
\verb+OddIdentities+ creates a collection of arc relations displayed in terms of Kronecker deltas based on the odd crossings.  \\
\begin{verbatim}
EvenCross[L_PD] :=
 Drop[EvenParityPD[L],
  Length[EvenParityPD[L]] - (np[EvenParityPD[L]] +
     nm[EvenParityPD[L]])]
\end{verbatim}
\verb+EvenCross+ collects the even crossings from a given PD code.  \\
\begin{verbatim}
PDReduction[L_PD, d_del] := (mm := Min[d[[1]], d[[2]]];
  nn = Max[d[[1]], d[[2]]]; L /. {nn :> mm})
\end{verbatim}
\verb+PDReduction+ turns a single identity produced by \verb+OddIdentities+ into a reduction relation and applies this relation.\\
\begin{verbatim}
ForgetfulEvenParityPD[L_PD] := (RedPD = EvenCross[L];
  Do[RedPD = PDReduction[RedPD, OddIdentities[L][[i]]], {i,
    Length[OddIdentities[L]]}]; RedPD)
\end{verbatim}
\verb+ForgetfulEvenParityPD+ applies \verb+PDReduction+ for all of the relations in \verb+OddIdentities+ and returns the resulting PD code. \\
\begin{verbatim}
rule2 = {del[a_, b_][m_] del[a_, b_][m_] :> del[a, a][m],
   del[a_, b_][m_] del[b_, c_][n_] :> del[a, c][Min[m, n]],
   del[a_, b_][m_] led[b_, c_][n_] :> led[a, c][Min[m, n]],
   del[a_, b_][m_] led[c_, b_][n_] :> led[c, a][Min[m, n]],
   led[a_, b_][m_] del[b_, c_][n_] :> led[a, c][Min[m, n]],
   led[a_, b_][m_] del[c_, b_][n_] :> led[a, c][Min[m, n]],
   led[a_, b_][m_] led[b_, c_][n_] :> del[a, c][Min[m, n]]};
rule3 = {del[a_, a_][m1_] :> c[m1, 0, 0],
   del[a_, a_][m1_]^_ :> c[m1, 0, 0],
   led[a_, b_][m1_]^2 :> c[m1, 1, 1],
   led[a_, b_][m1_] led[a_, b_][m2_] :> c[Min[m1, m2], 1, 1],
   led[a_, b_][m1_] led[x_, b_][m2_] led[x_, d_][m3_] led[a_, d_][
      m4_] :> c[Min[m1, m2, m3, m4], 2, 2],
   led[a_, b_][m1_] led[x_, b_][m2_] led[x_, d_][m3_] led[e_, d_][
      m4_] led[e_, f_][m5_] led[a_, f_][m6_] :>
    c[Min[m1, m2, m3, m4, m5, m6], 3, 1],
   led[a_, b_][m1_] led[x_, b_][m2_] led[x_, d_][m3_] led[e_, d_][
      m4_] led[e_, f_][m5_] led[g_, f_][m6_] led[g_, h_][
      m7_] led[a_, h_][m8_] :>
    c[Min[m1, m2, m3, m4, m5, m6, m7, m8], 4, 3],
   led[a_, b_][m1_] led[x_, b_][m2_] led[x_, d_][m3_] led[e_, d_][
      m4_] led[e_, f_][m5_] led[g_, f_][m6_] led[g_, h_][
      m7_] led[y_, h_][m8_] led[y_, z_][m9_] led[a_, z_][m10_] :>
    c[Min[m1, m2, m3, m4, m5, m6, m7, m8, m9, m10], 5, 1]};
\end{verbatim}
\verb+rule2+ and \verb+rule3+ are reduction relations used by \verb+S+. \verb+rule2+ joins arcs and cusps while \verb+rule3+ produces the labeled circles for a basic (unenhanced) state. We follow the convention \verb+c[m, p,k]+ is circle \verb+m+ with arrow number \verb+p+ and dot of order \verb+k+. Where \verb+p = l*2^(k - 1)+ for \verb+l+ odd.\\
\begin{verbatim}
ruleStar = {v___c u___ X[i_, j_, k_,
      l_] :> ((u del[i, j][Min[i, j]] del[k, l][Min[k, l]] //.
         rule2 //.
        rule3) -> (u led[l, i][Min[l, i]] led[j, k][Min[j, k]] //.
         rule2 //. rule3)),
   v___c u___ Y[i_, j_, k_,
      l_] :> ((u led[l, i][Min[l, i]] led[j, k][Min[j, k]] //.
         rule2 //.
        rule3) -> (u del[i, j][Min[i, j]] del[k, l][Min[k, l]] //.
         rule2 //. rule3))};
\end{verbatim}
\verb+ruleStar+ is a reduction relations used by \verb+S+ which produces notation corresponding to a bifurcation on an edge denoted by \verb+*+ on the cube complex.\\
\begin{verbatim}
S[L_PD, a_List] :=
 Times[(Times @@ (Thread[{List @@
            Drop[L, Length[L] - (np[L] + nm[L])],
           a}] /. {{X[i_, j_, k_, l_], 0} :>
           del[i, j][Min[i, j]] del[k, l][Min[k, l]], {X[i_, j_, k_,
             l_], 1} :>
           led[l, i][Min[l, i]] led[j, k][Min[j, k]], {Y[i_, j_, k_,
             l_], 0} :>
           led[l, i][Min[l, i]] led[j, k][Min[j, k]], {Y[i_, j_, k_,
             l_], 1} :>
           del[i, j][Min[i, j]] del[k, l][Min[k, l]], {x_X, "*"} :>
           x, {y_Y, "*"} :> y})), (Times @@ (Take[L,
          Length[L] - (np[L] + nm[L])] /. {Odd[i_, j_, k_, l_] :>
           del[i, k][Min[i, k]] del[j, l][Min[j, l]]}))] //.
    rule2 //. rule3 //. ruleStar
S[L_PD, s_String] := S[L, Characters[s] /. {"0" -> 0, "1" -> 1}]
\end{verbatim}
\verb+S+ produces the unenhanced state of the cube complex for \verb+L+ corresponding to the vertex \verb+a+.\\
\begin{verbatim}
MG[expr_] :=
 expr //. {c[m_, 0, k_] :> 1} //. {c[m_, p_, k_] :>
     a[p]} //. {a[i_]^_ :> a[i]}
\end{verbatim}
\verb+MG+ computes the multiple grading for an given unenhanced state. If there is no arrow numbers \verb+MG+ returns 1.  If there are arrow numbers \verb+MG+ returns a product of the form $a[i_1] a[i_ 2]\ldots a[i_n]$, where the $i_j$ are the distinct arrow numbers (ie $i_j = i_k$ iff $j = k$).\\
\begin{verbatim}
Deg[expr_] := Count[expr, _v1, {0, 1}] - Count[expr, _vX, {0, 1}]
V[L_PD, s_String, deg___] :=
 V[L, Characters[s] /. {"0" -> 0, "1" -> 1}, deg]
V[L_PD, a_List] :=
 List @@ Expand[S[L, a] /. x_c :> ((vX @@ x) + (v1 @@ x))]
V[L_PD, a_List, deg_Integer] :=
 Select[V[L, a], (deg == Deg[#] + (Plus @@ a)) &]
\end{verbatim}
The above subroutines provide information on the enhanced states. \verb+V+ replaces \verb+c[m,p,k]+ by \verb|vX[m,p,k]+v1[m,p,k]| throughout then expands each expression. Each summand corresponds to an enhanced (labeled by X and 1) state, which we separate into a list of enhanced states at each vertex. \verb+Deg+ computes ($\#$ X's - $\#$1's)) and when given to $V$ returns enhanced states with a given bi-degree. ((bi-degree) = (homological degree) + ($\#$ X's - $\#$1's))\\
\begin{verbatim}
VG[expr_] :=
 expr //. {vX[m_, 0, k_] :> 1} //. {v1[m_, 0, k_] :> 1} //. {vX[m_,
         p_, k_] :> vg[k, 1]} //. {v1[m_, p_, k_] :>
       vg[k, -1]} //. {vg[a_, i_]^m_ :>
      vg[a, i*m]} //. {vg[a_, i_] vg[a_, j_] :> vg[a, i + j]} //. {vg[
     a_, 0] :> 1}
\end{verbatim}
We need to compute the vector grading of an enhanced state. Recall that this (inf. dim.) vector is the sum over the labeling in the enhanced state where, for $i > 0$, we transform \verb+vX[*,*,i]+ into the vector that is 1 in the $i^th$ position and 0 elsewhere and similarly we transform \verb+v1[*,*,i]+ into the vector that is -1 in the $i^th$ position and 0 elsewhere. \verb+VG+ returns 1 if the vector grading is the zero vector else it returns the product of terms of the form \verb+v[k, n]+ corresponding to the $k^th$ position in the vector grading having value n.\\
\begin{verbatim}
d[L_PD, s_String] := d[L, Characters[s] /. {"0" -> 0, "1" -> 1}]
d[L_PD, a_List] :=
 S[L, a] //. {(c[x__] c[y__] -> c[z__])*_. :> {v1@x v1@y -> 0,
        v1@x vX@y -> 0, vX@x v1@y -> 0,
        vX@x vX@y -> 0} /; (MG[S[L, a //. {"*" -> 0}]] =!=
         MG[S[L, a //. {"*" -> 1}]]), (c[z__] ->
         c[x__] c[y__])*_. :> {v1@z -> 0,
        vX@z -> 0} /; (MG[S[L, a //. {"*" -> 0}]] =!=
         MG[S[L, a //. {"*" -> 1}]])} //. {(c[x__] c[y__] ->
        c[z__])*_. :> {v1@x v1@y -> v1@z, v1@x vX@y -> vX@z,
       vX@x v1@y -> vX@z,
       vX@x vX@y ->
        0} /; (VG[v1@x v1@y] === VG[v1@z]) && (VG[v1@x vX@y] ===
         VG[vX@z]) && (VG[vX@x v1@y] === VG[vX@z]), (c[x__] c[y__] ->
        c[z__])*_. :> {v1@x v1@y -> 0, v1@x vX@y -> vX@z,
       vX@x v1@y -> vX@z,
       vX@x vX@y ->
        0} /; (VG[v1@x v1@y] =!= VG[v1@z]) && (VG[v1@x vX@y] ===
         VG[vX@z]) && (VG[vX@x v1@y] === VG[vX@z]), (c[x__] c[y__] ->
        c[z__])*_. :> {v1@x v1@y -> v1@z, v1@x vX@y -> 0,
       vX@x v1@y -> vX@z,
       vX@x vX@y ->
        0} /; (VG[v1@x v1@y] === VG[v1@z]) && (VG[v1@x vX@y] =!=
         VG[vX@z]) && (VG[vX@x v1@y] === VG[vX@z]), (c[x__] c[y__] ->
        c[z__])*_. :> {v1@x v1@y -> v1@z, v1@x vX@y -> vX@z,
       vX@x v1@y -> 0,
       vX@x vX@y ->
        0} /; (VG[v1@x v1@y] === VG[v1@z]) && (VG[v1@x vX@y] ===
         VG[vX@z]) && (VG[vX@x v1@y] =!= VG[vX@z]), (c[x__] c[y__] ->
        c[z__])*_. :> {v1@x v1@y -> 0, v1@x vX@y -> 0,
       vX@x v1@y -> vX@z,
       vX@x vX@y ->
        0} /; (VG[v1@x v1@y] =!= VG[v1@z]) && (VG[v1@x vX@y] =!=
         VG[vX@z]) && (VG[vX@x v1@y] === VG[vX@z]), (c[x__] c[y__] ->
        c[z__])*_. :> {v1@x v1@y -> 0, v1@x vX@y -> vX@z,
       vX@x v1@y -> 0,
       vX@x vX@y ->
        0} /; (VG[v1@x v1@y] =!= VG[v1@z]) && (VG[v1@x vX@y] ===
         VG[vX@z]) && (VG[vX@x v1@y] =!= VG[vX@z]), (c[x__] c[y__] ->
        c[z__])*_. :> {v1@x v1@y -> v1@z, v1@x vX@y -> 0,
       vX@x v1@y -> 0,
       vX@x vX@y ->
        0} /; (VG[v1@x v1@y] === VG[v1@z]) && (VG[v1@x vX@y] =!=
         VG[vX@z]) && (VG[vX@x v1@y] =!= VG[vX@z]), (c[x__] c[y__] ->
        c[z__])*_. :> {v1@x v1@y -> 0, v1@x vX@y -> 0, vX@x v1@y -> 0,
        vX@x vX@y ->
        0} /; (VG[v1@x v1@y] =!= VG[v1@z]) && (VG[v1@x vX@y] =!=
         VG[vX@z]) && (VG[vX@x v1@y] =!= VG[vX@z]), (c[z__] ->
        c[x__] c[y__])*_. :> {v1@z -> v1@x vX@y + vX@x v1@y,
       vX@z ->
        vX@x vX@y} /; (VG[v1@z] === VG[v1@x vX@y]) && (VG[v1@z] ===
         VG[vX@x v1@y]) && (VG[vX@z] === VG[vX@x vX@y]), (c[z__] ->
        c[x__] c[y__])*_. :> {v1@z -> vX@x v1@y,
       vX@z ->
        vX@x vX@y} /; (VG[v1@z] =!= VG[v1@x vX@y]) && (VG[v1@z] ===
         VG[vX@x v1@y]) && (VG[vX@z] === VG[vX@x vX@y]), (c[z__] ->
        c[x__] c[y__])*_. :> {v1@z -> v1@x vX@y,
       vX@z ->
        vX@x vX@y} /; (VG[v1@z] === VG[v1@x vX@y]) && (VG[v1@z] =!=
         VG[vX@x v1@y]) && (VG[vX@z] === VG[vX@x vX@y]), (c[z__] ->
        c[x__] c[y__])*_. :> {v1@z -> v1@x vX@y + vX@x v1@y,
       vX@z ->
        0} /; (VG[v1@z] === VG[v1@x vX@y]) && (VG[v1@z] ===
         VG[vX@x v1@y]) && (VG[vX@z] =!= VG[vX@x vX@y]), (c[z__] ->
        c[x__] c[y__])*_. :> {v1@z -> 0,
       vX@z ->
        vX@x vX@y} /; (VG[v1@z] =!= VG[v1@x vX@y]) && (VG[v1@z] =!=
         VG[vX@x v1@y]) && (VG[vX@z] === VG[vX@x vX@y]), (c[z__] ->
        c[x__] c[y__])*_. :> {v1@z -> vX@x v1@y,
       vX@z ->
        0} /; (VG[v1@z] =!= VG[v1@x vX@y]) && (VG[v1@z] ===
         VG[vX@x v1@y]) && (VG[vX@z] =!= VG[vX@x vX@y]), (c[z__] ->
        c[x__] c[y__])*_. :> {v1@z -> v1@x vX@y,
       vX@z ->
        0} /; (VG[v1@z] === VG[v1@x vX@y]) && (VG[v1@z] =!=
         VG[vX@x v1@y]) && (VG[vX@z] =!= VG[vX@x vX@y]), (c[z__] ->
        c[x__] c[y__])*_. :> {v1@z -> 0,
       vX@z ->
        0} /; (VG[v1@z] =!= VG[v1@x vX@y]) && (VG[v1@z] =!=
         VG[vX@x v1@y]) && (VG[vX@z] =!= VG[vX@x vX@y])} //. {(c[
        x__] -> c[y__])*_. :> {v1@x -> 0, vX@x -> 0}}
\end{verbatim}
\verb+d+ computes the edge morphism for the edge corresponding to the label \verb+a+. Here \verb+a+ is a list of 0's and 1's  along with a single \verb+*+ where \verb+*+ corresponds to the crossing we are resmoothing and 0 and 1 correspond to $A-$ and $A^{-1}$-smoothings at the remaining crossings.\\
We now have enough to construct the cube complex.  The following collection of routines together collect this information and constructs a graph.  We then preform the previously mentioned graphical reduction algorithm to compute the homology.\\
\begin{verbatim}
dif[L_PD, s_String] := dif[L, Characters[s] /. {"0" -> 0, "1" -> 1}]
dif[L_PD, a_List] :=
 Flatten[MapThread[
     ge, {V[L, a /. ("*" :> 0)],
      Expand[V[L, a /. ("*" :> 0)] /. d[L, a]]}] /. (ge[u___,
       v__ + w__] :> {ge[u, v], ge[u, w]}) /. (ge[z___, 0] :> 0)]
Comp[L_PD] :=
  Join @@ (Join @@ {Expand[
          ed[((v @@ #) /. ("*" :> 0)), ((v @@ #) /. ("*" :> 1))] dif[
            L, #]]} & /@ Perms[L]) //. (ed[a__, b__] ge[c__, d__] :>
     edge[a*c, b*d]);
Edges[L_PD] :=
 Cases[If[((# === 0) || (#[[1]] === #[[2]])), 0, #] & /@ Comp[L],
  Except[0]]
\end{verbatim}
\verb+dif+ constructs the set of directed edges for the graph corresponding to an edge of the cube complex. It does the locally by applying the differential \verb+d+ to the tail of each edge. \verb+Comp+ produces the full set of vertices for the graph and connected the heads and tails of the directed edges formed by \verb+dif+. Some zero differentials remain. \verb+Edges+ removes these from the list. \\
\begin{verbatim}
KhColumn[L_PD, r_Integer] :=
  If[r < 0 || r > (np[L] + nm[L]), {0},
   Join @@ (((v @@ #) V[L, #]) & /@
      Permutations[
       Join[Table[0, {(np[L] + nm[L]) - r}], Table[1, {r}]]])];
Gens[L_PD] :=
  Cases[Flatten[{KhColumn[L, #] & /@ Range[0, (np[L] + nm[L])]}],
   Except[0]];
\end{verbatim}
\verb+Gens+ produces the collection of enhanced states corresponding to the enhanced states of the complex (i.e. the nodes in the graph) by calling \verb+KhColumn+ for each homological degree of the planar diagram for the knot.
\begin{verbatim}
Perms[L_PD] :=
  Join @@ (Permutations[
       Join[{"*"}, Table[0, {(np[L] + nm[L]) - # - 1}],
        Table[1, {#}]]] & /@ Range[0, (np[L] + nm[L]) - 1]);
\end{verbatim}
\verb+Perms+ generates the lists of 0's, 1's and a single \verb+*+ corresponding to the edges of the cube complex.
\begin{verbatim}
Height[gen___] := (gen /. {v1[a___] :> 1, vX[b___] :> 1,
     v[c___] :> Plus[c]});
EdgeHeight[e__] := If[IntegerQ[e], -1, Height[e[[1]]]];
\end{verbatim}
\begin{verbatim}
MultGrad[e__] :=
  e /. {v[expr___] :> 1, vX[a_, b_, c_] :> K[b],
      v1[a_, b_, c_] :> K[b]} /. {K[0] :> 1} /. {K[i_]^_ :> K[i]};
\end{verbatim}
\verb+MultGrad+ takes a homology class representative and outputs its multiple grading.
\begin{verbatim}
AddDelEdges[edges_List, e_edge] := (Off[Part::partd];
  listfs = Select[edges, #[[2]] == e[[2]] &];
  listgs = Select[edges, #[[1]] == e[[1]] &];
  remfs = Rule[#, 0] & /@ listfs; remgs = Rule[#, 0] & /@ listgs;
  sym = edges /. remfs /. remgs /. {edge[e[[2]], ___] :> 0};
  diff = Tuples[{listfs,listgs}] /.
    {{edge[f1__, e2__], edge[e1__, g2__]} :> edge[f1, g2]};
  symdiff =
   If[(listfs == {}) || (listgs == {}), sym,
    Union[Complement[diff, sym], Complement[sym, diff]]];
  On[Part::partd]; symdiff)
\end{verbatim}
Given an edge \verb+e+, \verb+AddDelEdges+ looks for all local subgraphs which share the same head as \verb+e+.  It then computes a symmetric difference with collection of edges whose tail is the same as \verb+e+.\\
\begin{verbatim}
Reduc[gens_List, edges_List,
   e_edge] := (Checks[(gens /. {e[[1]] :> 0, e[[2]] :> 0}),
    AddDelEdges[edges, e]]);
Checks[gens_List, unsortededges_List] := (
   edges = Cases[(SortBy[unsortededges, EdgeHeight[#] &]), Except[0]];
   If[edges == {}, Cases[gens, Except[0]],
    Reduc[gens, Delete[edges, 1], edges[[1]]]]);
\end{verbatim}
\verb+Checks+ looks for the remaining edge whose head is in the highest homological degree.  It then applies \verb+Reduc+ to remove the edge and apply the symmetric difference algorithm in \verb+AddDelEdges+.\\
\begin{verbatim}
HomReps[L_PD] :=
  Block[{$IterationLimit = Infinity, $RecursionLimit = Infinity},
   Checks[Cases[Gens[L], Except[0]], Cases[Edges[L], Except[0]]]];
\end{verbatim}
\verb+HomReps+ takes a planar diagram code, repeatedly applies \verb+Checks+  to run the graph reduction algorithm and outputs representatives for the homology classes.\\
\begin{verbatim}
QT[gen___, L_PD] :=
  (r = (Height[gen] - nm[L]); (t^r)
      *(q^(r + Deg[gen] + np[L] - nm[L])));
\end{verbatim}
\verb+QT+ computes the associated powers of q and t for a given representative of a homology class.\\
\begin{verbatim}
AKh[L_PD] :=
  Plus @@ (QT[#, L]*VG[# /. {v[a___] :> 1}]*
       MultGrad[# /. {v[a___] :> 1}] & /@ HomReps[L]);
ParityAKh[L_PD] :=
  If[Length[ForgetfulEvenParityPD[L]] == 1 ||
    Head[ForgetfulEvenParityPD[L]] === PDReduction, q + q^(-1),
   Plus @@ (QT[#, ForgetfulEvenParityPD[L]]*VG[# /. {v[a___] :> 1}]*
        MultGrad[# /. {v[a___] :> 1}] & /@
      HomReps[ForgetfulEvenParityPD[L]])];
\end{verbatim}

Finally, \verb+AKh+ and \verb+ParityAKh+ compute the corresponding categorifications for a given planar diagram L. \\
}


\begin{thebibliography}{10}

\baselineskip=12pt 
\parskip=2pt plus 1pt 

\bibitem{APS}
M.M. Asaeda, J.H. Przytycki, A.S. Sikora. Categorification of the Kauffman bracket skein module of I-bundles over surfaces. (English summary) \emph{Algebr. Geom. Topol.} 4 (2004), 1177–1210 (electronic).

\bibitem{BaldwinGillam}
J.A. Baldwin, W.D. Gillam. Computations of Heegaard-Floer knot homology. 2007. arXiv:math.GT/0610167.

\bibitem{KnotAtlas}
D. Bar-Natan. (2008, 17 December) \emph{The Jones Polynomial - The Knot Atlas}
URL \url{http://katlas.org/wiki/The_Jones_Polynomial}

\bibitem{DrorCat}
D. Bar-Natan. On Khovanov's categorification of the Jones polynomial. \emph{Algebr. Geom. Topol.} 2 (2002), 337–370.

\bibitem{DrorCob}
D. Bar-Natan. Khovanov's homology for tangles and cobordisms. \emph{ Geom. Topol.} 9 (2005), 1443–1499.

\bibitem{DrorFastKhoHo}
D. Bar-Natan. Fast Khovanov homology computations.\emph{J. Knot Theory Ramifications} 16 (2007), no. 3, 243–255.

\bibitem{CMW}
D. Clark, S. Morrison, K. Walker.Fixing the functoriality of Khovanov homology. (English summary)
\emph{Geom. Topol.} 13 (2009), no. 3, 1499–1582.

\bibitem{DKMinimalSurface}
H. A. Dye, L.H. Kauffman. Minimal surface representations of virtual knots and links. \emph{Algebr. Geom. Topol.} 5 (2005), 509–535 (electronic).

\bibitem{DKVirtualCrossingNumber}
H. A. Dye, L.H. Kauffman. Virtual crossing number and the arrow polynomial. \emph{J. Knot Theory Ramifications} 18 (2009), no. 10, 1335–1357.

\bibitem{DKM}
H. Dye, L. Kauffman, V.O. Manturov. "On Two Categorifications of the Arrow Polynomial for Virtual Knots." In \emph{The Mathematics of Knots} edited by M. Banagl and D. Vogel. Contributions in Mathematical and Computational Sciences, 2011, Volume 1, 95-124.

\bibitem{Elliott}
A. Elliot. State cycles, quasipositive modification, and constructing H-thick knots in Khovanov homology.
Thesis (Ph.D.)–Rice University. 2010.

\bibitem{FRR}
R. Fenn, R.Rimanyi, C. Rourke.
The braid-permtuation group. { \it Topology} 36 (1997), 133–135.


\bibitem{GreenTables}
J. Green. (2004, August 10) \emph{A Table of Virtual Knots}.
URL\url{http://www.math.toronto.edu/~drorbn/Students/GreenJ/}.

\bibitem{GPV}
M.N.Goussarov, M.Polyak, O. Viro.
Finite Type Invariants of Classical and Virtual Knots.{ \it  Topology } 39 No. 5, 1045-1068 (2000).

\bibitem{Hatcher}
A. Hatcher. Pants Decompositions of Surfaces. arXiv:math/9906084.

\bibitem{LinKnot}
S. Jablan and R. Sazdanovic. \emph{LinKnot}.
URL\url{http://math.ict.edu.rs:8080/webMathematica/LinkSL/}.

\bibitem{ParityBiquandles}
A. Kaestner, L. Kauffman. Parity Biquandles. 2011. arXiv:1103.2825.

\bibitem{ExtendedBracket}
L. H. Kauffman. An extended bracket polynomial for virtual knots and links.
\emph{J. Knot Theory Ramifications} 18 (2009), no. 10, 1369–1422.

\bibitem{IntroVKT}
L. Kauffman. Introduction to Virtual Knot Theory. arXiv:1101.0665v1.

\bibitem{KauffmanKhoHo}
L. Kauffman. Khovanov Homology. In \emph{Introductory Lectures in Knot Theory}, K\&E Series Vol. 46, edited by Kauffman, Lambropoulou, Jablan and Przytycki. World Scientific 2011, 248 - 280.

\bibitem{NewInvariants}
L. Kauffman. New invariants in the theory of knots.
{ \it Amer. Math. Monthly} \textbf{95} (1988), no. 3, 195–242.

\bibitem{OnKnots}
L. Kauffman. \emph{On Knots.} Princeton University Press, Princeton, NJ, (1987).

\bibitem{OddWrithe}
L. H. Kauffman. A self-linking invariant of virtual knots. \emph{Fund. Math.} 184 (2004), 135–158.

\bibitem{VKT}
L. Kauffman. Virtual Knot Theory. { \it European J. Combin.} \textbf{20} (1999), 663-690.

\bibitem{RationalKnots}
L. Kauffman, S.  Lambropoulou. Classifying and applying rational knots and rational tangles.
(English summary) \emph{Physical knots: knotting, linking, and folding geometric objects in $\mathbb{R}^3$} (Las Vegas, NV, 2001), 223–259, Contemp. Math., 304, Amer. Math. Soc., Providence, RI, 2002.

\bibitem{Khovanov1}
M. Khovanov. A categorification of the Jones polynomial. \emph{Duke Math. J.} 101 (2000), no. 3, 359–426.

\bibitem{Khovanov2}
M. Khovanov. A functor-valued invariant of tangles. \emph{Algebr. Geom. Topol.} 2 (2002), 665–741.

\bibitem{UnknotDetector}
P. B. Kronheimer, T. S. Mrowka. Khovanov homology is an unknot-detector. arXiv:1005.4346.

\bibitem{Endomorphism}
E. S. Lee. An endomorphism of the Khovanov invariant.
(English summary) \emph{Adv. Math.} 197 (2005).

\bibitem{AddnlGradings}
V. O. Manturov.  Additional Gradings in Khovanov homology. arXiv:0710.3741.

\bibitem{ParityBracket}
V. O. Manturov. Parity in knot theory. (Russian) \emph{Mat. Sb.} 201 (2010), no. 5, 65--110; translation in \emph{Sb. Math.} 201 (2010), no. 5-6, 693–733.

\bibitem{ArbitraryCoeffs}
V. O. Manturov. Khovanov homology for virtual knots with arbitrary coefficients. \emph{J. Knot Theory Ramifications} 16 (2007), no. 3, 345–377.

\bibitem{ManturovKnotTheory}
V.O. Manturov \emph{Knot theory}. Chapman $\&$ Hall/CRC, Boca Raton, FL, 2004.

\bibitem{ManturovWidth}
V. O. Manturov. Minimal diagrams of classical and virtual links. (2005), arXiv:math/0501393.

\bibitem{Miy06}
Y. Miyazawa. Magnetic Graphs and an Invariant for Virtual Links. \emph{J. Knot Theory Ramifications}, 15(10):1319–1334, 2006.

\bibitem{Miy08}
Y. Miyazawa. A Multi-Variable Polynomial Invariant for Virtual Knots and links. \emph{J. Knot Theory Ramifications}, 17(11):1311–1326, 2008.

\bibitem{Rasmussen}
J. Rasmussen. Khovanov homology and the slice genus. \emph{Invent. Math.} 182 (2010), no. 2, 419–447.

\bibitem{Shumakovitch}
A. Shumakovitch. Torsion of the Khovanov homology. (2004), arXiv:math.GT/0405474.

\bibitem{ViroKhoHo}
O. Viro. Virtual Links, Orientations of Chord Diagrams and Khovanov Homology. \emph{Proceedings of 12th Gokova Geometry-Topology Conference} International Press, 2006, pp. 184 - 209.

\bibitem{WehrliSpanningTrees}
S. Wehrli. A spanning tree model for khovanov homology.\emph{J. Knot Theory Ramifications}, 17(12):1561–1574, 2008.


\end{thebibliography}
\end{document}